\documentclass[12pt, a4]{book}  
\usepackage{times,amssymb,amsmath,amsfonts}       
\usepackage{makeidx} 

\usepackage{ifthen} 
\usepackage{mhequ} 
\usepackage{theorem}
\usepackage{enumerate}

{  \theorembodyfont{\rmfamily}
\newtheorem{theorem}{Theorem}[section] 
\newtheorem{proposition}[theorem]{Proposition}
\newtheorem{lemma}[theorem]{Lemma}
\newtheorem{corollary}[theorem]{Corollary}

\newtheorem{definition}[theorem]{Definition}
\newtheorem{example}[theorem]{Example}

\newtheorem{remark}[theorem]{Remark}

 }

\makeatletter
\def\sqr#1#2{%
	\fboxsep0pt%
	\fboxrule#2pt%
	\fbox{\hbox to #1pt{\hss\vbox to #1pt{\vss}}}}
\def\qed{\unskip\nobreak\hfil\penalty50\hskip2em\hbox{}\nobreak
      \hfil\sqr{5.5}{0.2}\parfillskip=\z@ \finalhyphendemerits=0 \par}
\newenvironment{proof}[1][ ]{\ifthenelse{\equal{#1}{ }}{\def\MH@pfName{.}}{\def\MH@pfName{ #1.}}%
	\trivlist\item[\hskip\labelsep{\it Proof\MH@pfName}]\ignorespaces}{\phantom{a}\nobreak\qed\endtrivlist}
\makeatother

\newcommand{\A}{{\bf \mathcal A}}
\newcommand{\Ab}{{\mathbb A}}

\newcommand{\B}{{ \bf \mathcal B }}
\def\Borel{\mathop{\rm Borel}}
\newcommand{\Bb}{{ \mathbb B }}
\newcommand{\C}{{\mathcal C}}
\newcommand{\D}{{\mathcal D}}
\newcommand{\E}{{ \mathbf E }}
\def\Epsilon{{\mathcal E}}
\newcommand{\F}{{ \mathcal F }}
\newcommand{\G}{{ \mathcal G }}
\newcommand{\HH}{{ \mathcal H}}

\def\g{\mathfrak g}
\def\h{{\mathfrak h}}
\def\k{{\mathfrak k}}
\def\m{{\mathfrak m}}
\def\PP{{\mathbf P}}
\def\Q{\mathbf Q}
\newcommand{\HHH}{{\mathbf H}}
\newcommand{\R}{{\mathbf R}}
\newcommand{\CR}{{\mathcal R}}
\newcommand{\LL}{{\mathbf L}}
\newcommand{\Lo}{{\mathcal L}}
\newcommand{\CL}{{\mathbb L}}
\newcommand{\X}{{\mathbb X}}
\newcommand{\Y}{{\mathbb Y}}
\newcommand{\Z}{{\mathbb Z}}

\newcommand{\1}{{\bf 1}}
\def\loc{\mathop{\rm{loc}}}
\def\Ric{\mathop{\rm Ric}}
\def\comp{\mathop{\rm Comp}}
\def\Diff{\mathop{\rm Diff}}
\def\ad{\mathop{\mathrm{ad}}}
\def\det{\mathop{\mathrm{det}}}
\def\cf{{\rm {c.f. }}}
\def\eg{\textit{e.g. }}

\def\ie{\textit{i.e. }}
\def\Diff{\mathop{\mathrm {Diff}}}

\def\ker{\mathop{\mathrm {ker}}}
\def\Image{\mathop{\mathrm {Image}}}
\def\dim{\mathop{\mathrm {dim}}}
\def\Id{{\mathop{\mathrm {Id}}}}
\def\id{{\mathop{\mathrm {id}}}}

\def\GLM{\mathop {\mathrm {GLM}}}
\def\span{\mathop {\mathrm {span}}}
\def\trace{\mathop{\mathrm {trace}}}

\def\rank{\mathop{\mathrm {rank}}}

\def\dom{\mathop{\mathrm{Dom}}}

\def\paral{/\kern-0.3em/}
\def\parals_#1{/\kern-0.3em/_{\!#1}}

\begin{document}

\title{The Geometry of   Filtering\\(Preliminary Version)}
\author{K. D. Elworthy, Yves Le Jan,  and Xue-Mei Li}

\maketitle

\chapter*{Introduction}

\footnote{ K. D. Elworthy,  Mathematics Institute, University of Warwick, Coventry CV4 7AL, U.K. E-mail: K.D.Elworthy@warwick.ac.uk}
\footnote{Xue-Mei Li, Mathematics Institute, University of Warwick, Coventry CV4 7AL, U.K. E-mail: Xue-Mei.Hairer@warwick.ac.uk}
 \footnote{ Y. LeJan,  D\'epartment de Mathematques, Universit\'e Paris Sud 11, 91405 Orsay, France.  E-mail:  Yves.LeJan@math.u-psud.fr}

Filtering is the science of  finding the law of a process given a partial observation of it. The main objects we study here are diffusion processes. These are naturally associated with second order linear differential operators which are semi-elliptic and so introduce a possibly degenerate Riemannian structure on the state space.  In fact much of what we discuss is simply about two such operators intertwined by a smooth map, the ``projection from the state space to the observations space", and does not involve any stochastic analysis.

From the point of view of stochastic processes our purpose is to present and to study the underlying geometric structure which allows
us to perform the filtering  in a Markovian framework with the resulting conditional law being that of a Markov process. This geometry is determined by the symbol of the operator on the state space which projects to a symbol on the observation space. The projectible symbol induces a (possibly non-linear and partially defined) connection  which lifts the observation process to the state space and gives a  decomposition of the operator on the state space and of the noise.  As is standard we can recover the classical filtering theory in which the observations are not usually Markovian by application of the Girsanov-Maruyama-Cameron -Martin Theorem.

This structure we have is examined in relation to a number of geometrical topics. In one direction this leads to a generalisation of Hermann's theorem on the fibre bundle structure of certain Riemannian submersions. In another it gives a novel description of generalised Weitzenb\"ock curvature. It also applies to infinite dimensional state spaces such as arise naturally for stochastic flows of diffeomorphisms defined by stochastic differential equations, and for certain stochastic partial differential equations. 



\medskip

Let $M$ be a smooth manifold. Consider  a smooth second order semi-elliptic
differential operator $\Lo   $ such that $\Lo   1\equiv 0$. In a local chart, such an operator takes the following form 
\begin{equation} \label{loc-1}
\Lo={1\over 2}\sum_{i,j=1}^n a^{ij}{\partial \over \partial x^i}{\partial \over \partial x^j}+\sum b^i{\partial \over \partial x^i}
\end{equation}
where the $a^{ij}$'s and $b^i$'s are smooth functions and the matrix $(a^{ij})$ is positive semi-definite.

 Such differential operators are called diffusion operators. An elliptic
diffusion operator induces a Riemannian metric on $M$. In the degenerate case we shall have to assume that the ``symbol" of $\Lo$ (essentially the matrix $[a^{ij}]$ in the representation (\ref{loc-1})) has constant rank and so determines a sub-bundle $E$ of the tangent bundle $TM$ together with a Riemannian metric on $E$.  In
Elworthy-LeJan-Li \cite{Elworthy-LeJan-Li-Tani} and
\cite{Elworthy-LeJan-Li-book} it was
 shown that a diffusion operator in H\"ormander form\index{H\"ormander!form}, satisfying this condition, induces a linear
 connection on $E$ which is adapted to the Riemannian metric induced on $E$, but not necessarily torsion free. It was also shown that all metric
 connections on $E$ can be constructed by some choice of H\"ormander form for a given $\Lo$ in this way. The use of such
 connections has turned out to be instrumental in the decomposition of noise
and calculation of covariant derivatives of the derivative flows.

 A related construction of connections can extend to principal fibre bundles $P$,
indeed to more general situations, such as foliated manifolds and
stratified manifolds. An equivariant differential operator
 on $P$ induces naturally a diffusion operator on the base manifold.
 Conversely given  a connection on $P$  one can lift horizontally a
 diffusion operator on the base manifold of the form of  sum of
 squares of vector fields by simply lifting up the vector fields. It
still need to be shown that the lift is independent of  choices
of its H\"ormander form.  Consider now a diffusion  operator  not given in H\"ormander
 form. Since it has no zero order term we can associate with it an operator $\delta$
which send differential one forms to functions. In  Proposition
\ref{pr:delta} a class of such  operators are described, each
of which determines a diffusion operator. Horizontal lifts of diffusion
 operators  can then be defined in terms of the $\delta$ operator. This construction extends to situations where there is no equivariance and we have only partially defined 
 and non-linear connections.

The connections discussed here arise in much more general situations, including for foliations though these are not discussed in this volume,  We show that given a smooth $p:N\to M$: a diffusion operator
 $\B $  on $N$ which  lies over a diffusion operator $\A$ on $M$ satisfying a "cohesiveness" property  gives rise to a \emph{semi-connection},  a partially defined, non-linear, connection which can be
characterised by the property that, with respect to it,  $\B  $
can be written as the direct sum of the horizontal lift of its induced
 operator and a vertical diffusion operator. Of particular importance are examples where $p:N\to M$ is a principal bundle. In that case the vertical component of $\B$ induces differential operators on spaces of sections of associated vector bundles: we observe that these are zero-order operators, and can have geometric significance. 
 
This geometric significance  and the relationship between these partially defined connections and the metric connections determined by the H\"ormander form as in \cite{Elworthy-LeJan-Li-Tani} and
\cite{Elworthy-LeJan-Li-book} is seen when taking $\B$ to be the generator of the diffusion given on the frame bundle $\GLM$ of $ M$  by the  action of the derivative flow of a stochastic differential equation on $M$.  The semi- connection determined by $\B$ is then equivariant and is the \emph{adjoint} of the metric connection induced by the SDE in a sense extending that of Driver \cite{Driver92} and described in \cite{Elworthy-LeJan-Li-book}. The zero-order operators induced on differential forms as mentioned above turn out to be generalised Weitzenb\"ock curvature operators,in the sense of \cite{Elworthy-LeJan-Li-book}, reducing to the classical ones when $M$ is Riemannian  for particular choices of stochastic differential equations for Brownian motion on $M$. Our filtering then
reproduces the conditioning results for derivatives of stochastic flows in \cite{Elworthy-Yor}and \cite{Elworthy-LeJan-Li-book}. 

Our approach is also applied to the case where $M$ is compact and  $N$ is its diffeomorphism group, $\Diff(M)$ , with $P$ evaluation at a chosen point of $M$. The operator $\B$
is taken to be the generator of the diffusion process on $\Diff(M)$ arising from a stochastic flow.  However our constructions can be made in terms of the reproducing Hilbert space of vector fields on $M$ defined by the flow. From this we see that stochastic flows are essentially determined by a class of semi-connections on the bundle
$p:\Diff(M)\to M$ and smooth stochastic flows whose one point motions have a cohesive generator determine semi- connections on all natural bundles over $M$.  Apart from these geometrical aspects of stochastic flows we also obtain a skew product decomposition which, for example, can be used to find conditional expectations of functionals of such flows given knowledge of the one point motion from our chosen point in $M$. 

A feature of our approach is that  in general we use canonical processes as solutions of martingale problems to describe our processes, rather than stochastic differential equations and semi-martingale calculus, unless we are explicitly dealing with the latter. This leads to some some new constructions, for example of integrals along the paths of our diffusions in Section \ref{se:int-pred-proc}, which are valid more generally than in the very regular cases we discuss here.

In more detail:
In Chapter One we describe various representations of diffusion operators and when they are available. We also define the notion of such an operator being 
\emph{along a distribution}. In Chapter Two we introduce the notion of \emph{semi-connection} which is fundamental for what follows, show how these are induced by certain intertwined pairs of diffusion operators and how they relate to a canonical decomposition of such operators. We also have a first look at the topological consequences on $p:N\to M$ of having $\B$ on $N$ over some $\A$ on $M$ which posses  hypo-ellipticity type properties. This is a minor extension of part of Hermann's theorem, \cite{Hermann},  for Riemannian submersions. In Chapter Three we specialise to the case of principal bundles, introduce the example of derivative flow, and show how the generalised Wietzenbock curvatures arise.

It is not really until Chapter Four that stochastic analysis  plays a major role. Here we describe methods of conditioning functionals of the $\B$-process given information about its projection onto $M$. We also use our decomposition of $\B$ and resulting decomposition  of the $\B$-process  to describe the  conditional $\B$-process. In the equivariant case of principal bundles the decomposition of the process can be considered as a skew product decomposition. In Chapter \ref{ch:nonlinear-filtering} we show how our constructions can apply to classical filtering problems, where the projection of the
$\B$-process is non-Markovian. We can follow the classical approach and obtain, in Theorem \ref{th:kushner},  a version of Kushner's formula for non-linear filtering in somewhat greater generality than is standard. This requires some discussion of analogues of \emph {innovations} processes in our setting.

We return to more geometrical analysis in Chapter Six, giving further extensions of 
Hermann's theorem and analysing the consequences of the horizontal lift of $A$ commuting with $B$, thereby
 extending the discussion in \cite{Berard-Bergery-Bourguignon}. In particular we see that such commutativity, plus hypo-ellipticity conditions on $\A$, gives a bundle structure and a diffusion operator on the fibre which is preserved by the trivialisations of the bundle structure. This leads to an extension of the "skew-product" decomposition given in \cite{Elworthy-Kendall} for Brownian motions on the total space of Riemannian submersions with totally geodesic fibres. In fact the well known theory for Riemann
 submersions, and the special case arising from Riemannian symmetric spaces  is presented in Chapter Seven.
 
 Chapter Eight is where we describe the theory for the diffeomorphism bundle $p:\Diff(M)\to M$ with a stochastic flow  of diffeomorphism on $M$. Initially this is done independently of stochastic analysis and in terms of reproducing kernel Hilbert spaces of vector fields on $M$. The correspondence between such Hilbert spaces and stochastic flows is then used to get results for flows and in particular skew-product decompositions
 of them.

 In the Appendices we present the Girsanov Theorem in a way which does not rely on 
 having to use  conditions such as Novikov's criteria for it to remain valid. This has been known for a long time, but does not appear to be as well known as it deserves. We also look at conditions for degenerate, but  smooth, diffusion operators to have smooth H\"ormander forms, and so to have stochastic differential equation representations for their associated processes. Finally we discuss semi-martingales and $\Gamma$-martingales along a sub-bundle of the tangent bundle with a connection.
 
 For Brownian motions on the total spaces of Riemannian submersions much of our basic discussion, as in the first two and a half  Chapters,  of skew-product decompositions is very close to that in \cite{Elworthy-Kendall} which was taken further by Liao in \cite{Liao89}.  A major difference from Liao's work is that for degenerate diffusions we use the semi-connection determined by our operators rather than an arbitrary one, so obtaining canonical decompositions. The same holds for the very recent
	 work of Lazaro-Cami \& Ortega, \cite{ LazaroCami-Ortega-reduction} where they are motivated by the reduction and reconstruction of Hamiltonian systems and consider similar decompositions for semi-martingales. An extension  of \cite{Elworthy-Kendall}  in a different direction, to shed light on the Fadeev-Popov procedure for gauge theories in theoretical physics was given by Arnaudon \&Paycha in \cite{Arnaudon-Paycha}.  Much of the equivariant theory presented here was announced with some sketched proofs in \cite {Elworthy-LeJan-Li-principal}.

\subsection*{Key Words}
semi-elliptic,  second order differential operator, H\"ormander forms, connection,  semi-connection,
diffusion processes, Girsanov theorem,  intertwined diffusions, conditioned laws, filtering, Weitzenb\"ock curvature, skew-product decomposition,
stochastic flows, manifolds, Riemannian submersions, bundles, principal bundles, Diffeomorphism bundles.

 \tableofcontents

\chapter{Diffusion Operators}
If $\Lo$ is a second order differential operator on a manifold $M$,
 denote by ${\sigma^\Lo: T^*M\to TM} $ its symbol \index{operator!symbol} \index{symbol} determined by
$$df\left( \sigma^\Lo(dg)\right) =\frac 12\Lo\left(
fg\right) -\frac 12(\Lo f) g-\frac 12f (\Lo g),$$
for $C^2$ functions  $f, g$. We will often write $\sigma^\Lo(\ell^1, \ell^2)$ for $\ell^1\sigma^\Lo(\ell^2)$ and consider $\sigma^\Lo$ as a bilinear form on $T^*M$. Note that it is symmetric. The operator is said to be { semi-elliptic}\index{semi-elliptic} \index{operator!semi-elliptic} if
 $ \sigma ^\Lo(\ell^1, \ell^2) \geqslant  0$ for all $\ell^1, \ell^2 \in T_u M^*$, all $u\in M$,  and elliptic if the inequality holds strictly. Ellipticity is equivalent to
 $\sigma^\Lo$ being  onto. 
 \begin{definition}
A semi-elliptic  smooth second order differential operator $\Lo$ is said to be a {\it diffusion operator}\index{diffusion!operator}\index{operator!diffusion} if $\Lo 1=0$.
\end{definition}
\section{Representations of Diffusion Operators}
Apart from local representations as given by equation \ref{loc-1} there are several global ways to represent a diffusion operator $\Lo$. One is
to take a connection $\nabla$ on $TM$. Recall that a \emph{connection} on $TM$ gives, or is given by, a \emph{ covariant derivative operator} $\nabla$ acting on vector fields. For each $C^r$ vector field $U$ on $M$ it gives a $C^{r-1}$ section $\nabla_-U$ of $\CL(TM;TM)$. In other words for each $x\in M$ we have a linear map $v\mapsto \nabla_v U$ of $T_xM$ to itself. This covariant derivative of $U$ in the direction $v$ satisfies the usual rules. In particular it is a derivation with respect to multiplication by differentiable functions $f:M\rightarrow \R$, so that $\nabla_vfU=df(v)U(x)+f(x)\nabla_vU$.
Given any smooth vector bundle $\tau E\rightarrow M$ over $M$ a \emph {connection on }$E$ gives a similar covariant derivative acting on sections $U$ of $E$. This time $v\rightarrow \nabla_vU$ is in $\CL(T_xM;E_x)$, where $E_x$ is the fibre over $x$ for $x\in M$.  Such connections always exist. 

Then we can write
\begin{equation}
\label{representation-1}
\Lo f(x)={\trace}_{T_xM}\nabla_-(\sigma^\Lo(df))+df(V^0(x))
\end{equation}
for some smooth vector field $V^0$ on $M$. The trace is that of the mapping $v\mapsto \nabla_v(\sigma^\Lo(df))$
from $T_xM$ to itself. To see this it is only necessary to check that
the right hand side has the correct symbol since the symbol determines the diffusion operator up to a first order term.

If a smooth `square root' to $2\sigma^\Lo$ can be found we have a H\"{o}rmander representation. The `square root' is a smooth $X: M\times \R^m\to TM$ with each $X(x)\equiv X(x,-): \R^m\to T_xM$ linear, such that
$$2\sigma_x^\Lo=X(x)X(x)^*: T_x^*M\to T_xM.$$
Thus there is a smooth vector field $A$ with 
\begin{equation}
\label{representation-2}
\Lo={1\over 2}\sum_{j=1}^m \LL_{X^j}\LL_{X^j}+\LL_A,
\end{equation}
where $\LL_V$ denotes Lie differentiation with respect to a vector field $V$, so $\LL_V f(x)=df_x(V(x))$, and $X^j(x)=X(x)(e_j)$ for $\{e_j\}$ an orthonormal basis of $\R^m$.
{\bf If $\sigma^\Lo$ has constant rank such $X$ may be found.}  Otherwise it is only known that locally Lipschitz square roots exist (see the discussions in Appendix A).
 In that case $\LL_{X^j}\LL_{X^j}$ is only defined almost surely everywhere and the vector field $A$ can only be assumed measurable and locally bounded.
 Nevertheless uniqueness of the martingale problem still holds (see below). Also there is still the hybrid representation, given a connection $\nabla$ on $TM$:
\begin{equation} \label{representation-3}
\Lo f(x)={1\over 2}\sum_{j=1}^m \nabla _{X^j(x)}(df)(X^j(x))+df(V^0(x)).
\end{equation}
for $V^0$ locally Lipschitz.

The choice of a H\"ormander representation\index{H\"ormander!representation} for a diffusion operator, if it exists, determines a locally defined  stochastic flow of diffeomorphisms $\{\xi_t:0\leqslant t<\zeta\}$ whose one point motion solves the martingale problem for the diffusion operator. In particular on bounded measurable compactly supported $f:M\to \R$
the associated (sub)Markovian semigroup is given by $ P_tf=\E(f\circ \xi_t)$. See also Appendix II.

Despite the discussion above we can always write $\Lo$ in the following form:
\begin{equation}
\Lo=\sum_{ij=1}^N a_{ij}(\cdot)\LL_{X^i}\LL_{X^j}+\LL_{X^0},
\end{equation}
where $N$ is a finite number, $a^{ij}$ and $X^k$  are respectively smooth functions and smooth vector fields  with $a_{ij}=a_{ji}$.

\section{The Associated First Order Operator}

 Denote by $C^r \Lambda^p\equiv C^r \Lambda^p T^*M$, $r\geqslant  0$,  the  space of $C^r$ smooth  differential p-forms on a manifold $N$. To each diffusion operator $\Lo$ we shall
associate an operator $\delta^\Lo$, see Elworthy-LeJan-Li \cite{Elworthy-LeJan-Li-Tani}, \cite{Elworthy-LeJan-Li-book} \cf Eberle \cite{Eberle-book}. The horizontal lift of 
$\Lo$ will then be defined in terms of a lift of $\delta^\Lo$.
\begin{proposition}
\label{pr:delta}
For each diffusion operator $\Lo$ there is a unique  smooth linear differential operator
$\delta ^\Lo: C^{r+1} \Lambda^1 \to C^r \Lambda^0$
such that
\begin{enumerate}
\item[(1)]  $\delta ^\Lo\left( f\phi \right) =df\sigma ^\Lo( \phi ) 
+f \cdot \delta ^\Lo\left( \phi \right) $
\item[(2)]  $\delta ^\Lo\left( df\right) =\Lo f .$
\end{enumerate}
Equivalently $\delta^\Lo$ is determined by either one of the following:
\begin{eqnarray}\label{delta-1}
\delta^\Lo(fdg)&=&\sigma^\Lo(df, dg)+f\Lo g\\
\delta^\Lo(fdg)&=&{1\over 2}\Lo(fg)-{1\over 2} g \Lo f
+{1\over 2} f\Lo g.\label{delta-2}
\end{eqnarray}

\end{proposition}
\begin{proof}
Take a connection $\nabla$ on $TM$ then, as in (\ref{representation-1}), $\Lo$ can be written as
$\Lo f=\trace \nabla \sigma^\Lo (df)+\LL_{V^0}f$
for some smooth vector field $V^0$. Set
$$\delta^\Lo \phi=\trace \nabla (\sigma^\Lo\phi)+\phi(V^0).$$
Then $\delta^\Lo (df)=\Lo f$ and $$\delta^\Lo (f\phi)
=\trace \nabla (f (\sigma^\Lo\phi))+f\phi(V^0)
=f\delta^\Lo \phi+df(\sigma^\Lo \phi).$$
Note that a general $C^r$ 1-form $\phi$ can be written as 
$\phi=\sum_{j=1}^k f_i dg_i$ for some $C^r$ function $f_i$ and smooth $g_i$, for example, by taking $(g^1,\dots, g^m): M\to \R^m$ to be an immersion. This shows that (1) and (2)  determine $\delta^\Lo$ uniquely. Moreover since $\Lo$ is a smooth operator so is
$\delta^\Lo$.
\end{proof} 
\begin{remark}
 If the diffusion operator $\Lo$ has a  representation
\[\Lo=\sum_{j=1}^m a_{ij}\LL_{X^j}\LL_{X^j} +
\LL_{X^0} \]
for some smooth vector fields $X^i$ and smooth functions $a_{ij}$, $i,j =0,1,\dots, m$ then
$$\delta ^\Lo
  =  \sum_{j=1}^m a_{ij}\LL_{X^j}\iota_{X^j}+\iota _{X^0},$$
where $\iota_A$ denotes the interior product of the vector field $A$ with a differential form.  One can check directly that $\delta^\Lo (df)=\Lo f$ and that (1) holds. 
In particular in a local chart, for the representation given in equation  (\ref{loc-1}) we see that 
$\delta ^\Lo$ is given by
 $$\delta ^\Lo\phi= \sum_{j=1}^m a_{ij}\frac{\partial}{\partial x_i}\phi_j(x)+\sum b^i\phi_i(x)$$ where $\phi$ has the representation $$ \phi_x=\sum \phi_j(x) ~dx^i$$
\end{remark}

\section{Diffusion Operators Along a Distribution}\label{se:distribution}
Let $N$ be a smooth manifold. By a {\bf distribution}\index{distribution}
 $S$ in $N$ we mean a family $\{S_u: u\in N\}$ where $S_u$ is a linear subspace of $T_uN$; for example $S$ could be a sub-bundle of $TN$. 
 Given such a distribution $S$ let $S^0=\cup_u S_u^0$ for {\bf  $S_u^0$ the  annihilator of} $S_u$ in $T_u^*N$.
\begin{definition}
\label{def:op-along}
\item  Let $S$ be a distribution in $TN$. Denote by $C^rS^0$ the set of $C^r$ 1-forms which vanish on $S$. A diffusion operator $\Lo$ on $N$ is said to be {\it \bf along $S$}\index{along $S$!diffusion operator} if $\delta^\Lo \phi=0$ for $\phi\in C^1S^0$.
\end{definition}




 Suppose $\Lo$ is along $S$ and take $\phi \in C^rS^0$.  By Proposition  \ref{pr:delta} and the symmetry of $\sigma^\Lo$, $0=(df)(\sigma^\Lo(\phi))=\phi(\sigma^\Lo (df)$ giving $\phi_x\in \Image[\sigma_x^\Lo]^0$. 
This proves  Remark~\ref{re:3.4}~(i):
\begin{remark}
\label{re:3.4}
\begin{enumerate}
\item [(i)]
 if $\delta^\Lo \phi=0$ for all $\phi\in C^1S^0$,
then $\sigma^\Lo  \phi=0$ for all such $\phi$ and $\Image[\sigma^\Lo_x]\subset \cap_{\phi \in C^1S^0}[\ker \phi_x]$  for all
$x \in N$. \item[(ii)]  If $S$ is a sub-bundle of $TN$ and $\Lo$ is along $S$ then without ambiguity we can define $\delta^\Lo \phi$ 
 for $\phi$ a $C^0$ section of $S^*$ by $\delta^\Lo \phi:=\delta^\Lo \tilde \phi$ for 
 any 1-form $\tilde \phi$ extending $\phi$. Recall that $S^*$ is canonically isomorphic to the quotient $T^*N/S^0$.
\end{enumerate}
\end{remark}
\begin{definition}
If $$S_x=\cap_{\phi \in C^1S^0}[\ker \phi_x]$$
for all $x$ we say $S$ is a {\bf regular distribution}.\index{distribution!regular} \end{definition}
Clearly sub-bundles are regular. 
\begin{proposition}
\label{pr:along}
\begin{enumerate}
\item[(1)]
 Let $S$ be a regular distribution of $N$ and $\Lo$ an operator written in H\"{o}rmander form:  \begin{equation}\label{op-B}
\Lo={1\over 2}\sum_{j=1}^m\LL_{Y^j}\LL_{Y^j} +
\LL_{Y^0}
\end{equation}
where the vector fields $Y^0$ and  $Y^j, j=1,\dots, m$ are $C^0$ and $C^1$ respectively. Then $\Lo$ is along $S$ if
 and only if $Y^i$ are sections of $S$. 
 \item[(2)]
 If $\B$ is along a smooth sub-bundle $S$ of $TN$ then for any connection $\nabla^S$ on $S$ we can write $\B$ as
 $$\B f={\trace}_{S_x}\nabla^S_- \big(\sigma^\B(df)\big)+\LL_{X^0}f.$$
Also we can find smooth sections $X^0,\dots, X^m$ of $S$ and smooth functions $a_{ij}$ such that
$$\B=\sum_{i,j}a_{ij}(\cdot)\LL_{X^i}\LL_{X^j}+\LL_{X^0}.$$ 
\end{enumerate}
 \end{proposition}
\begin{proof}
For part (1), if $Y^i$ are sections of $S$, take $\phi\in C^1S^0$ then
$$\delta^\Lo \phi={1\over 2}\sum_{j=1}^m\LL_{Y^j}\phi(Y^j) +
\phi(Y^0)=0$$ and so $\Lo$ is along $S$.

Conversely suppose $\Lo$ is along $S$. Define a $C^1$ bundle map $Y: \R^m\to TN$ by $Y(x)(e)=\sum_{j=1}^m Y^j(x)e_j$
for $\{e_j\}_{j=1}^m$ an orthonormal base of $\R^m$. Then
$$2\sigma_x^\Lo=Y(x)Y(x)^*$$
and
$$\Image[Y(x)]=\Image[\sigma_x^\Lo]\subset S,$$
by Remark~\ref{re:3.4}.
Now
$$\delta^\Lo\phi ={1\over 2}\sum\LL_{Y^j}(\phi(Y^j))+\phi(Y^0)=\phi(Y^0),$$
which can only vanish for all $\phi\in C^1S^0$ if $Y^0$ is a section of $S$. 
 Thus $Y^1, \dots, Y^m$, and $Y^0$ are all sections of $S$.

 For part (2), we use (\ref{representation-1}) and take $\nabla$ there to be the direct sum of $\nabla^S$ with an arbitrary connection on a complementary bundle, obtaining $\sigma^\B$ has image in $S$ by Remark \ref{re:3.4}(i).
 \end{proof}

\section{Lifts of Diffusion Operators}

Let $p:N\to M$ be a smooth map and $E$ a sub-bundle of $TM$. Let $S$ be a sub-bundle of $TN$ transversal to the fibre of $p$, \ie $VT_uN\cap S=\{0\}$ all $u\in N$
 and such that $T_yp$ maps $S_y$ isomorphically onto $E_{p(y)}$, for each
 $y$.

\begin{lemma}
\label{le:lift}
Every smooth 1-form on $N$ can be written as a linear combination
of sections of the form $\psi+\lambda p^*(\phi)$ for
 $\lambda: N\to {\Bbb R}$ smooth, $\phi$ a 1-form on $M$, and $\psi$
annihilates $S$. In particular any 1-form annihilating $VTN$
is of the form $\lambda p^*(\phi)$. If $E=TM$ then $\psi$ is uniquely determined.
\end{lemma}

\begin{proof}
Take Riemannian metrics on $M$ and $N$ such that
the isomorphism between $S$ and $p^*(E)$ given by $Tp$ is isometric.
Fix $y_0\in N$. Take a neighbourhood $V$ of $p(y_0)$ in $M$ over which
$E$ is trivializable. Let $v^1, v^2, \dots, v^p$ be a trivialising
family of sections over $V$. Set $U=p^{-1}(V)$. If $\phi^j=(v^j)^*$,
the dual 1-form to $v^j$, $j=1$ to $p$, over $V$ then
$\{p^*(\phi^j)^{\#}, j=1 \hbox{ to } p\}$ gives a trivialization
of $S$ over $U$.
[Indeed
 $p^*(\phi^j)_{y}(-)=\phi^j_{p(y)}(T_yp -)=\langle(T_yp)^*(v^j), -\rangle$.]
Since any vector field over $V$ can therefore be written as one orthogonal
 to  $S$ plus a linear combination of the $p^*(\phi^j)^{\#}$, by duality the
 result  holds for forms with support in $U$. The global result follows
 using a partition of unity. 

 For the uniqueness note that if $E=TM$ then $TN=VTN+S$.
\end{proof} 

\begin{proposition}
\label{pr:lift}
Let $\A $ be a diffusion operator on $M$ along the sub-bundle $E$ of $TM$. There is a unique lift of $\A $ to a smooth diffusion generator $\A ^S$ along the transversal bundle $S$. Write
 $\bar \delta=\delta^{\A ^S}$. Then $\A ^S$ is
 determined by
\begin{enumerate}
\item[(i)]
$\bar \delta(\psi)=0$ if $\psi$ annihilates $S$.
\item[(ii)]
$\bar \delta\; (p^*\phi)=(\delta^\A \phi)\circ p$, \quad for $\phi\in \Omega^1(M)$.
\end{enumerate}
Moreover (iii) for $y\in N$ let $h_y: E_{p(y)}\to T_yN$ be the right inverse of $T_yp$ with image $S_y$. Then 
\begin{enumerate}
\item [(a)]
$\sigma_y^{\A^S} =h _y\; \sigma^\A \; h_y^*$
\item[(b)]
If $\A $ is given by
\begin{equation}
\label{op-hormander}
\A =\sum_{i,j=1}^N a_{ij}\LL_{X^i}\LL_{X^j} +
\LL_{X^0}\end{equation}
where $X^1,\dots, X^N$ and $X^0$ are sections of $E$ then
\begin{equation}
\label{op-lift-h}
\A ^S=\sum_{i,j=1}^N (a_{ij}\circ p) \;\LL_{\bar X^i}\LL_{\bar X^j}
 +\LL_{\bar X^0}
\end{equation}
for $\bar X^j(y)=h_y(X^j(p(y))$.
\end{enumerate}
\end{proposition}
\begin{proof}
 Lemma \ref{le:lift}  ensures that (i) and (ii) determine $\bar\delta$
uniquely as a smooth operator on smooth 1-forms if it exists. 
On the other hand we can represent $\A$ as in (\ref{op-hormander})
and define $\A^S$ be (\ref{op-lift-h}). It is straightforward to check that
then $\delta^{A^S}$ satisfies (i) and (ii).

By definition and the observation after (\ref{op-lift-h}) this must
 be  the horizontal lift, if it is a diffusion generator.  
  On the other hand  if $\A $ is given by (\ref{op-hormander}) we use it to define
 $\A^S$ by  (\ref{op-lift-h}).
It is easy to see that   $\delta^{\A^S}$
 satisfies (i) and (ii) and so
 $\delta^{\A^S}=\bar \delta$. From this
$\bar \A=\A^S$ and $\A^S $ is a smooth
 diffusion generator.
 \end{proof}

In the terminology of section~\ref{se:distribution} $S_u=\ker[T_up]$, sometimes written as $VT_uN$, is a distribution.
\begin{definition}
When an operator $\B$ is along the vertical distribution $\ker[Tp]$ we say $\B$ is {\bf vertical}\index{vertical},  and when
 there is a horizontal distribution such as $\{H_u: u\in N\}$ as given by Proposition
 \ref{pr:h-map} below and $\B$ is along that horizontal
 distribution we say $\B$ is {\bf horizontal} \index{horizontal}.
\end{definition}

\begin{proposition}
\label{pr:vert-eq}
Let $\B$  be a smooth diffusion operator on $N$ and
 $p: N\to M$ any smooth map, then the following  conditions are
 equivalent:
\begin{enumerate}
\item [(1)]
The operator $\B$ is vertical.
\item [(2)]
The operator  $\B$ has a expression of the form of
$\sum_{j=1}^m a^{ij}\LL_{Y^i}\LL_{Y^j} +
\LL_{Y^0}$
where $a^{ij}$ are smooth functions and $Y^j$ are smooth sections of the vertical tangent bundle of $TN$.
\item[(3)]  $\B(f\circ p )=0$ for all $C^2$ $f:M\to \R$.
 \end{enumerate}

\end{proposition}

\begin{proof}
(a). From (1) to (3) is trivial. From (3) to (1) note that
every $\phi$ which vanishes on vertical vectors is a linear combination of elements of the form $fp^*(dg)$ for some smooth $g:M\to \R$ by Lemma \ref{le:lift}. To show that $\B$ is vertical we only need to show that $\delta^\B(fp^*(dg))=0$. But $\B(g\circ p)=0$ implies 
$\delta^\B(p^*(dg))=0$ and also 
$p^*(dg)\sigma^\B(p^*(dg))={1\over 2}\B(g\circ p)^2-(g\circ p)\B(g\circ p)=0$. By semi-ellipticity of $\B$, $\sigma^\B(p^*(dg))=0$. 
Thus assertion (1) follows since $\delta ^\B\left( fp^*(dg) \right) =df\sigma ^\B( p^*(dg) ) +f \cdot \delta ^\B\left( p^*(dg) \right) $
from Proposition \ref{pr:delta}(1), and so (1) and (3) are equivalent.

 Equivalence of (1) and (2) follows from Proposition \ref{pr:along}.
\end{proof}

\begin{remark}
\label{re:vertical}
\begin{enumerate}
\item [(1)]If $\B$ is vertical, then by Proposition \ref{pr:delta}, for all  $C^2$ functions $f_1$ on $N$ and $f_2$ on $M$,
$\displaystyle{\B\left(f_1 (f_2\circ p )\right)
 =(f_2\circ p ) \B f_1}$;
\item[(2)] If $\B$ and $\B^\prime$ are both over a diffusion operator  $\A$ of constant rank nonzero rank  such that  $\A$ is along the image of $\sigma^\A$, then $\B-\B'$ is not in general vertical,  although $(\B-\B')(f\circ p)=0$ for all $C^2$ function $f:M\to \R$,  since it may not be semi-elliptic. For example take $p:\R^2\to \R$ to be the projection $p(x,y)=x$ with $\A={\partial^2 \over \partial x^2}$, $\B={\partial^2 \over \partial x^2}+{\partial^2 \over \partial y^2}$. Let 
$\B'={\partial^2 \over \partial x^2}+{\partial^2 \over \partial y^2}+{\partial^2 \over \partial x \partial y}$. Then $\B$ is also over $\A$ but
$\B-\B'=-{\partial^2 \over \partial x \partial y}$ is not vertical.
\end{enumerate}
\end{remark}

\chapter{Decomposition of  Diffusion Operators }
\label{ch:deco-1}
Consider a smooth map $p:N\to M$ between smooth manifolds $M$ and $N$.
By a {\bf lift}  of a diffusion operator $\A $ on $M$ over
  $p$ we mean a diffusion operator $\B  $ on $N$ such that
\begin{equation}\label{op-lift}
\B  (f\circ p)=(\A f)\circ p
\end{equation}
for all $C^2$ functions $f$ on $M$. 
In this situation we adopt the following terminology:
\begin{definition}
 If (\ref{op-lift}) holds we  say that {\bf $\B$ is over $\A$}, or that $\A$ and $\B$ are {\bf intertwined} by $p$.  A diffusion operator $\B$ on $N$ is said to be {\bf projectible} (over $p$), or {\it $p$-projectible},  if it is over some diffusion operator $\A$. \end{definition}

 Recall that the pull back $p^*\phi$ of a 1-form $\phi$ is defined by
$$p^*(\phi)_u=\phi_{p(u)} (Tp(-))=(Tp)^*\phi_{p(u)}.$$  For our map $p: N\to M$, a diffusion operator $\B$ is over $\A$ if and only if
\begin{equation}
\delta^\B\left(p^*\phi)\right)=(\delta^\A\phi)(p),
\end{equation}
for all $\phi \in C^1\wedge^1 T^*M$.



\section{The Horizontal Lift Map}
\label{se:deco-1-1}

\begin{lemma}
\label{le:commutative}
Suppose that $\B$ is over $\A$. Let $\sigma ^\B$ and $\sigma ^\A$
be respectively the symbols for $\B$ and $\A$. Then  
\begin{equation}
\label{commuta}
(T_up)  \sigma_u^\B  (T_up)^*=\sigma_{p(u)}^\A, \qquad
\forall  u\in N,
\end{equation}
\ie the following diagram is commutative :

\begin{picture}(200, 80)(-40,-57)
\put(30,0) {$T_u ^*N$} \put(85,8){$\sigma_u^\B  $}
\put(55,2){\vector(1,0){70}} \put(128,0){$T_uN$}
\put(45,-45){\vector(0,1){40}} \put(10,-28){$(T_up)^*$}
\put(28,-55){$T_{p(u)} ^*M$} \put(135,-4){\vector(0,-1){40}}
\put(128,-55){$T_{p(u)}M$.} \put(65,-52){\vector(1,0){60}} \put(86,
-65){$\sigma_{p(u)}^\A $} \put(138,-28){$T_up$}
\end{picture}
\end{lemma}

\begin{proof} 
Let $f$ and $g$ be two smooth functions on $M$. Then for $u\in N$,
$x=p(u)$,
\begin{eqnarray*}
\left( df_{x}\right) \sigma _{x}^\A\left( dg_{x}\right)
 &=&{1\over 2} \A (fg)(x)-{1\over 2}(f\A g)(x)
  -{1\over 2} (g\A f)(x)  \\
&=&\frac 12\B \left((f g)\circ p\right)(u)  
 -\frac 12 f\circ p \B(g\circ p )(u)
-{1\over 2}g\circ p \B( f\circ p )(u)  \\
&=&d\left( g\circ p\right)_u \sigma _u^{\B  }\left( d\left(
f\circ p \right)_u \right) \\
&=&\left( dg\circ T_up \right) \sigma _u^{\B  }\left( df\circ
T_up \right),
\end{eqnarray*}
which gives the desired equality.
\end{proof}

For $x$ in $M$, set $E_x:= \Image[\sigma_x^\A]\subset T_xM$. If $\sigma ^\A$ has constant rank, \ie $\dim[E_x]$ is independent of $x$, then $E:=\cup_x E_x$ is a smooth sub-bundle of $TM$.

\begin{proposition}
\label{pr:h-map}
 Assume $\sigma^\A $ has constant rank and $\B$ is over $\A$. Then there is a
unique, smooth,  horizontal lift map $h_u: E_{p(u)}\to T_uN$, $u\in N$,
 characterised by
\begin{equation}
\label{h-lift-1}
h_u\circ \sigma^{\A}_{p(u)}
=\sigma^\B_u (T_up )^*.
\end{equation}
In particular
\begin{equation}
\label{h-lift-2}
h_u(v) =\sigma^\B_u\left( (T_up )^*\alpha\right)
\end{equation}
where $\alpha\in T_{p (u)}^*M$  satisfies $\sigma^\A_{p(u)}(\alpha)=v$.
\end{proposition}

\begin{proof}
Clearly (\ref{h-lift-2}) implies (\ref{h-lift-1}) by Lemma
\ref{le:commutative} and so it suffices to prove $h_u$ is well defined by (\ref{h-lift-2}). For this we only need to show
$\sigma^\B ((T_up )^*(\alpha))=0$ for every  $\alpha$ in $\ker[\sigma_{p(u)}^\A]$.
Now  $\sigma^\A \alpha=0$  implies that
$$(Tp) ^*(\alpha)\sigma^\B ((Tp )^*\alpha)=0,$$
by Lemma \ref{le:commutative}.
Considering $\sigma^\B$ as a semi-definite bilinear form this implies $\sigma_u^\B (T_up)^*\alpha$ vanishes as required.
\end{proof} 

\medskip

Note that  the vertical distribution $\ker[Tp]$ is regular as  $ \ker[Tp]$ is annihilated by all differential 1-forms of the form $\theta\circ Tp$. 

Let $H_u=\Image[h_u]$ and $H=\sqcup_u H_u$. Set $F_u=(T_up)^{-1} [E_{p(u)}]$ so we have a 
splitting 
\begin{equation}
\label{split}
F_u=H_u+VT_uN
\end{equation}
where $VT_uN=\ker[T_uP]$ the `vertical' tangent space at $u$ to $N$. 
In the elliptic case $p$ is a submersion, the vertical tangent spaces have constant rank, and $F:=\sqcup_u F_u$ is a smooth sub-bundle of $TN$. In this case we have a splitting of $TN$, a {\bf connection} in the terminology of Kolar-Michor-Slovak \cite{Kolar-Michor-Slovak}. 
In general we will define a {\bf semi-connection}\index{connection!semi-} on $E$ to be a sub-bundle $H_u$ of $TN$ such that $T_up$ maps each fibre $H_u$ isomorphically to $E_{p(u)}$. In the equivariant case considered in Chapter \ref{ch:deco-2} such objects are called $E$-connections by Gromov. For the case when $:N\to M$ is the tangent bundle projection , or the orthonormal frame bundle note that the "partial connections" as defined by Ge in \cite{Ge-92} are rather different from the semi-connections  we would have: they give parallel translations along  $E$-horizontal paths which send vectors in $E$ to vectors in $E$, and preserve the Riemannian metric of $E$ , whereas the parallel transports of  our semi-connections do not in general  preserve the fibres of $E$, nor any Riemannian metric, and they act on all tangent vectors.

\begin{lemma}
\label{le:lift-2}
 Assume $\sigma^\A $ has constant rank and $\B$ is over $\A$. For all $u\in N$ the image of $\sigma_u^\B$ is in $F_u$. 
\end{lemma}
\begin{proof}
Suppose $\alpha \in T_u^*N$ with $\sigma^\B(\alpha)\not \in F_u$. Then there exists $k$ in the annihilator of $E_{p(u)}$ such that $k\left(T_up\,\sigma^\B(\alpha)\right)\not =0$.
However $$k\left(T_up\,\sigma^\B(\alpha)\right)=\alpha\left(\sigma^\B((T_up)^*(k))\right)
=\alpha\, h_u \sigma^\A_{p(u)}(k)$$
by Proposition \ref{pr:h-map}; while $\sigma^\A_{p(u)}(k)=0$ because for all $\beta\in T_{p(u)}^*M$,
$$\beta\, \sigma^\A_{p(u)}(k)=k\, \sigma^\A_{p(u)}(\beta)=0$$
giving a contradiction.
\end{proof}

\begin{proposition}
\label{pr:h-funct-1}
Let $\A$ be a diffusion operator on $M$ with $\sigma^\A$ of constant rank. For $i\in\{1,2\}$, let $p^i: N^i\to M$ be smooth maps and $\B^i$ be diffusion operators on $N^i$ over $\A$.
Let $F: N^1\to N^2$ be a smooth map with $p^2\circ F=p^1$. Assume $F$ intertwines $\B^1$ and $\B^2$. Let $h^1$, $h^2$ be the horizontal lift maps determined by $\A, \B^1$ and  $\A,\B^2$. Then
 \begin{equation}
 \label{map-lift}
h^2_{F(u)}=T_uF(h_u^1), \qquad u\in N^1;
\end{equation} 
\ie the diagram

\begin{picture}(600,100)(0,0)
\put(85,78){$T_uN^1$}
 \put(140, 87){$T_uF$}
 \put(112, 82){\vector(1,0){76}}
 \put(194, 78){$T_{F(u)}N^2$}
\put(143, 17){$E_{p^1(u)}$}
\put (150,30){\vector(-1,1){45}}
\put(150,30){\vector(1,1){45}}
\put(108,47){$h^1_u$}
\put(180,47){$h^2_{F(u)}$}
\end{picture}
commutes for all $u\in N$.
\end{proposition}

\begin{proof}
Since $F$ intertwines $\B^1$ and $\B^2$,  Lemma \ref{le:commutative} gives
$$\sigma^{\B^2}_{F(u)}=T_uF\circ \sigma_u^{\B^1}
\circ (T_uF)^*.$$
Now take $\alpha \in T_{p^1(u)}^*M$ with $\sigma_{p^1(u)}^\A (\alpha)=v$,
some given $v\in E_{p^1(u)}$. From (\ref{h-lift-2})
\begin{eqnarray*}
h^2_{F(u)}(v)&=&\sigma_{F(u)}^{\B  ^2}((Tp^2)^*\alpha)\\
&=&T_uF\circ \sigma_u^{\B^1}\circ(T_uF)^*(Tp^2)^*\alpha\\
&=&T_uF\circ \sigma_u^{\B^1}(T_up^1)^*\alpha\\
&=&T_uh^1_u(v)
\end{eqnarray*}
as required.
\end{proof} 
\begin{definition} 
\label{basic-symbol}
A diffusion operator $\B$ on $N$ will be said to have {\bf projectible symbol}\index{symbol!projectible} for $p:N\to M$ if there exists a map $\eta:T^*M\to TM$ such that for all $u\in N$ the diagram:

\begin{picture}(200, 80)(-40,-57)
\put(30,0) {$T_u ^*N$} \put(85,8){$\sigma_u^\B  $}
\put(55,2){\vector(1,0){70}} \put(128,0){$T_uN$}
\put(45,-45){\vector(0,1){40}} \put(10,-28){$(T_up)^*$}
\put(28,-55){$T_{p(u)} ^*M$} \put(135,-4){\vector(0,-1){40}}
\put(128,-55){$T_{p(u)}M$.} \put(65,-52){\vector(1,0){60}} \put(86,
-65){$\eta_{p(u)}$} \put(138,-28){$T_up$}
\end{picture}
\bigskip

\noindent  commutes, \ie if $(T_up)\sigma_u^\B (T_up)^*$ depends only on $p(u)$.
\end{definition}

In this case we also get a uniquely defined horizontal lift map as in Proposition \ref{pr:h-funct-1} defined by equation (\ref{map-lift}) using $\eta$ instead of the symbol of $\A$. 
This situation arises naturally in the standard non-linear filtering literature as described later see  chapter \ref{ch:nonlinear-filtering}.

\section{Example: The Horizontal Lift Map of SDEs}

Let us consider the horizontal lift connection in more detail  when $\B$ and $\A$ are given by stochastic differential equations. 
For this write $\A $ and $\B  $ in H\"ormander form\index{H\"ormander!form}
 corresponding to factorisations
 $\sigma_x^{\A}=X(x)X(x)^*$
and $\sigma_x^{\B}=\tilde X(x)\tilde X(x)^*$
for
$$X(x):\R^m\to T_xM, \hspace{0.7in} x\in M$$
$$\tilde X(u):\R^{\tilde m}\to T_uN, \hspace{0.7in} u\in N.$$
Then $X(x)$ maps onto $E_x$ for each $x\in M$. Define $Y_x: E_x\to \R^m$
to be its right inverse: $Y(x)=\Big[X(x)\big |_{\ker X(x)^\perp}\Big]^{-1}$.
\begin{lemma}
\label{factorization-connection}
For each $u\in N$ there is a unique linear $\ell_u:\R^m\to \R^{\tilde m}$
such that $\ker \ell_u=\ker X(x)$ and the diagram\\
\hspace {0.2in}
\begin{picture}(500, 140)(0,-90)
\put(30,0) {$T_u ^*N$}
\put(100,10){$\tilde X(u)^*$}
\put(55,2){\vector(1,0){100}}
\put(167,0){$\R^{\tilde m}$}
\put(45,-65){\vector(0,1){60}}
\put(50,-35){$(T_u p)^*$}
\put(30,-75){$T_x ^*M$}
\put(175,-65){\vector(0,1){60}}
\put(167,-75){$\R^m$}
\put(65,-73){\vector(1,0){90}}
\put(100, -87){$X(x)^*$}
\put(180,-35){$\ell_u$}
\put(185,2){\vector(1,0){100}}
\put(220,10){$\tilde X(u)$}
\put(285,0){$T_uN$}
\put(292,-5){\vector(0,-1){60}}
\put(300,-35){$T_up$}
\put(220, -87){$X(x)$}
\put(285,-75){$T_xM$}
\put(185,-73){\vector(1,0){100}}
\end{picture}
commutes, for $x=p(u)$, i.e.
 $\displaystyle{\sigma_x^\A 
=T_up\circ\sigma_x^\B  (T_up)^*}$
and  $\displaystyle{X(x)=T_up\circ \tilde X(u)\circ \ell_u}$.
In particular the horizontal lift map is given by
$h_u=\tilde X(u)\ell_uY(p(u))$.
\end{lemma}
\begin{proof} The larger square commutes by Lemma
 \ref{le:commutative}. For the rest we need to construct $\ell_u$.
 It suffices to define $\ell_u$ on $[\ker X(x)]^\perp$.
Note that $[\ker X(x)]^\perp =\Image X(x)^*$ in $\R^m$.
 We only have  to show that  $\alpha \in \ker X(x)^*$ implies
$$\tilde X(u)^*(T_up)^*\alpha=0.$$  In fact for such $\alpha$
  the proof of part (i) of  Proposition \ref{pr:h-map} is valid
and therefore $(T_up)^*\alpha\in \ker \sigma_u^\B  $.
  However since $\tilde X(u)$ is injective on the image of $\tilde X(u)^*$ we see $\ker \sigma_u^\B  =\ker \tilde X(u)^.$.
Thus $\ell_u$ is defined with $ker \ell_u=\ker X(x)$ and
such that the left hand square of the diagram commutes. Since the perimeter
commutes it is easy to see from the construction of $\ell_u$
that the right hand side also commutes. The uniqueness of $\ell_u$
with kernel equal that of $X(x)$ is clear since  on $[\ker X(x)]^\perp$
 $\ell_u(e)   =\tilde X(u)^*(T_up)^*X(x)(e)$.
\end{proof}

{\bf Note.} 
The horizontal lift of $X(x)$, which can be used to
construct a H\"ormander form\index{H\"ormander!form}  representation $X^V$ of $\A^H$, as in
Proposition \ref{th:deco} and Theorem
 \ref{th:op-deco} below is given by:
$$X^V(u): \R^m\to T_uP$$
$$X^V(u)=h_uX(u)=\tilde X(u)\ell_u$$
since $Y_xX(x)$ is the projection onto $\ker X(x)^\perp$.
(In the terminology of Elworthy-LeJan-Li \cite{Elworthy-LeJan-Li-book}
$X^V$ does not involve the `redundant noise'.)
Furthermore consider the special case  that $\tilde m=m$ and also that $\tilde X$ and $X$ are {\it p-related}, i.e.
$$T_up(\tilde X(u)e)=X(p(u))e, \hspace {0.7 in} u\in N, e\in \R^m.$$
Then $\ell_u$ is the projection of $\R^m$ onto $[\ker X(p(u))]^\perp$:
$$\ell_u=Y(p(u))X(p(u))$$
giving\begin{equation}
\label{eq:hor-lift}
h_u=\tilde{X}(u)Y(p(u))
\end{equation}
In this case the `diffusion coefficients' $X^V$, above, is
obtained from $\tilde X$ by restriction to the `relevant noise'
for $X$.

\section{Lifts of Cohesive Operators \& Decomposition Theorem \; }
\label{se:lift-dif}

A diffusion generator $\Lo$ on a manifold is said to be {\it \bf cohesive} if
\begin{enumerate}
\item [(i)] $\sigma_x^\Lo$, $x\in X$, has constant non-zero rank and
\item[(ii)] $\Lo$ is along the image of $\sigma^\Lo$.
\end{enumerate}

\begin{remark}
\label{re:cohesive}
From Theorem 2.1.1 in Elworthy-LeJan-Li \cite{Elworthy-LeJan-Li-book} we see that if the rank of $\sigma^\Lo_x$ is bigger than $1$ for all $x$ then $\Lo$ is  cohesive if and only if it has a representation
$$\Lo={1\over 2}\sum_{j=1}^m \LL_{X^j}\LL_{X^j}$$
where $E_x=\span\{X^1(x), \dots X^m(x)\}$ has constant rank.
\end{remark}

\begin{proposition}
\label{pr:cohesive}
Let $\B$ be a smooth diffusion operator on $N$ over $\A$ with $\A$ cohesive.  The following are equivalent:
\begin{enumerate}
\item [(i)] $\B=\A^H$
\item[(ii)]  $\B$ is cohesive and $T_up$ is injective on the image of $\sigma_u^\B$ for all $u\in N$.
\item[(iii)]
$\B$ can be written as 
$$\B={1\over 2}\sum_{j=1}^m \LL_{\tilde X^j}\LL_{\tilde X^j}+\LL_{\tilde X^0}$$
where $\tilde X^0, \dots, \tilde X^m$ are smooth vector fields on $N$ lying over smooth vector fields $X^0, \dots, X^m$ on $M$, \ie $T_up(\tilde X^j(u))=X^j(p(u))$ for $u\in N$
for all $j$.
\end{enumerate}
\end{proposition}
\begin{proof}
If (i) holds take smooth $X^1, \dots X^m$ with $\A={1\over 2}\sum_{j=1}^m \LL_{X^j}\LL_{X^j}+\LL_{X^0}$, by Proposition \ref{pr:along}, and set $\tilde X^j(u)=h_u X^j(p (u))$
to see (iii) holds.
Clearly (iii) implies (ii) and (ii) implies (i), so the three statements are equivalent.
\end{proof}
\begin{definition}
\label{no-vert}
If any of the equivalent conditions of the proposition holds we say that $\B$ {\it has no vertical part}.
\end{definition}

Recall that is $S$ is a distribution, $S^0$ denotes the set of annihilators of $S$.
\begin{lemma}
\label{le:symbols-2}
For $\ell\in H_u^0$ and $k\in (V_u TN)^0$, some $u\in N$ we have:
\begin{enumerate}
\item [A.] $\ell\sigma^\B (k)=0$
\item[B.]  $\sigma^\B(k)=\sigma^{\A^H}(k)$
\item [C.] $\sigma^{\A^H}(\ell)=0$.
\end{enumerate}
In particular $H_u$ is the orthogonal complement of $VT_uN\cap \Image (\sigma_u^\B)$ in $\Image (\sigma_u^\B)$ with its inner product induced by $\sigma_u^\B$.
\end{lemma}
\begin{proof}
Set $x=p(u)$. For  part A and  part B  it suffices to take $k=\phi\circ T_up$
some $\phi\in T_x^*M$. Then by (\ref{h-lift-1}),
$\sigma_u^\B(\phi\circ T_u p)=h_u \circ \sigma^\A_x(\phi)$ giving part A, and also part B by Proposition \ref{pr:lift} (iii)(a) since $\phi=h_u^*(\phi\circ T_u p)$, part C comes directly from Proposition \ref{pr:lift} (iii)(a).
\end{proof}
\begin{theorem}
\label{th:deco}
For $\B$ over $\A$ with $\A$  cohesive there is a unique decomposition
$$\B=\B^1+\B^V$$
where $\B^1$ and $\B^V$ are smooth diffusion generators with $\B^V$ vertical and $\B^1$ over $\A$ having no vertical part. In this decomposition $\B^1=\A^H$, the horizontal lift of $\A$ to $H$.
\end{theorem}
\begin{proof}
Set $\B^V=\B-\A^H$. To see that $\B^V$ is semi-elliptic take $u\in N$ and observe that any element of $T_u^*N$ can be written as $\ell +k$ where $\ell \in H_u^0$ and $k \in (VT_uN)^0$ by Lemma \ref{le:symbols-2} and
$$(\ell+k)\sigma^\B(\ell+k)=\ell \sigma^\B(\ell)\geqslant  0.$$
Since  $\B^V(f\circ p )=0$ any $f\in C^2(M; \R)$ Proposition \ref{pr:vert-eq} implies $\B^V$ is vertical.\\
Uniqueness holds since the semi-connections determined by $\B$ and $\B'$ are the same by Remark \ref{re:3.4}(i) applied to $\B^V$ and so by Proposition \ref{pr:cohesive} we must have $\B^1=\A^H$.
\end{proof}
For $p$ a Riemannian submersion and $\B$ the Laplacian, 
Berard-Bergery and Bourguignon \cite{Berard-Bergery-Bourguignon}
define $\B^V$ directly by
$\B^V f(u)=\Delta_{N_x}(f|_{N_x})(u)$ for $x=p(u)$ and $N_x=p^{-1}(x)$ with $\Delta_{N_x}$ the Laplace-Beltrami operator of $N_x$.\\
\begin{example}
\begin{enumerate}
\item Take $N=S^1\times S^1$ and $M=S^1$ with $p$ the projection on the first factor. Let
$$\B=\frac{1}{2}\big(\frac{\partial^2}{\partial x^2} +\frac{\partial^2}{\partial y^2})+\tan \alpha\frac{\partial^2}{\partial x\partial y}.$$
Here $0<\alpha<\frac{\pi}{4}$ so that $\B$ is elliptic. Then $\A=\frac{1}{2}\frac{\partial^2}{\partial x^2}$ \\ and $\B^V=\frac{1}{2}(1-(\tan\alpha)^2)\frac{\partial^2}{\partial y^2}$ with $\A^H=\frac{1}{2}(\frac{\partial^2}{\partial x^2} +(\tan\alpha)^2\frac{\partial^2}{\partial y^2})+\tan\alpha \frac{\partial^2}{\partial x\partial y}$.  
This is easily checked since, with this definition $A^H$ has H\"ormander form\index{H\"ormander!form} $$\A^H=\frac{1}{2}(\frac{\partial}{\partial x}+\tan \alpha\frac{\partial}{\partial y})^2$$ and 
so is a diffusion operator which has no vertical part. Also $\B^V$ is clearly vertical and elliptic.
 Note that this is an example of a Riemannian submersion: several more of a similar type can be found in \cite {Berard-Bergery-Bourguignon}. In this case the horizontal distribution is integrable and if $\alpha$ is irrational the foliation it determines has dense leaves.
 \item Take $N=\R^3$ with
 \index{Heisenberg group}%
  \emph{Heisenberg group} structure. This is defined by 
 $$ (x,y,z)\cdot(x',y', z')= \big(x+x',y+y', z+z'+\frac{1}{2}(xy'-yx')\big).$$
 Let $X,Y,Z$ be the left-invariant vector fields which give the standard basis for $\R^3$ at the origin. As operators: 
\begin{eqnarray*}
X(x,y,z)&=&\frac{\partial}{\partial x}-\frac{1}{2}y\frac{\partial}{\partial z}, \hspace{.3in}
 Y(x,y,z)=\frac{\partial}{\partial y}+\frac{1}{2}x\frac{\partial}{\partial z}\\
Z(x,y,z)&=&\frac{\partial}{\partial z}. 
\end{eqnarray*}

 Take $\B$ to be half the sum of the squares of $X,Y$, and $Z$. This is  half  the left invariant Laplacian: $$\B=\frac{1}{2}\left(\frac{\partial^2}{\partial x^2}+\frac{\partial^2}{\partial y^2}+ (1+\frac{1}{4}(x^2+y^2))\frac{\partial^2}{\partial z^2}+\frac{1
}{2} (x\frac{\partial^2}{\partial y\partial z}-y\frac{\partial^2}{\partial x\partial z})\right).$$

Take $M=\R^2$ and $p:\R^3\to \R^2$ to be the projection on the first $2$ co-ordinates.
Then \begin{eqnarray*}
\A&=&\frac{1}{2}(\frac{\partial^2}{\partial x^2}+\frac{\partial^2}{\partial y^2}), \qquad
 \A^H=\frac{1}{2}(X^2+Y^2);\\
\B^V&=&\frac{1}{2}Z^2=\frac{1}{2}\frac{\partial^2}{\partial z^2}.
\end{eqnarray*}
Note that the horizontal lift  $\tilde \sigma$, of a smooth curve $\sigma:[0,T]\to M$ with $\sigma(0)=0$, is given by 
 \begin{equation}
\label{eq:Heis.hor.lift}
\tilde{\sigma}(t)=\left(\sigma^1(t),\sigma^2(t), \frac{1}{2}\int_0^t \left(\sigma^1(t)d\sigma^2(t)-\sigma^2(t) d\sigma^1(t) \right)\right).
\end{equation}
Thus the ``vertical" component of the horizontal lift is the area integral of the curve.  Equation (\ref{eq:Heis.hor.lift}) remains valid for the horizontal lift of Brownian motion on $\R^2$ , or more generally for any continuous semi-martingale, provided it is interpreted as a Stratonovich equation ( or equivalently an Ito equation in the Brownian motion case).
		This example is also that of a Riemannian submersion. In this case the horizontal distributions are not integrable. Indeed the Lie brackets satisfy  $[X,Y]=Z$ and \emph{H\"ormander's  condition} for hypoellipticity: a diffusion operator $\Lo$ satisfies H\"ormander's condition if for some (and hence all) H\"ormander form representation such as in equation (\ref{op-B}) the  vector fields $Y^1,\dots,Y^m$ together with their iterated Lie brackets span the tangent space at each point of the manifold. For an enjoyable discussion of the Heisenberg group and the relevance of this example to ``Dido's problem" see \cite{Montgomery-book}. See also \cite {Baudoin-book},\cite{Bismut-book}, and \cite{Gromov-CC}.
\end{enumerate}
\end {example}

\bigskip

Recall that $F\equiv \sqcup_u F_u=\cup_u (T_up )^{-1}[E_{p (u)}]$, we can now strengthen Lemma \ref{le:lift-2} which states that $\Image[\sigma^\B_u]\subset F_u$.
\begin{corollary}
\label{co:deco}
If $\B$ is over $\A$ with $\A$ cohesive, then $\B$ is along $F$.
\end{corollary}
\begin{proof}
Since $H_u\in F_u$ and  $VT_uN\subset F_u$ both $\B^1$ and $B^V$ are along $F$.
\end{proof}

\section{Diffusion Operators with Projectible Symbols}

\label{se-basic symbol}

Given $p:N\to M$ as before,  suppose now that we have a diffusion operator $\B$ on $M$ with a projectible symbol, c.f. Definition \ref {basic-symbol}. This means that
$\sigma^{\B}$ lies over some positive semi-definite linear map $\eta:T^*M\to TM$.\emph{ Assume that $\eta$ has constant rank}. We will show that in this case we also have a decomposition of $\B$. To do this first choose some  cohesive diffusion operator $\A$ on $M$ with $\sigma^{\A}=\eta$. In general there is no canonical way to do this, though if $\eta$ were non-degenerate we could choose $\A$ to be a multiple of the Laplace-Beltrami operator of the induced metric on $M$. 

From above we also have an induced semi-connection with horizontal sub-bundle $H$, say, of $TN$. 

\begin {definition} We will say that $\B$ {\bf descends cohesively}\index{ cohesive!descends} (over $p$) if it has a projectible symbol and there exists a horizontal vector field, $b^H$, such that $$\B-\mathbf{ L}_{b^H}$$
is  projectible  over $p$.
\end{definition}
The following is a useful observation. Its proof is immediate from the two lemmas and proposition which are given after it:
\begin{proposition}\label{pr-descends-1}
 If $\B$ descends cohesively then for each choice of $\A$ satisfying $\sigma^\A_{p(u)}=T_up\sigma_u^\B (T_up)^*$ there is a horizontal vector field $b^H$
such that $\B-\mathbf{ L}_{b^H}$ lies over $\A$.
\end{proposition}

%
\begin{lemma}\label {le-der-1} 
Assume that $\eta$ has constant rank. If $f$ is a function on $M$ let  $\tilde f=f\circ p$.  For any choice of $\A$ with symbol $\eta$ the map $$f\mapsto \B(\tilde f)-\widetilde{\A(f)}$$
is a derivation from $ C^\infty M$ to $ C^\infty N$ where any $f\in C^\infty M$ acts on $C^\infty N$ by multiplication by $\tilde f$.
\end {lemma}
\begin {proof}
The map is clearly linear and 
for  smooth $f,g:M\to \R$ we have $$\widetilde{\eta(df,dg)}=\sigma^\B( d\tilde f, d\tilde g )$$  so by definition of symbols: \begin{equation*}
\B(\tilde f \tilde g)-\widetilde{{\A(fg)}}= \B(\tilde f)\tilde g+\B(\tilde g)\tilde f - \widetilde{\A(f)}\tilde g-\widetilde{\A(g)}\tilde f
\end{equation*}
as required.
\end{proof}
Let $\mathfrak{D}$ denote the space of derivations from $C^\infty M$ to $C^\infty N$ using the  above action.
Note that for $p^*TM\to N$ the pull back of $TM $ over $p$, the space $C^\infty\Gamma p^*TM$ of smooth sections of $p^*TM$ can be considered as the space of smooth functions $V:N\to TM$ with $V(u)\in T_{p(u)}M$ for all $u\in N$. We can then define
 $$\Theta:C^\infty\Gamma p^*TM\to \mathfrak D$$
 by $$\Theta (V)(f)(u)=df_{p(u)}(V(u)).$$
 \begin {lemma} \label{le-der-2}
Assume that $\eta$ has constant rank  The map $\Theta:C^\infty\Gamma p^*TM\to \mathfrak D$ is a linear bijection.
 \end{lemma}
\begin {proof}
Let $\mathfrak{d}\in \mathfrak D$. Fix $u\in N$. The map from $C^\infty M$ to $ \R$ given by $f\mapsto \mathfrak {d}f(u) $ is a derivation at $p(u)$, here the action of any $f\in C^\infty M$ on $\R$ is multiplication by $f(p(u))$, and so corresponds to a tangent vector, $V(u)$ say, in $T_{p(u)}M$. Then  $\mathfrak {d}f(u)=df_{p(u)}(V(u))$. By assumption $\mathfrak {d}f(u)$ is smooth in $u$, and so by suitable choices of $f$ we see that $V$ is smooth. Thus $\Theta (V)=\mathfrak {d}$ and $\Theta$ has an inverse.
\end{proof}
From these lemmas we see there exists $b\in C^\infty\Gamma p^*TM$ with the property that
 \begin{equation}
 \label{b1}
\left(\B\tilde f-\widetilde{\A f}\right)(u)=df_{p(u)}\big(b(u)\big)
\end{equation} for all $u\in N$ and $f\in C^\infty M$. \emph{Assume that $b$ has image in the subbundle $E$ of $TM$ determined by $\eta$.}  Using the horizontal lift map $h$ determined by $\B$ define  a vector field $b^H$
on $N$: $$b^H(u)=h_u\big(b(u)\big).$$
\begin {proposition}\label{pr-descends-2}
 Assume that $\eta$ has constant rank and 
that $b$ has image in the subbundle $E$  determined by $\eta$.
The vector field $b^H$  is such that $\B-b^H$ is over $\A$.
\end{proposition}
\begin{proof} For $f\in C^\infty M$, $$(\B-b^H)(\tilde f)=\widetilde{\A f}+df(b(-))-df\circ Tp(b^H(-))=\widetilde{\A f}$$
using the fact that $Tp \big(b^H (-)\big)=b(-)$.
\end {proof}
We can now extend the decomposition theorem:
\begin{theorem} \label{the-extended-deco}
Let $\B$ be a diffusion operator on $N$ which descends cohesively over $p:N\to M$. Then $\B$ has a unique decomposition: $$ \B=\B^H+\B^V$$
into the sum of diffusion operators such that
\begin{enumerate}
\item [(i)]  $\B^V$  is vertical 
\item[(ii)]  $\B^H$ is cohesive and $T_up$ is injective on the image of $\sigma_u^{\B^H}$ for all $u\in N$.
\end {enumerate}
With respect to the induced semi-connection $\B^H$ is horizontal.
\end{theorem}
\begin {proof}
Using the notation of the previous proposition we know that $\B-b^H$ is over  a cohesive  diffusion operator $\A$. By Theorem \ref{th:deco} we have a canonical decomposition
$$\B-b^H=\B^1+\B^V,$$ leading to $$ \B=(b^H+\B^1)+\B^V.$$
If we set $\B^H=b^H+\B^1$ we have a decomposition as required.
On the other hand if we have two such decompositions of $\B$ we get two decompositions of $\B-b^H$. Both components of the latter must agree by the uniqueness in Theorem \ref{th:deco}, and so we obtain uniqueness in our situation.
\end {proof}
Extending Definition \ref{no-vert} we could say that a diffusion operator $\B^H$ satisfying condition (ii) in the theorem {\bf has no vertical part}. 

Note that if we drop the hypothesis that $b^H$ is horizontal, or equivalently that $b$ in Proposition \ref{pr-descends-2} has image in $E$,  we still get a decomposition by taking an arbitrary lift of $b$ to be $b^H$ but we will no longer have uniqueness.

\section{Horizontal lift of paths \& completeness of semi-connections \;}
A semi-connection on $p:N\to M$ over a sub-bundle $E$ of $TM$ gives a procedure for horizontally  
lifting paths on $M$ to paths on $N$  as for ordinary connections but now we require the original path to have derivatives in $E$; such paths may be called {\bf E-horizontal}.

\begin{definition}\label{def:horizontal lifts of paths}
A Lipschitz path $\tilde{\sigma}$ in $N$ is said to be a {\bf horizontal  lift} of a path $\sigma$ in $M$ if  \begin{itemize}
  \item $p\circ\tilde{\sigma}=\sigma$
  \item The derivative of $\tilde{\sigma}$ almost surely takes values in the horizontal subbundle $H$ of $TN$.
 \end{itemize}
 \end{definition}
 
 Note that a Lipschitz path  $\sigma:[a,b]\to M$ with $\dot{\sigma}(t)\in\E_{\sigma(t)}$ for almost all $a\leqslant t\leqslant b$ has at most one horizontal lift from any starting point
 $u_a$ in $p^{-1}(\sigma (a))$. To see this first note that any such lift must satisfy\begin{equation}
\label{eq;deterministic lift}
\dot{\tilde{\sigma}}(t)=h_{\tilde{\sigma}(t)}\dot{\sigma}(t).
\end{equation}
This equation can be extended to give an ordinary differential equation on all of $N$. For example  take a smooth embedding $j:M\to \R^m$ into some Euclidean space. Set $\beta(t)=j(\sigma(t))$. Let $X(x): \R^m\to E_x$ be the adjoint of the restriction of the derivative  $T_xj$ of $j$ to $E_x$, using some Riemannian metric on $E$. Then  $\sigma$ satisfies the differential equation
 \begin{equation}
\label{eq:schwartz}
\dot{x}(t)=X(x(t))(\dot{\beta}(t))
\end{equation}
 and it is easy to see that the  horizontal lifts of $\sigma$ are precisely the solutions of 
 $$ \dot{u}(t)=h_{u(t)} X(p(u(t)))(\dot{\beta}(t))$$
 starting from points above $\sigma(a)$ and lasting until time $b$.
 
 In the generality in which we are working there may not be any such  solutions, for example because of "holes" in $N$. We define the semi-connection to be {\bf complete}\index{connection!complete} if every Lipschitz path $\sigma$  with derivatives in $E$ almost surely, has a horizontal lift starting from any point above the starting point of $\sigma$.
 
 Note that completeness is assured if the fibres of $N$ are compact, or if an $X$, with values in $E$, and $\beta$,  can be found so that $\sigma$ is a solution to equation (\ref{eq:schwartz}) and there is a complete metric on $N$ for which the horizontal lift of $X$ is bounded on the inverse image of $\sigma$ under $p$. In particular the latter will hold if $p$ is a principal bundle and we have an equivariant semi-connection as in the next chapter. It will also hold if there is a complete metric on $N$ for which  the  horizontal lift map  $h_u\in \mathbb{L}(E_{p(u)};T_uN)$ is uniformly bounded for $u$ in the image of $\sigma$.

\section{Topological Implications} \label{se:topology}

Although our set up of intertwining diffusions with a cohesive $\A$ seems quite general it implies strong topological restrictions if the manifolds are compact and more generally. Here we partially extend the approach Hermann used for Riemannian submersions in \cite{Hermann} with a more detailed discussion in Chapter \ref{ch:comm} below.

For this let $\D^0(x)$ \index{$\D^0(x)$}%
 be the set of points $z\in M$ which can be reached by  Lipschitz curves $\sigma:[0,t]\rightarrow M$ with $\sigma(0)=x_0$ and $\sigma(t)=z$ with 
 derivative in $E$ almost surely.  Its closure $\D'(x)$ \index{$\D'(x)$}%
  relates to the propagation set for the maximum principle for $\A$, and to the support of the $\A$- diffusion as in Stroock-Varadhan \cite{Stroock-Varadhan-deg}, see Taira\cite {Taira}.

\begin{theorem}
\label{th:topology}
 For $\B$ and $\A$ as before with $\A$ cohesive take $x_0\in M$ and $z\in \D^0(x_0)$. 
 Assume the induced semi-connection is complete.
 Then if $p^{-1}(x_0)$ is a submanifold of $N$ so is $p^{-1}(z)$ and they are diffeomorphic. Also if $z$ is a regular value of $p$ so is $x$.
 \end {theorem}
 \begin {proof}
Let $\sigma ;[0,T]\to M$ be a Lipschitz $E$-horizontal path from $x$ to $z$. There is 
 a smooth factorisation
 $\sigma_x^\A=X(x)X(x)^*$ for $X(x)\in \Lo(\R^m;T_xM)$, $x\in M$. Take the horizontal lift  $\tilde X: \underline \R^,\to TN$ of $X$.

By the completeness hypothesis the time dependent ODE on $N$, 
 $$ {dy_s\over ds}= \tilde{X}(y_s)X\Big(\sigma(s)|_{[\ker X(x_0)]^\perp}\Big)^{-1}(\dot \sigma(s))$$
 will have solutions from each point above $\sigma(0)$ defined up to time $T$ and so  a flow giving the required diffeomorphism of fibres. Moreover, by the usual lower semi-continuity property of the "explosion time",  this holonomy flow \index{flow!holonomy} gives a diffeomorphism of a neighbourhood of  $p^{-1}(x)$ in $N$ with a neighbourhood of the fibre above $z$.  The diffeomorphism commutes with $p$. Thus if one of $x$ and $z$ is a regular value so is the other.
 \end{proof}

 \begin{corollary}\label {co:hor}
 Assume the conditions of the theorem and that $E$ satisfies
the standard H\"{o}rmander condition that the Lie algebra of vector fields generated by sections of $E$ spans each tangent space $T_yM$ after evaluation at $y$.  Then $p$ is a submersion all of whose fibres are diffeomorphic.
\end{corollary} 
\begin{proof}
The H\"{o}rmander condition implies that $\D^0(x)=M$ for all $x\in M$ by Chow's theorem (\eg see Sussmann \cite{Sussmann} or \cite {Gromov-CC}. In \cite{Gromov-CC} Gromov shows that under this condition any two points of $M$ can be joined by a smooth E-horizontal curve.
\end{proof}

 \begin {corollary}
 \label{cor:top}
 Assume the conditions of the theorem and that $\D^0(x)$ is dense in $M$ for all $x\in M$ and $p:N\to M$ is  
 proper. Then $p$ is a locally trivial bundle over $M$.
 \end{corollary}
 \begin {proof}
 Take $x\in M$. The set $Reg(p)$ of regular values of $p$ is open by our properness assumption. It is also non-empty, even dense in $M$, by Sard's theorem,  and so since $\D^0(x)$ is dense, there exists  a regular value $z$ which is in $\D^0(x)$. It follows from the theorem that $x\in Reg(p)$, and so $p$ is a submersion. However it is a well known consequence of the inverse function theorem that a proper submersion is a locally trivial bundle.
  \end{proof}

 Note that we only need $Reg(p)$ to be open, rather than $p$ proper, to ensure that $p$ is a submersion. The density of $\D^0(x)$ can hold
because of global behaviour, for example if $M$ is a torus and $E$ is  tangent to the foliation given by an irrational flow.

\chapter{Equivariant Diffusions on Principal Bundles}
\label{ch:deco-2}
Let $M$ be a smooth finite dimensional manifold and
 $\displaystyle{P\left( M,G\right)} $ a
 principal fibre bundle over $M$ with structure group $G$ a  Lie group.
 Denote by  $\displaystyle{\pi :P\rightarrow M}$  the projection and $R_a$
 right  translation by $a$. Consider on $P$ a diffusion generator
 $\B$, which is {\bf equivariant}, i.e. for all $f \in C^2(P;\R)$,
$$\B f\circ R_a=\B(f\circ R_a),
\qquad \; a\in G.$$
Set $\displaystyle{f^a(u)=f(ua)}$. Then the above equality can be written as
$\displaystyle{\B f^a=\left( \B f\right) ^a}$.
 The operator $\B$ induces an
operator $\A $ on the base manifold $M$. Set
 \begin{equation}
\label{operator-equivariant-induced}
\A f(x)=\B\left( f\circ \pi \right) (u),
\hskip 20pt u\in \pi^{-1}(x), f\in C^2(M),
\end{equation}
which  is well defined since
 $$\B\left( f\circ \pi\right) \left( u\cdot a\right)
 =\B\left( \left( f\circ \pi \right)^a\right) (u)
=\B\left( \left( f\circ \pi \right) \right) (u).$$

\section{Invariant Semi-connections on Principal Bundles}

\begin{definition}
Let $E$ be a sub-bundle of $TM$ and $\pi:P\to M$ a principal $G$-bundle.
An {\bf invariant semi-connection } over $E$,  or  {\bf principal semi-connection } in the terminology of  Michor, on $\pi:P\to M$ is a smooth sub-bundle
$H^ETP$ of $TP$ such that

\begin{enumerate}
\item[(i)]
$T_u\pi$ maps the fibres $H^ET_uP$ bijectively onto $E_{\pi(u)}$
for all $u\in P$.
\item[(ii)]
$H^ETP$ is $G$-invariant.
\end{enumerate}
\end{definition}

{\bf Notes. }
\begin{enumerate}
\item  Such a semi-connection determines and is determined by, a smooth
horizontal lift:
$$h_u: E_{\pi(u)}\to T_uP$$
such that
(i). $T_u\pi \circ h_u(v)=v$, for all $v\in E_x\subset T_xM$;\\
(ii). $h_{u\cdot a}=T_uR_a\circ h_u$.

\item The action of $G$ on $P$ induces a homomorphism of the Lie algebra
 $\g$ of $G$ with the algebra of left invariant vector fields
on $P$: if $A\in \g$,
$$A^*(u)=\left.{d\over dt}\right|_{t=0}\; u \exp(tA), \qquad u\in P,$$
and $A^*$ is called the fundamental vector field corresponding to $A$.

Using the splitting (\ref{split}) of $F_u$ our semi-connection determines, (and is determined by), a `semi-connection one-form' $\varpi\in \Lo(H+VTN;\g)$ which vanishes on $H$ and has $\varpi(A^*(u))=A$.

\item  Let $F$ be an associated vector bundle to $P$ with fibre $V$.
 An $E$ semi-connection on $P$ gives a covariant derivative
$\nabla_w Z\in F_x$ for $w\in E_x$, $x\in M$ where $Z$ is a section of $F$.
This is defined, as usual for connections, by
$$\nabla_w Z=u\big(d(\tilde Z)(h_u(w)) \big),$$
$u\in \pi^{-1}(x)$. Here $\tilde Z: P\to V$ is
$$\tilde Z(u)=u^{-1}Z\left(\pi(u)\right)$$
considering $u$ as an isomorphism $u:V\to F_{\pi(u)}$.
This agrees with the `semi-connections on $E$' defined in
Elworthy-LeJan-Li \cite{Elworthy-LeJan-Li-book}
 when $P$ is taken to be the linear frame bundle of $TM$  and $F=TM$.
\end{enumerate}

\begin{theorem}
\label{theorem:connection}
Assume $\sigma^\A $ has constant rank. Then $\sigma^\B$
 gives rise to an invariant semi-connection on  the principal bundle $P$ whose
 horizontal map is given by (\ref{h-lift-2}).
\end{theorem}
\begin{proof}
It has been shown that $h_u$ is well defined by
(\ref{h-lift-2}).  Next we show $h_u$ defines a semi-connection.
 As noted earlier, $h$  defines a  semi-connection if
(i) $T_u\pi \circ h_u(v)=v$,
 $v\in E_x\subset T_xM$ and (ii) $h_{u\cdot a}=T_uR_a\circ h_u$.
The first is immediate by Lemma \ref{le:commutative} and for
the second observe
$\pi \circ R_a=\pi $.
So $T\pi \circ TR_a=T\pi $ and $\left( T\pi \right) ^{*}=\left( TR_a\right)
^{*}\cdot \left( T\pi \right) ^{*}$ while the following diagram

\begin{picture}(500, 140)(0,-90)
\put(30,0) {$T_u ^*P$}
\put(100,10){$\sigma_u^\B$}
\put(65,2){\vector(1,0){100}}
\put(167,0){$T_uP$}
\put(45,-65){\vector(0,1){60}}
\put(50,-35){$(T_uR_a)^*$}
\put(30,-75){$T_{u\cdot a}^*P $}
\put(175,-4){\vector(0,-1){60}}
\put(167,-75){$T_{ua}P$}
\put(65,-73){\vector(1,0){90}}
\put(100, -87){$\sigma_{u\cdot a}^\B$}
\put(177,-35){$T_u R_a$}
\end{picture}
commutes by equivariance of $\B  $. Therefore
\begin{eqnarray*}
T_uR_a\circ h_u &=&T_uR_a\cdot \sigma _u^{\B  }\left( T_u\pi \right)
^{*}\circ \left( \sigma _x^{\A  }\right) ^{-1} \\
&=&T_uR_a\cdot \sigma _u^{\B  }\circ \left( T_uR_a\right) ^{*}\circ
\left( T_{u\cdot a}\pi \right) ^{*}\circ \left( \sigma _x^{\A  %
}\right) ^{-1} \\
&=&\sigma _{u\cdot a}^{\B  }\circ \left( T_{u\cdot a}\pi \right)
^{*}\circ \left( \sigma _x^{\A  }\right) ^{-1}=h_{u\cdot a}.
\end{eqnarray*}
\end{proof}

Curvature forms \index{curvature!form} and holonomy groups \index{holonomy group} etc for semi-connections are defined analogously to those associated two connections, we note the following: 

\begin{proposition}
\label{pr:h-funct-2}
In the situation of Proposition \ref{pr:h-funct-1} suppose
 $\A $ is elliptic, $p^1$, $p^2$ are principal bundles with groups $G^1$ and $G^2$ respectively, and $F$ is a homomorphism of principal bundles with corresponding homomorphism $f:G^1\to G^2$. Let $\Gamma^1$ and $\Gamma^2$ be the semi-connections on $N^1$, $N^2$ determined by $\B^1$ and $\B^2$. Then
\begin{enumerate}
\item[(i)]
$\Gamma^2$ is the unique semi-connection on $p^2:N^2\to M$ such that $TF$ maps
the horizontal subspaces of $TN^1$ into those of $TN^2$.
\item[(ii)] If $\omega^j$, $\Omega^j$ are the semi-connection and curvature form of $\Gamma^j$, for $j=1,2$, then
$$F^*(\omega^2)=f_*\circ \omega^1$$
and
$$F^*(\Omega^2)=f_*\circ \Omega^1$$
for $f_*: \underline{g}_1\to  \underline{g}_2$ the homomorphism of Lie
 algebras induced by $f$.
\item[(iii)]
Moreover $f:G^1\to G^2$ maps the $\Gamma^1$ holonomy group at $u\in N^1$ onto the $\Gamma^2$ holonomy group at $F(u)$ for each $u \in N^1$ and similarly for the restricted holonomy groups.
\end{enumerate}
\end{proposition}

\begin{proof}
Proposition \ref{pr:h-funct-1} assures us that $TF$ maps
horizontal to horizontal. Uniqueness together with (ii), (iii) come
as in Kobayashi-Nomizu \cite{Kobayashi-Nomizu-II}
(Proposition 6.1 on p79).\end{proof}

\section{Decompositions of Equivariant Operators}
\label{se-decomposition-2}
Take a basis ${A_1, \dots, A_n}$ of $\g$ with
corresponding fundamental vector fields $\{A_i^*\}$. Write the semi-connection 1-form as $\varpi=\sum\varpi^k\,A_k$ so that  $\varpi^k$ are
real valued, partially defined, 1-forms on $P$. 

In our equivariant situation we can give a more detailed description of the decomposition in Proposition \ref{th:deco}.

 \begin{theorem}
\label{th:op-deco}
 Let $\B$ be an equivariant  operator on $P$ and $\A $ be the induced operator on the base manifold. Assume that $\A$ is cohesive and let $\B=\A^H+\B^V$ be the decomposition of Proposition \ref{th:deco}. Then $\B^V$ has a unique expression of the form
 $\displaystyle{\sum \alpha^{ij} \Lo   _{A_i^*}\Lo   _{A_j^*}
  +\sum\beta^k\Lo   _{A_k^*}}$,
 where $\displaystyle{\alpha^{ij}}$ and $\displaystyle{\beta^k}$ are
 smooth functions on $P$, given by
$\displaystyle{\alpha^{k\ell}
=\varpi^k\left(\sigma^\B  (\varpi ^\ell)\right)}$,
and $\displaystyle{\beta^\ell=\delta^\B  (\varpi^\ell)}$
for  $\varpi$  the semi-connection 1-form on $P$.
Define  $\alpha: P\to \g\otimes\g$ and
 $\beta: P\to \g$ by
\begin{equation}\label{alpha}
\alpha(u)=\sum\alpha^{ij}(u)A_i\otimes A_j,  \qquad
\beta(u)=\sum\beta^k(u) A_k.
\end{equation}
These are independent of the choices of basis of $\g$ and are equivariant:
$$\alpha(ug)=\left(ad(g)\otimes ad(g)\right) \alpha(u)$$
and
$$\beta(ug)=ad(g)\beta(u).$$
 \end{theorem}

\begin{proof}
Since every vertical vector field is a linear combination of the fundamental
 vertical vector fields,  Proposition \ref{pr:vert-eq},
shows that   $$\B^V
 =\sum \alpha^{i,j}\Lo   _{A_i^*}\Lo   _{A_j^*}
+\sum \beta^k \Lo   _{A_k^*}$$
for certain functions $\alpha^{ij}$, $\beta^k$.
For $f,g: P\to \R$ setting
 $\sigma:=\sigma^{\B  -\A ^H}$,
 \begin{eqnarray*}
df\left(\sigma(dg)\right)&=&{1\over 2}
\sum \alpha^{i,j}\Lo   _{A_i^*}\Lo   _{A_j^*}(fg)
-{1\over 2}\sum g\alpha^{i,j}\Lo   _{A^*_i}\Lo   _{A^*_j}(f)\\
&&
-{1\over 2}\sum f\alpha^{i,j}\Lo   _{A^*_i}\Lo   _{A^*_j}(g)\\
&=&\sum \alpha^{i,j}\Lo   _{A^*_i}(f)\Lo   _{A^*_j}(g)\\
&=&\sum \alpha^{i,j}df({A_i}^*)dg({A_j}^*).
 \end{eqnarray*}
  Since  $\varpi(A_k^*)=A_k$,  we see that
  $ \varpi^k(A_{\ell}^*)=\delta_{k\ell}$ and
 $$\varpi^k(\sigma(\varpi^\ell))
=\sum \alpha^{i,j}\delta_{ik}\delta_{j\ell}=\alpha^{k\ell}.$$
Since $\A ^H$ is horizontal $\sigma^{\A ^H}$ has image
in the horizontal tangent bundle and so is annihilated by $\varpi^k$.
Thus
\begin{equation}
\label{alpha-2}
\alpha^{k\ell}
=\varpi^k\left(\sigma^\B  (\varpi ^\ell)\right).
\end{equation}
Note that by the characterisation, Proposition \ref{pr:delta},
$$\delta^{\B^V}=
\sum \alpha^{i,j}\Lo   _{{A_i}^*}{\iota}_{{A_j}^*}
+\sum\beta^k \iota_{{A_k}^*}.$$
Since $\varpi^\ell({A^*}^\ell)$ is identically $1$,
it follows that $\delta^{\B^V}(\varpi)=\beta^\ell$.
 Again $\displaystyle{\delta^{\A ^H}(\varpi^\ell)=0}$
and  so
\begin{equation}
\label{beta-2}
\beta^\ell=\delta^\B  (\varpi^\ell)
\end{equation}
 as required.

For the last part $\alpha$ and $\beta$ can be considered as
 obtained from the extension of the symbol $\sigma^\B  $ and
$\delta^\B  $ to   $\g$-valued two and one forms
 respectively:  $\alpha=\varpi(-)\sigma^\B  \varpi(-)$ and
 $\beta=\delta^\B  (\varpi(-))$.  
  To make this precise consider
$\sigma_u^\B  $  as a bilinear form and so as a linear map
$$\sigma_u^\B  : T_u^*P\otimes T_u^*P\to \R.$$
The extension is the trivial one given by
$$\sigma_u^\B   \otimes \1\otimes \1:
T_u^*P\otimes T_u^*P\otimes\g\otimes\g\to
\R\otimes\g\otimes\g
\simeq \g\otimes\g$$
 using the identification of  $T_u^*P \otimes \g$ with
  $L(T_uP;\g)$.
Similarly the extension of $\displaystyle{\delta^\B  }$ is
$$\delta^\B  _u\otimes \1:  T_u^*P\otimes \g
\to \R\otimes\g\simeq \g.$$
 Thus $$\alpha(u)(\omega\otimes \omega)
=\left(\sigma^\B \otimes \1\otimes \1\right)\left (P_{23}
\omega\otimes \omega\right)$$
where $P_{23}: T^*P\otimes \g\otimes T^*P\otimes \g
\to T^*P\otimes T^*P\otimes \g \otimes\g$ is
the standard permutation and
 $\beta_u(\omega)=(\delta^\B_u\otimes \1)(\omega)$.

The equivariance of $\displaystyle{\varpi}$
$$(R_g)^*\varpi =ad(g^{-1})(\varpi), \hskip 15pt g\in G$$
is equivalent to the invariance of $\varpi$ when considered
as a section of $T^*M\otimes \g$ under
$$TR_g\otimes ad(g): T^*M\otimes \g\to
T^*M\otimes \g, \hskip 25pt g\in G.$$
 \end{proof}

\begin{remark}
\label{re:alpha}
\begin{enumerate}
\item[(a)] For any equivariant operator of the form $\B=\sum_{i,j}\alpha^{ij}L_{A_i^*}L_{A_j^*}+\sum \beta^k L_{A_k^*}$ with $(\alpha^{ij}(u))$
positive semi-definite for each $u\in P$ we can define maps $\alpha$ and $\beta$ by (\ref{alpha}). Note that $\alpha(u)$ is essentially  the symbol of $\B$ restricted to the fibre $P_{\pi(u)}$ through $u$:
$$\sigma_u^{\B}|_{P_{\pi(u)}}:  T_u^*P_{\pi(u)} \to T_uP_{\pi(u)}$$
with $\varpi_u$ identifying $T_uP_{\pi(u)}$ with $\mathfrak g$. Similarly $\beta$ determines $\delta^\B$ on a basis of sections of $(VTP)^*$.
\item[(b)] Let $\{u_t: 0\leqslant t\leqslant \zeta\}$ be a $\B$-diffusion on $P$. By (\ref{alpha-2}), $2\alpha^{kl}(u_t)$ is the derivative of the bracket 
$\Big\langle \int_0^\cdot \varpi^k_{u_s}\circ du_s, 
\int_0^\cdot \varpi^l_{u_s}\circ du_s \Big\rangle$ of
the integrals of $\omega^k$ and $\omega^l$ along $\{u_t: 0\leqslant t<\zeta\}$.
See chapter \ref{ch:intertwined} below for a detailed discussion. Thus
$\alpha(u_t)$ is the derivative of the tensor quadratic variation:
$$\alpha(u_t)={1\over 2} {d \over dt}\int_0^t \Big(\varpi_{u_t}\circ du_t \otimes \varpi_{u_t} \circ du_t\Big).$$
Moreover by (\ref{beta-2}) and Lemma \ref{le:mart-1} below $\int_0^t \beta(u_s) ds$ is the bounded variation part of $\int_0^t \varpi_{u_s}\circ du_s$.
\item[(c)] If we fix $u_0\in P$ and take an inner product on $\g$ we can diagonalise  $\alpha(u_0)$ to write 
$$\alpha(u_0)=\sum_n \mu_n A_n\otimes A_n$$
where $\{A_n: n=1, \dots \dim(\g)\}$ is an orthonormal basis. The
$\mu_n$ are the eigenvalues of $\alpha(u_0)^{\#}:\g\to \g$ obtained using the isomorphism:
\begin{eqnarray*}
\g\otimes \g& \to & \LL(\g;\g)\\
a\otimes b&\mapsto& (a\otimes b)^{\#},
\end{eqnarray*}
where $(a\otimes b)^{\#}(v)=\langle b, v \rangle a$.

Note that for $g\in G$, 
$\alpha(u_0\cdot g)=\sum_n \mu_n \ad(g)A_n\otimes \ad(g)A_n$.
When the inner product is $\ad(G)$-invariant then $\{\ad(g)A_n\}_{n=1}^{\dim(\g)}$ is still orthonomal and the $\{\mu_n\}_n$ are the eigenvalues of $\alpha(u_0\cdot g)^{\#}$. They are therefore independent of the choice of $u_0$ in a given fibre, (but depend on the inner product chosen).

\end{enumerate}
\end{remark}

\section{Derivative Flows and Adjoint Connections} 
\label{se:deri-flow}

Let $\A$ on $M$ be given in H\"ormander form\index{H\"ormander!form}
\begin{equation}
\label{op-deri}
\A={1\over 2}\sum_{j=1}^m \Lo   _{X^j}\Lo   _{X^j}
+\Lo   _{A}
\end{equation}
for some smooth vector fields $X^1,\dots X^m$, $A$. As before let
$E_x=\span\{X^1(x),\dots, X^m(x)\}$ and assume
 $\dim E_x$ is constant, denoted by $p$, giving a sub-bundle $E\subset TM$. The vector fields
$\{X^1(x),\dots, X^m(x)\}$ determine a vector bundle map
$$X:\underline {\R}^m\to   TM$$
with $\sigma^\A=X(x)X(x)^*$.

We can, and will, consider $X$ as a map $X: \underline{\R}^m\to E$.
Let $Y_x$ be the right inverse $[X(x)|_{\ker X(x)^\perp}]^{-1}$ of $X(x)$ and $\langle, \rangle_x$ the inner product, induced on $E_x$ by $Y_x$. Then $X$ projects the flat connection on $\R^m$ to a metric connection $\breve \nabla$ on $E$ defined by
\begin{equation}
\label{connection-1}
\breve \nabla _v U=X(x)d[y\mapsto Y_y U(y)](v), \qquad U\in C^1\Gamma E, v\in T_yM,
\end{equation}
(In \cite{Elworthy-LeJan-Li-book} we have studied the properties of this construction together with the SDE induced by $X$, and there $\breve \nabla$ is referred as the LW connection \index{connection!LW}%
\index{LW connection}%
 for the SDE.)
Moreover any connection $\nabla$ on a subbundle $E$ of $TM$ has an adjoint semi-connection \index{adjoint!semi-connection}%
 $ \nabla^\prime$ on $TM$ over $E$ defined by
$$\nabla^\prime_UV=\nabla_V U+[U,V], \qquad U\in \Gamma E, V\in \Gamma TM.$$

Let $\pi: GLM\to M$ \index{$GLM$} be the frame bundle \index{frame bundle} of $M$, so $u\in \pi^{-1}(x)$ is a linear isomorphism $u:\R^n\to T_xM$. It is a principal bundle with group $GL(n)$. If $g\in GL(n)$ and $\pi(u)=x$ then $u\cdot g: \R^n\to T_xM$ is just the composition of $u$ with $g$. 


Any smooth vector field $A$ on $M$ determines smooth vector fields $A^{TM}$ and $A^{GL}$ on $TM$ and $GLM$ respectively as follows: Let $\eta_t: t\in (-\epsilon, \epsilon)$ be a (partial) flow for $A$ and $T\eta_t$  its derivative. Then $v\mapsto T\eta_t(v)$ is a partial flow on $TM$ and $u\mapsto T\eta_t\circ u$  one on $GLM$, Let $A^{TM}$ and $A^{GL}$ be the vector fields generating  these flows.
In fact $A^{TM}$ is $\tau\circ TA: TM\to TTM$ where $\tau:TTM\to TTM$ is the canonical twisting map:
$$\tau(x,v, w,v^\prime)=(x,v,v^\prime,w)$$
in local coordinates.

Using this, the choice of our  H\"{o}rmander form representation
induces a diffusion operator $\B$ on $GLM$ by setting 
$$\B={1\over 2}\sum \LL_{(X^j)^{GL}} \LL_{(X^j)^{GL}}+\LL_{A^{GL} }.$$
Then $\pi$ intertwines $\B$ and $\A$.
For $w\in E_x$, set
$$Z^w(y)=X(y)Y_x(w).$$

\begin{theorem}
\label{th:conn-deri}
Assume the diffusion operator $\A$ given by (\ref{op-deri}) is cohesive
and let $\B$ be the operator on $GLM$ determined by $\A$. Let $E$ be the image of $\sigma^\A$, a vector bundle. 

\begin{enumerate}
\item[(a)] The semi-connection $\nabla$ induced by $\B  $ is the adjoint of $\breve \nabla$ given by (\ref{connection-1}). Consequently
$\nabla_wV=L_{Z^w}V$ for any vector field $V$ and $w\in E_x$,

\item[(b)] For $u\in GLM$, identifying ${\mathfrak gl}(n)$ with $\Lo(\R^n;\R^n)$,
\begin{eqnarray*}
\alpha(u)&=&{1\over 2}\sum \left(u^{-1}(-)\breve \nabla_{u(-)}X^p\right)
\otimes \left(u^{-1}(-)\breve\nabla_{u(-)} X^p\right),\\
\beta(u)&=&-{1\over 2}\sum u^{-1}
\breve \nabla_{\breve \nabla_{u(-)}X^p}X^p
-{1\over 2} u^{-1} {\Ric }^{\#}{u(-)}+u^{-1}\breve\nabla_{u(-)}A.
\end{eqnarray*}
Here ${\Ric}^{\#}:TM \to E$ is the Ricci curvature \index{curvature!Ricci} of $\breve \nabla$ considered as an operator from $TM$ to  $E$, defined by 
$${\Ric}^{\#}(v)=\sum_{j=1}^m \breve R\big(v, X^j(x)\big)X^j(x)$$
for $\breve R$ the curvature operator \index{curvature!operator}\index{operator!curvature} of $\breve \nabla$.
\end{enumerate}
\end{theorem}

\begin{proof}
The first part can be deduced from the stochastic flow results in chapter \ref{ch:flow}
but we give a direct proof here.
 Let  $\pi_t^e$ be the flow of $X(\cdot)(e)$.  It induces a linear map
 $\tilde X(u): \R^m\to T_uGLM$
 on the general linear bundle $GLM$:
\begin{eqnarray*}
\tilde X(\cdot)e&=&[X(\cdot)(e)]^{GL}\\
\tilde X(u)(e)&=&{d\over dt}(TS^e_t\circ u)|_{t=0},\hspace{0.3in}
 u\in GLM.
\end{eqnarray*}
We can apply lemma \ref{factorization-connection} with  $\tilde \R^m=\R^m$ and so $\ell_u=Y(p(u))X(p(u))$.
 If  $x= p(u)$ and  $e\perp \hbox{ker}[X(x)]$ then the horizontal lift
map $h_u$ defined by Theorem \ref{theorem:connection} is
\begin{equation}
\label{lift-derivative-flow}
h_u\left(X(x)(e)\right)=\tilde X(u)\left(\ell_u(e)\right)
=\left.{d\over dt}\right\vert_{t=0} \left(T\pi_t^e\circ u\right).
\end{equation}
Note this will not hold in general if $e\in \hbox{ker} [X(x)]$.

Let $\sigma: [0,T]\to M$ be a $C^1$ curve with
 $\dot \sigma(t)\in E_{\sigma(t)}$ each $t$. Then
$$Z^{\dot\sigma(t)}(x):=X(x)Y_{\sigma(t)}\dot \sigma(t).$$
Let $S_{s,t}^\sigma$ be the flow, from time $s$ to time $t$,
of the time dependent vector field $Z^{\dot \sigma(t)}$.
Now $S_{s,t}^\sigma(\sigma(s))=\sigma(t)$ for $ 0\leqslant s\leqslant t\leqslant T$.
Also, for any  torsion free connection  and any $v\in T_{\sigma(s)}M$
$$\left. {D\over dt}\right\vert _{t=s} TS_{s,t}^\sigma(v)
=\nabla Z^{\dot \sigma(t)}\left(TS_{s,t}^\sigma(v)\right)|_{t=s}
=\nabla_vZ^{\dot \sigma(s)}.$$
Thus $${D\over dt} TS_{0,t}^\sigma(v)
= \nabla_{TS_{0,t}^\sigma} Z^{\dot \sigma(t)}.$$
If $\varpi$ is the connection form of this torsion free connection then
\begin{eqnarray*}
\varpi\left({D\over dt}TS_{0,t}^\sigma\circ u_0\right)
&=& [e\mapsto  \left(TS_{0,t}^\sigma\circ u_0\right)^{-1}
{D\over dt} TS_{0,t}^\sigma(u_0(e))]\\
&=&[e \mapsto  \left(TS_{0,t}^\sigma\circ u_0\right)^{-1}
   \nabla_{TS_{0,t}^\sigma u_0(e)} Z^{\dot \sigma(t)}]\\
&=&\varpi\left(h_{TS_{0,t}^\sigma\circ u_0}(\dot \sigma(t))\right)
\end{eqnarray*}
by (\ref{lift-derivative-flow}),  showing that  the vertical parts of
 ${d\over dt} \left(TS_{0,t}^\sigma\circ u_0\right)$
and $h_{TS_{0,t}^\sigma\circ u_0}(\dot \sigma(t))$ equal.

On the other hand, using this auxiliary connection, the horizontal parts of
${d\over dt} \left(T S_{0,t}^\sigma\circ u_0\right)$
and $h_{TS_{0,t}^\sigma\circ u_0}(\dot \sigma(t))$
are both equal to the horizontal lift of $\dot \sigma(t)$.
Thus
$${d\over dt} \left(TS_{0,t}^\sigma\circ u_0\right)
=h_{TS_{0,t}^\sigma\circ u_0} (\dot \sigma(t))$$
and so $\{TS_{0,t}^\sigma\circ u_0: 0\leqslant t\leqslant T\}$ is the horizontal
lift of $\{\sigma(t): 0\leqslant t \leqslant T\}$ with respect to the semi-connection
induced by $\B  $. However  by Lemma 1.3.4 in
Elworthy-LeJan-Li \cite{Elworthy-LeJan-Li-book},
 $TS_{0,t}^\sigma(v)$ of $S_{0,t}^\sigma$
is  the parallel  translation of $v$ along  $\sigma$  by the adjoint
 semi-connection  $\hat \nabla$  of the LeJan-Watanabe connection on $E$
associated to $X$ and  $\{TS_{0,t}^\sigma\circ u_0: 0\leqslant t\leqslant T\}$ is
the horizontal lift of $\{\sigma(t): 0\leqslant t \leqslant T\}$ with respect
 to $\hat \nabla$. This proves the first claim.
And $\nabla_wV=L_{Z^w}V $ by  Lemma 1.3.4  of
Elworthy-LeJan-Li \cite{Elworthy-LeJan-Li-book}.

For the last part let $\varpi: H\oplus VTGLM\to \g=L(\R^n;\R^n)$ be the semi-connection 1-form. For $u_0\in GLM$, set
$u_t=T\xi_t \circ u_0$ 
where $\{\xi_t\}$ is a local flow \index{flow!of SDE} for the stochastic differential equation \begin{equation}
\label{SDE-1}
dx_t=X(x_t)\circ dB_t +A(x_t)dt 
\end{equation}
on $M$ where $\{B_t\}$ is a Brownian motion on $\R^m$. (This defines the \emph{derivative flow} \index{flow!derivative} on $GLM$.)

As for ordinary connections
$$\varpi(\circ du_t)=u_t^{-1}{\hat D\over dt}(u_t-)\in \Lo(\R^n;\R^n).$$
Here, on the right hand side $u_t$ is differentiated as a process of linear maps $u_t\in \Lo(\R^n; T_{x_t}M)$ over $(x_t)$. 
[It suffices to check the equality for $C^1$ curves $(u_t)$ with $x_t=\pi(u_t)$ having $\dot x_t\in E_{x_t}$, $t\geqslant  0$. For this we can write $u_t=\tilde x_t\cdot g_t$ for $\tilde x_t$ a horizontal lift of $\{x_t\}$ and $g_t\in G$. Then observe that
${\hat D\over dt}(u_t-)=\tilde x_t{d\over dt} (\tilde x_t^{-1} u_t-)$.]
However as in \cite{Elworthy-LeJan-Li-book},
$$u_t^{-1}{\hat D\over dt}(u_t-)=u_t^{-1}\breve \nabla_{u_t-} X\circ dB_t+
u_t^{-1}\breve \nabla_{u_t-} A dt.$$

From this the formula for $\alpha(u)$ follows by Remark \ref{re:alpha}(b). For $\beta(u)$ we need to identify the bounded variation part of 
$\int_0^t \varpi(\circ du_t)$. For this write
$$u_t^{-1}\breve \nabla_{u_t-} X\circ dB_t=u_0^{-1} T_{x_0}\xi_t^{-1}
\hat {\parals_t} \circ \hat {\parals_t}^{-1} \breve \nabla_{T\xi_t\circ u_0}X\circ dB_t$$
where $\hat{\parals_t}$ is the parallel translation along $\{\xi_s(x_0): 0\leqslant s\leqslant t\}$ using our semi-connection, which is the adjoint of $\breve \nabla$ by Theorem \ref{th:conn-deri}. As in \cite{Elworthy-LeJan-Li-book}
$$\hat {\parals_t}^{-1} \breve \nabla_{T\xi_t\circ u_0}X\circ dB_t
=\hat {\parals_t}^{-1} \breve \nabla_{T\xi_t u_0}XdB_t -{1\over 2} \hat {\parals_t}^{-1} {\Ric}^{\#}(T\xi_t\circ u_0-)dt$$
while 
$$u_0^{-1}T\xi_t^{-1} \hat{ \parals_t}=u_0^{-1}-\int_0^t u_0^{-1} T\xi_s^{-1} \breve \nabla_{\hat{{/\kern-0.2em/_{\!}}_s}-}X\circ dB_s-\int_0^t u_0^{-1} T\xi_s^{-1} \breve \nabla_{\hat{{/\kern-0.2em/_{\!}}_s}-}A ds$$
giving the formula claimed for $\beta$.
\end{proof}
\subsubsection{Example: Gradient Brownian SDE}\label{ex: gbm-spheres}
An isometric immersion $j:M\to \R^m$ of a Riemannian manifold $M$ determines a stochastic differential equation on $M$:
$$ dx_t=X(x_t)\circ dB_t$$
where $X(x):\R^m\to T_xM$ is the orthogonal projection and $B_.$ is a Brownian motion on $\R^m$.  More precisely
 $$ X(x)(e)=\nabla [y\mapsto\langle j(y), \rangle](x).$$
 It is well known that the solutions of the SDE are Brownian motions on $M$, see \cite{ Elworthy-book},\cite{Rogers-Williams-II}, \cite{Elworthy-Stflour}, and the equation is often called a "gradient Brownian SDE" \index{Gradient Brownian!SDE}%
 . Moreover the LW connection \index{ connection!LW} given by equation (\ref{connection-1}) is the Levi-Civita connection \index{connection!Levi-Civita}%
, (by the classical construction of the latter), see \cite{Elworthy-LeJan-Li-book}. Since the adjoint of the Lev-Civita connection  is itself, Theorem \ref{th:conn-deri}, shows that our connection induced on $GLM$ by the derivative flow \index{flow!derivative} of a gradient Brownian system is also the Levi-Civita connection.
 Almost by definition,
 \begin{equation}
\label{eq:SFF}
\langle \nabla_vX^p,w\rangle _{\R^m}=\langle \mathbf{a}(v,w),e_p\rangle_{\R^m}
\end{equation}
where $\mathbf{ a}:TM\times TM\to \R^m$ is the \emph{ second fundamental form } \index{second fundamental form} of the immersion with \begin{equation}
\label{eq;shape}
\nabla _vX(e)=\mathbf{ A}(v,n_xe) \qquad v\in T_xM,x\in M, e\in \R^m
\end{equation}
for $n_x:\R^m\to T_xM^\perp$ the projection and $\mathbf{A }: TM\oplus T M^\perp\to TM$ the \emph{shape operator} \index{shape operator}%
 given by $$ \langle\mathbf{A}(v,e), w\rangle_{\R^m}=\langle \mathbf{a}(v,w),e_p\rangle_{\R^m}.$$
 Here $T M^\perp$ refers to the normal bundle of $M$ and $T_x M^\perp$ to  the normal space at $x$ to $M$, though we are considering its elements as being in the ambient space $\R^m$. 
 Thus the vertical operator in the decomposition of the generator of the derivative flow  \index{derivative!flow} on $GLM$ for gradient flows \index{flow!gradient} is given by Theorem \ref{th:conn-deri} with 
 $$ \alpha(u)=\frac{1}{2}\sum_{j=1}^{m-n}u^{-1}\mathbf A(u-,l^j)\otimes u^{-1} \mathbf A(u-,l^j)$$
$$\beta(u)=-\frac{1}{2}\sum_{j=1}^{m-n}\mathbf A(\mathbf A(u-,l^j),l^j)-\frac{1}{2}u^{-1}Ric^\# (u-)$$
at a frame $u$ over a point $x$. Here $l^1,...,l^{m-n}$ denotes an orthonormal base for $T_xM^\perp$
 
 For the standard embedding of $S^n$ in $\R^{n+1}$ we have 
 $$\mathbf{a}(u,v)=\langle u,v\rangle x$$
 for $u,v\in T_xS^n$. Also the Ricci curvature \index{curvature!Ricci} is given by $Ric^{\#}(v)=(n-1)v$ for all $v\in TM$. Thus for the standard gradient SDE on $S^n$, at any frame $u$ we have\begin{eqnarray}\label {eq:sphere}
\alpha(u) & = &\frac{1}{2} \Id\otimes \Id \\
\beta (u) & = &-\frac{1}{2}n \;\Id.  
\end{eqnarray}

\section{Associated Vector Bundles \& Generalised Weitzenb\"ock Formulae \;}
\label{se:Weitz}

 As before let $\pi: P\to M$ be a smooth principal $G$-bundle and $\rho: G\to {\Bbb L}(V;V)$ a $C^\infty$ representation of $G$ on some separable Banach space $V$. There is then the (possibly weakly) associated vector bundle
$\pi^\rho: F\to M$ where $F=P\times V/\sim$ for the equivalence relation
given by $(u,e)\sim (ug, \rho(g^{-1})e)$ for $u\in P$, $e\in V$, $g\in G$.
If $[(u,e)]\in F$ denotes the equivalence class of $(u,e)$ we can identify
any $u\in P$ with a linear isomorphism
$$\bar{\bf u}: V\to F_{\pi(u)}$$
by
\begin{equation}
\label{associated-map}
 \bar{\bf u} (e)=[(u,e)].
\end{equation}
Consider the set of smooth maps from $P$ to $V$,  equivariant  by $\rho$:
$$M_\rho(P;V)=\{\hbox{smooth } Z : P\to V, Z(ug)=\rho(g)^{-1}Z(u), \; u\in P, g\in G\}.$$
There is the standard bijective correspondence ${\mathfrak F}^\rho $ between $M_\rho(P,V)$ and $\Gamma(F)$, the space of smooth sections of $F$ defined by
$${\mathfrak F}^\rho(Z)(x)=\bar{\bf u}[Z(u)], \qquad u\in \pi^{-1}(x), Z\in M_\rho(P;V).$$
Via this map, an equivariant diffusion generator  $\B$
 on $P$ induces a differential operator $\B^\rho\equiv  {\mathfrak F}^\rho(\B)$ on $ \Gamma(F)$, of order at most $2$, by
 \begin{equation}
 {\mathfrak F}^\rho( \B)({\mathfrak F}^\rho(Z) )={\mathfrak F}^\rho[\B(Z)], \quad Z\in M_\rho(P;V).
 \end{equation}
Here $\B  $ has been extended trivially to act on $V$-valued
 functions. 
 Note that the definition makes sense since,
$$\B(Z)(ug)=\B  \left(Z\circ R_g\right)(u)
=\B  \left(\rho(g)^{-1}Z\right)(u)
=\rho(g)^{-1}\B  (Z)(u).$$

For such a representation $\rho$ let 
$$\rho_*: \g\to \Lo (V;V)$$
be the induced representation of the Lie algebra $\g$
(the derivative of $\rho$ at the identity). 

\begin{theorem}
\label{th:comp}
When $\B  $ is a vertical equivariant diffusion generator  the induced operator on sections of any associated vector bundle is a zero order operator. With the notation of Theorem \ref{th:op-deco},
 the zero order operator in $\Gamma(F)$ induced by $\B  $ is represented by  $\lambda^\rho: P\to \Lo(V;V)$ for
\begin{equation}
\label{operator-vert}
\lambda^\rho(u)=\rho_*(\beta(u))
+\comp\circ(\rho_*\otimes \rho_*)(\alpha(u)),
\hspace{0.7in} u\in P
\end{equation}
for $\comp: \Lo(V;V) \otimes \Lo(V;V)\to \Lo (V;V)$
the composition map $A\otimes B\mapsto AB$.
\end{theorem}

\begin{proof}
The operator $\B^\rho $ is a zero order operator if
$\F^\rho (\B)(S)(x_0)=\F^\rho(\B)(S')$
 whenever two sections $S$ and $S^\prime$ of $F$ agree at $x_0$.
This holds if  $\displaystyle{\B (fZ)=f\B(Z)}$
for any invariant function $f:P\to R$ and $V$-valued function $Z$ on $P$.
But this holds by Remark \ref{re:vertical}.

For the representation (\ref{operator-vert}), suppose $Z: P\to V$ is equivariant:
$$Z(u\circ g)=\rho(g)^{-1} Z(u), \qquad g\in G.$$
Then
\begin{eqnarray*}
\Lo   _{A_j^*}(Z)(u)
&=&{d\over dt}\, Z( u\cdot  e^{A_jt}))|_{t=0}\\
&=&{d\over dt}\, \rho( e^{-A_jt}) Z(u)|_{t=0}\\
&=&-\rho_*(A_j)Z(u).
\end{eqnarray*}
Iterating we have
$$\B  (Z)(u)=
\sum \alpha^{ij}(u)\rho_*(A_j)\rho_*(A_i)Z(u)
+\sum \beta_k \rho_*(A_k)Z(u)
$$
proving (\ref{operator-vert}).
\end{proof}

From this theorem we easily have the following estimate, which combined with the discussions below, when applied to the associated bundle $\wedge F$ to the orthonormal bundle, shows that the Weitzenb\"ock curvature \index{curvature!Weitzenb\"ock} is positive if the curvature is.
\begin{corollary}
\label{co:ortho-repre}
If $\rho$ is an orthogonal representation, \i.e. $(\rho_*(\alpha))^*=-\rho_*(\alpha)$ for all $\alpha\in \mathfrak {g}$, then $\lambda^{\rho}(v,v)\leqslant 0$ for all $v\in V$. 
\end{corollary}
\begin{proof}
Write $\alpha=\sum_k \mu_k A_k\otimes A_k$ where $\{A_k\}$ is as in Remark \ref{re:alpha}(c).  Then for $v\in F$,
\begin{eqnarray*}
&&\langle \comp\circ(\rho_*\otimes \rho_*)(\alpha(u))(v),v\rangle
=\langle \sum \mu_k  [ \rho_* A_k]^2 (v), v\rangle\\
& =&-\sum \mu_k \langle  \rho_* (A_k)(v),  \rho_* (A_k)(v)\rangle\leqslant 0,
\end{eqnarray*}
since $\mu_k\leqslant 0$. The result follows from (\ref{operator-vert}) since $\rho_*(\beta(u))$ is skew symmetric.
\end{proof}

The situation of Corollary \ref{co:ortho-repre} arises when considering the derivative
flow \index{flow!derivative} for an SDE on a Riemannian manifold whose flow consists of isometries \index{flow!of isometries}; for example canonical SDE's on symmetric spaces as in \cite{Elworthy-LeJan-Li-book}.

Quantitative estimates can be obtained by some representation theory.
For example suppose $G=O(n)$ with $\rho$ the standard representation on $\R^n$. Consider the representation $\wedge^k\rho$ on $\wedge^k \R^n$.

We use the following conventions, as in \cite{Elworthy-LeJan-Li-book}. Let $V$ be an $N$ dimensional real inner product space. 
 For $1\leqslant i\leqslant n$, 
$$a_1\wedge \dots \wedge a_n={1\over n!} \sum_\pi {\rm sgn} \,(\pi) a_{\pi(1)} \otimes \dots \otimes a_{\pi(n)},$$
\begin{equation}
\iota_v (u_1\wedge \dots \wedge u_q)= \sum_{j=1}^q
(-1)^{j+1}\langle v,u_j\rangle
u_1\wedge \dots \wedge \widehat{u_j}\wedge \dots \wedge u_q
\end{equation}
$\langle \otimes a_i, \otimes b_i\rangle=n!\Pi_i\langle a_i,  b_i\rangle$, and $\langle \wedge a_i, \wedge b_i\rangle=\det (\langle a_i, b_j\rangle)$.
Let $\wedge V$ stand for the exterior algebra of $V$ and  $a_j^*$ the ``creation operator"on $\wedge V$ given by $a_j^*v=e_j\wedge v$ for $(e_1,\dots, e_N)$  an orthonormal basis for $\wedge V$. 
 Let $a_j$ be its adjoint, the ``annihilation operator" given by $a_j=\imath_{e_j}$. Note the commutation law:\begin{equation} \label{commutation}
 a_ia^*_j+a^*_ja_i=d_{i j}
\end{equation}

For linear forms 
we have the corresponding operators:  $(a^j)^* \phi(v)=\phi(a_jv)$
and $(a^j\phi)(v)=\phi(a_j^*v)$.
In particular $a^j\phi(v)=\phi(e_j\wedge v)$ and
$(a^j)^*\phi(v)=e_j^*\wedge \phi$.

If $A: V\to V$ is a linear map  on $V$, there are the
 operators $\wedge A$ and $(d\Lambda) (A)$  on $\wedge V$,  which restricted to  $\wedge^p V$ are:
 $$(d\Lambda )(A) \,(u_1\wedge\dots \wedge u_p)  =
\sum_1^p u_1\wedge\dots \wedge u_{j-1}\wedge  Au_j
\wedge  u_{j+1}\wedge\dots \wedge u_p,$$
and also
$$(\wedge A) (u_1\wedge\dots \wedge u_p) =Au_1\wedge\dots \wedge A u_p.$$ 

Note that since $\alpha(u)$ is symmetric, $(\rho_*\otimes \rho_*)\alpha(u): V\otimes V\to V\otimes V$ has
\begin{eqnarray}
(\rho_*\otimes \rho_*) \alpha(u)(v^1\wedge v^2)
& = &\sum_{i,j}\alpha^{ij}(u)\rho_*(A_i)\otimes \rho_*(A_j)(v^1\wedge v^2) \\
&=& \sum_{ij}\alpha^{ij}(u) A_i v^1\wedge A_j v^2.
\label{rho-star}
\end{eqnarray}
and so $(\rho_*\otimes \rho_*)\alpha(u)$ restricts to a map of $\wedge^2 V$ to itself.

\begin{corollary}
\label{co:}
Take the Hilbert-Schmidt inner product on ${\mathfrak so}(n)$ and let
$0\leqslant \mu_{1}(x)\leqslant \dots \leqslant \mu(x)_{{1\over 2}n(n-1)}$
be the eigenvalues of $\alpha$ on the fibre $p^{-1}(x)$, $x\in M$, as described in Remark \ref{re:alpha}(c). Then for all $V\in \wedge^k\R^n$, 
$$-{1\over 2}k(n-k)\mu_{{1\over 2}n(n-1)}(x)
\leqslant \Big\langle \lambda^{\wedge^k}(u)V, V\Big\rangle
\leqslant -{1\over 2}k(n-k)\mu_{1}(x).$$
\end{corollary}
\begin{proof}
Following Humphreys \cite{Humphreys}, \S6.2, consider the bilinear form $\beta$ on ${\mathfrak so}(n)$ given by
$$\beta(A,B)=\trace\Big((d\wedge^k)(A),(d\wedge^k)(B)\Big)
={(n-2)!\over (k-1)!(n-k-1)!} \trace(AB)$$
by a short calculation using elementary matrices.
By Remark \ref{re:alpha}(c) since our inner product on ${\mathfrak so}(n)$ is $\ad(O(n))$-invariant we can write 
$$\alpha(u)=\sum_{l=1}^{{1\over 2}n(n-1)} \mu_l(x)A_l(u)\otimes A_l(u)$$
with $x=p(u)$ and $\{A_l(u)\}_l$ an orthonormal base for ${\mathfrak so}(n)$ at each $u\in P$.

For each $u\in P$, set $$A_l'(u)={(k-1)!(n-k-1)!\over (n-2)!}A_l(u)$$ to ensure $\beta(A_l'(u), A_j(u))=\delta_{lj}$ for each $u$.

Then \begin{eqnarray*}
&&\Big\langle \comp \circ (\rho^{\wedge ^k}_*\otimes \rho^{\wedge^k}_*)(\alpha(u)V, V\Big\rangle= 
\sum \mu_l(x) \Big\langle (d\wedge^k)A_l(u)\circ (d\wedge^k) A_l(u)V,V\Big\rangle \\
&=&\Big[{(k-1)!(n-k-1)!\over (n-2)!}\Big]^{-1} \Big\langle (d\wedge^k)A_l(u)\circ (d\wedge^k) A_l'(u)V,V\Big\rangle \\
&\le&-{(n-2)!\over (k-1)!(n-k-1)!}\big\langle c_{\wedge^k}V, V\big\rangle,
\end{eqnarray*}
where $$c_{\wedge^k}=(d\wedge^k)A_l(u)\circ (d\wedge^k) A_l'(u),$$
the Casimir element of our representation $d\wedge^k$ of ${\mathfrak so}(n)$. Since the representation is irreducible, (for example see \cite {Boerner-book} Theorem 15.1 page 278),

 this element is a scalar, and we have, see Humphreys \cite{Humphreys}
$$c_{\wedge^k}={\dim {\mathfrak so}(n)\over \dim \wedge^k \R^n}
={1\over 2}n(n-1)/{n(n-1)\dots(n-k+1)\over k!}.$$
Thus $\lambda^{\wedge^k}(u)
\leqslant -{1\over 2}k(n-k)\mu_1$. The lower bound follows in the same way.
\end{proof}

When $\B$ has an equivariant H\"ormander form\index{H\"ormander!form} representation the zero order operator $\F^\rho(V)$ can be given  in a simple way by (\ref{F-rho-2}) below. This was noted for the classical Weitzenb\"ock curvature terms using derivative flows in Elworthy \cite{Elworthy-flows}.

\begin{proposition}
Suppose $\B$ lies over a cohesive operator $\A$ and has a smooth H\"ormander form\index{H\"ormander!form}:
$\B  ={1\over 2}\sum \Lo   _{Y^j}\Lo   _{Y^j}
+\sum\beta_k\Lo   _{Y^0}$ with the vector fields $Y^j$, $j=1,\dots, m$, being $G$-invariant. Let $(\eta_t^j)$ be the flow of $Y^j$. For a representation $\rho$ of $G$ with associated vector bundle $\pi^\rho: F\to M$ the zero order operator $\F^\rho(\B^V)$ corresponding to the vertical component of $\B$ is given by
\begin{equation}
\label{F-rho-2}
\F^\rho(\B^V)(x_0)={1\over 2}\sum_{j=1}^m {D^2\over d t^2} \overline{
\eta_t^j(u_0)}\Big|_{t=0} \circ (\bar u_0)^{-1}
+{D\over dt} \overline{
\eta_t^0(u_0)}\Big|_{t=0} \circ (\bar u_0)^{-1}
\end{equation}
for any $u_0\in \pi^{-1}(x_0)$.
\end{proposition}
\begin{proof}
Set $u_t^j=\eta_t^j(u_0) \in P$ and $\sigma(t)=\pi(u_t^j)$ so $\bar u_t^j\in \Lo(V; F_{\sigma(t)})$. From Remark \ref{re:alpha}(b)
$$\alpha(u_0)={1\over 2} \sum_{j=1}^m \varpi (Y^j(u_0))\otimes
\varpi (Y^j(u_0))$$
and so 
$$(\rho_*\otimes \rho_*)\alpha(u_0) 
={1\over 2} \sum_{j=1}^m (\bar u_0)^{-1} {D\over dt} \bar u_t^j\Big|_{t=0} \otimes (\bar u_0)^{-1} {D\over dt} \bar u_t^j\Big|_{t=0}$$
as in the proof of Theorem \ref{th:conn-deri}.

Also from equation (\ref{beta-2})
$$\beta(u_0)={1\over 2} \sum_{j=1}^m \Lo_{Y^j} \big(\varpi (Y^j(-)\Big)(u_0)+ {1\over 2}  \big(\varpi (Y^0(-)\Big)(u_0).$$
Let $(\parals_t)$ denotes parallel translation in $F$ along $\sigma$. Then
\begin{eqnarray*}
\rho_* \Lo_{Y^j} \big(\varpi (Y^j(-)\Big)(u_0)
&=&{d\over dt} \rho_* \varpi \Big(Y^j(u_t^j)\Big)\Big|_{t=0}\\
&=& {d\over dt} \Big( (\bar u_t^j)^{-1} {D\over dt} \bar u_t^j\Big)\Big|_{t=0}\\
&=&{d\over dt} \Big( (\parals_t^{-1} \overline{ u_t^j})^{-1} \parals_t^{-1} {D\over dt} \overline{  u^j_t}\Big)\Big|_{t=0}\\
&=& -\bar u_0^{-1} {D\over dt} \overline{  u_t^j}\Big|_{t=0} 
\circ \bar u_0^{-1} {D\over dt} \overline{ u_t^j}\Big|_{t=0}
+{\bar u_0}^{-1} {D^2\over dt^2} \overline{ u_t^j}\Big|_{t=0}
\end{eqnarray*}
leading to the required result via Theorem \ref{th:comp}.
\end{proof}

To examine particular examples we will need to have detailed information
about the zero order operators determined by a vertical diffusion generator.
For this suppose $\B$ is vertical and given by
$$\B  ={1\over 2}\sum \alpha^{ij}\Lo   _{A_i^*}\Lo   _{A_j^*}
+\sum\beta_k\Lo   _{A_k^*}$$
for $\alpha: P\to \g\otimes \g$ and
$\beta: P\to \g$ as in Theorem \ref{th:op-deco}
and (\ref{alpha}).

Motivated by the Weitzenb\"ock formula for the Hodge-Kodaira Laplacian on differential forms, see Corollary \ref{co:derivative} below,  \cite{Rosenberg}, \cite{CFKS87}, we shall examine in more detail the case of the exterior power $\wedge \rho: G\to \LL(\wedge V; \wedge V)$ of a fixed representation $\rho$ showing that $\lambda^{\wedge \rho}$ has expressions in terms of annihilation and creation operators which are structurally the same as these of the Weitzenb\"ock curvature \index{curvature!Weitzenb\"ock} (which are shown to be a special case in Corollary \ref{co:derivative}).

\begin{lemma}
\label{le:Weitzenbock}
If $\B$ is a vertical operator on $P$ and $(e_i, i=1,2,\dots, N)$ is an orthonormal basis of $V$, the zero order operator  on the associated bundle $\wedge F\to M$ is represented by $\lambda ^{\wedge \rho}: P\to \LL(\wedge^pV; \wedge^p V)$ with
\begin{eqnarray*}
\lambda ^{\wedge \rho}(u)
&=& {1\over 2} \sum_{i,j,k,l=1} ^N
 \left \langle \left( (\rho_* \otimes \rho_* )\alpha(u)\right)(e_j\otimes e_l), e_i \otimes e_k \right \rangle  a_i^* a_j  a_k^* a_l \\
 &&+\sum_{i,j=1} ^N \langle ( \rho_* \beta(u))e_j, e_i\rangle a_i^* a_j, \qquad u\in P
\end{eqnarray*}
\end{lemma}

\begin{proof} 
Recall that if $A\in \LL(V;V)$ then
\begin{equation} 
\label{tensor-op}
d\Lambda(A)=\sum_{i,j=1} ^N\langle  Ae_j, e_i\rangle a_i^* a_j,
\end{equation}
\eg see Cycon-Froese-Kirsch-Simon \cite{CFKS87}.
Consequently
\begin{equation}
\label{tensor-op-2}
d \Lambda (\rho_*  \beta(u))
=\sum_{i,j=1} ^N \langle \rho_*  \beta(u)e_j, e_i\rangle a_i^* a_j
\end{equation}

On the other hand by Theorem \ref{th:op-deco} and  (\ref{alpha}),  we can represent $\alpha$ as:
$$\alpha(u)=\sum_{n,m }  a_{n,m}(u)A_n\otimes A_m$$
where $\{A_i\}_{i=1}^N$ is a basis of $\g$. So
\begin{eqnarray*}
&& \comp \circ (\wedge\rho_*\otimes \wedge\rho_*)(\alpha(u))\\
&=&\comp\circ \sum_{m,n} a_{n,m}(u) \, d\Lambda (\rho_* A_m)\otimes d\Lambda(\rho_* A_n)\\\
&=&\sum_{m,n} a_{n, m}(u) d\Lambda (\rho_* A_m)\circ d\Lambda (\rho_* A_n)\\
&=&\sum_{m,n}a_{n,m}(u)\sum_{i,j,k,l=1}^N\langle \rho_* A_m e_j, e_i \rangle \langle \rho_* A_n e_l, e_k \rangle  a_i^* a_j  a_k^* a_l  \\
&=&{1\over 2}\sum_{m,n}a_{n,m}(u) \sum_{i,j,k,l=1}^N\left \langle \left(\rho_* A_m \otimes \rho_* A_n \right)(e_j\otimes e_l), e_i \otimes e_k \right \rangle  a_i^* a_j  a_k^* a_l  \\
&=&{1\over 2}\sum_{i,j,k,l=1}^N\left \langle (\rho_*\otimes \rho_*) \alpha(u)(e_j\otimes e_l), e_i \otimes e_k \right \rangle  a_i^* a_j  a_k^* a_l,
\end{eqnarray*}
since our convention for the inner product on tensor products gives $$ \langle u_1\otimes v_1,u_2 \otimes v_2\rangle =2\langle u_1,u_2\rangle\langle v_1,v_2 \rangle.$$
The desired conclusion follows. 
\end{proof}

\begin{theorem}
\label{th:Weit}
 Let $R(u):\wedge^2V\to \wedge^2V$ be the restriction of $(\rho_*\otimes \rho_*)\alpha(u): V\otimes V\to V\otimes V$, then
\begin{eqnarray*}
&&\lambda^{\wedge \rho}(u)
=-\sum_{i<k, j<l} \left \langle R(u)(e_j\wedge e_l), e_i \wedge e_k \right \rangle  a_i^*  a_k^* a_j a_l\\ &&
+{1\over 2}\sum_{i,j,l=1}^N\left \langle (\rho_*\otimes \rho_*)\alpha(u)(e_j\otimes e_l), e_i \otimes e_j \right \rangle  a_i^* a_l
+\sum_{i,j} \left \langle \rho_* \beta(u)e_j, e_i\right\rangle (a_i)^*a_j.
\end{eqnarray*}
This can be rewritten as:
\begin{equation}
\label{Weit-1}
\lambda^{\wedge \rho}(u)
=-\sum_{i<k, j<l} \left \langle R(u)(e_j\wedge e_l), e_i \wedge e_k \right \rangle  a_i^*  a_k^* a_j a_l +{1\over 2}d\wedge (Z^\rho(u) ) + d\wedge( \rho_* \beta(u)).
\end{equation}
where $Z^\rho(u)\in \LL(V;V)$ is defined by
$$\langle Z^\rho(v_1), v_2\rangle
=\sum_{j=1}^N \big\langle (\rho_*\otimes \rho_*)(\alpha(u))(e_j\otimes v_1),
v_2\otimes e_j\big \rangle _{V\otimes V}.$$
\end{theorem}
\begin{proof}
This follows from Lemma \ref{le:Weitzenbock} since
\begin{eqnarray*}
&&{1\over 2} \sum_{i,j,k,l=1}^N\left \langle (\rho_*\otimes \rho_*)\alpha(u)(e_j\otimes e_l), e_i \otimes e_k \right \rangle  a_i^* a_j  a_k^* a_l\\
=&&-{1\over 2} \sum_{i,j,k,l=1}^N\left \langle (\rho_*\otimes \rho_*)\alpha(u)(e_j\otimes e_l), e_i \otimes e_k \right \rangle  a_i^* a_k^*a_j   a_l\\
&&+{1\over 2} \sum_{i,j,l=1}^N\left \langle (\rho_*\otimes \rho_*)\alpha(u)(e_j\otimes e_l), e_i \otimes e_j \right \rangle  a_i^* a_l\\
=&&-\sum_{j<l; i<k}^N\left \langle R(u)(e_j\wedge e_l), e_i \wedge e_k \right \rangle  a_i^*  a_k^* a_j a_l \\
&&+{1\over 2} \sum_{i,j,l=1}^N\left \langle (\rho_*\otimes \rho_*) \alpha(u)(e_j\otimes e_l), e_i \otimes e_j \right \rangle  a_i^* a_l.
\end{eqnarray*}
\end{proof}

\begin{remark}
\label{re:Weit}
\begin{enumerate}
\item[(a)] Note that the second term in (\ref{Weit-1}) in general depends on
the symmetric part of $(\rho_*\otimes \rho_*)(\alpha(u))$ as well as on $R$.
\item[(b)] If we write 
$$\alpha(u)=\sum \mu_k(u) A_k(u)\otimes A_k(u)$$
as in Remark \ref{re:alpha}(c),
Then $Z^\rho(u)$ in (\ref{Weit-1}) has
$$Z^\rho(u)=2\sum_k \mu_k(u) \Big(\rho_*(A_k(u))\rho_*(A_k(u))\Big).$$
\end{enumerate}
\end{remark}

\begin{corollary}
\label{co:derivative}
For the derivative process in $GLM$ of a cohesive generator $\A$ given in H\"{o}rmander form without a drift, the zero order operator induced by the vertical diffusion on the exterior bundles $\wedge TM$  is the generalized  Weitzenb\"ock curvature \index{curvature!Weitzenb\"ock}  given by:
\begin{equation}
-{1\over 2} d\wedge^q ({  \Ric}^{\#})(V) - \sum_{1\leqslant i\leqslant k\leqslant n\atop1\leqslant j<l\leqslant p}    R_{ikjl}a_l^*a_j^*a_ka_i V
\end{equation}
for all $V\in \wedge^q TM$.
Here $   R_{ikjl}=\left\langle   R(e_i,e_k)e_l,e_j\right\rangle, 1\leqslant i,k\leqslant n, 1\leqslant j, l \leqslant p$
for $R$ the curvature tensor of the associated connection.\end{corollary}

\begin{proof} 
By Theorem \ref{th:conn-deri},
$$\alpha(u)={1\over 2}\sum \left(u^{-1}    \nabla_{u(-)}X^p\right)
\otimes \left(u^{-1}\nabla_{u(-)} X^p\right), \qquad u\in GLM.$$
By Corollary C.5 in \cite{Elworthy-LeJan-Li-book} the restriction of $\alpha$ to anti-symmetric tensors is 
${1\over 2}\CR$.

By the relation between the curvature tensor and the curvature oeprator:
$$\langle    \CR(u\wedge v), w\wedge z\rangle =\langle     R(u, v)z,w\rangle,$$
 the first term in $\lambda^\rho(u)$ of Lemma \ref{le:Weitzenbock}  is:
$$-2\sum_{i<k,j<l}^N\left \langle    \CR(u)(e_j\wedge e_l), e_i \wedge e_k \right \rangle  a_i^*  a_k^* a_j a_l=-2\sum_{i<k, j<l}^N     R_{jlki}a_i^*a_k^*a_ja_l.  $$
By (ii) of Remark \ref{re:Weit}, the second term is
$${1\over 2}d\wedge ^q \left(\sum_{p=1}^mu^{-1}    \nabla_{   \nabla_{u(-)}X^p} X^p\right).$$
The required result follows since
$$\beta(u)=-{1\over 2}\sum_{p=1}^m u^{-1} \left(   \nabla_{   \nabla_{u(-)}X^p}X^p\right)
- {1\over 2}u^{-1} \left ({\Ric }^{\#}{u(-)}\right).$$
\end{proof}

Corollary \ref{co:derivative} reflects the results in \cite{Elworthy-LeJan-Li-book}, Theorem 2.4.2, concerning Weitzenb\"ock formula for H\"ormander form\index{H\"ormander!form} operators on differential forms. In particular it gives another approach to the result that when $\breve \nabla$ is the Levi-Civita connection, as holds for gradient stochastic differential equations, the generator induced on differential forms by the derivative process is the Hodge-Kodaira Laplacian up to a first order term.

Note that if $\B$ is the operator on $GLM$ determined  by the H\"{o}rmander form (\ref{op-deri}) of $\A$ then for a representation $\rho: GL(M)\to \Lo(V;V)$ with associated $\pi^\rho: GL(n)\to \Lo(V;V)$ 
the induced operator $\F^\rho(\B)$ on sections of $\pi^\rho$ is also given by the `H\"{o}rmander form' 
${1\over 2}\sum_j \Lo_{X^j}\Lo_{X^j}+\Lo_A$, where for any $C^1$ vector field $Y$ on $M$ and any $C^1$ section $U$ of $\pi^\rho$ the Lie derivative $\Lo_YU \in \Gamma F$ is given by
$$(\Lo_YU)(x)
=\bar{\mathbf u} {d\over dt}
\Big(\overline {{\mathbf {T\eta_t^Y}} \circ {\mathbf u}}\Big)^{-1}
U\Big(\eta_t^Y(x)\Big)\Big|_{t=0}$$ for $x\in M$, $u$ a frame at $x$, and $(\eta_t^Y)$ the flow of $Y$, using the notation of (\ref{associated-map}). 

Indeed by (\ref{associated-map}), for $Z(u)=\bar{\mathbf u} U(\pi(u))$, so $U=\F^\rho(Z)$,
$$\F^\rho(\B)(U)
=\F^\rho\Big[\Big( {1\over 2}\sum_j \Lo_{(X^j)^{GL}}\Lo_{(X^j)^{GL}}+\Lo_{A^{GL}}\Big) (Z)\Big]$$
while 
$\Lo_{(X^j)^{GL}}(Z)(u)={d\over dt} Z(T\eta_t^{X^j}\circ u)\Big|_{t=0}$ so that
$$\F^\rho\Big[ \Lo_{(X^j)^{GL}}(Z)\Big](x)
=\bar{\mathbf u}{d\over dt} Z(T\eta_t^{X^j}\circ u)\Big|_{t=0}
=\Lo_{X^j} (U)(x).$$
This representation of $\F^\rho(\B)$ was noted in the case of the operator induced on differential forms by a stochastic flow\ index{flow!on differential forms} in \cite{Elworthy-LeJan-Li-book}, and for the case of the Hodge-Kodaira Laplacian in Elworthy \cite{Elworthy-flows}.

\begin{example}
Let $P$ be the orthonormal frame bundle for a Riemannian metric on $M$. Let $C: \R^n\to \R^n$ be a symmetric map,  define 
$$\alpha =\sum_{i,j}\trace \langle C(A_i -), A_j-\rangle A_i\otimes A_j,$$
where $\{A_i=\sqrt 2 e_i\wedge e_j\}$ is an orthonormal basis of $\mathfrak so(n)$.   Then  
$$\comp \circ \alpha =-{1\over 4}(\trace C)\id+{1\over 4}(2-n)C.$$
Let ${\Ric}^{\#}: TM\to TM$ be the Ricci curvature \index{curvature!Ricci}(for the Levi-Civita connection, say). When applied to $C(u)=u\Ric_{\pi(u)}^{\#}(u^{-1}-)$ for $u\in P$ with $\Ric$ positive, we see it defines a vertical operator on the orthonormal frame bundle with coefficients $\alpha$ as given above, $\beta=0$. Its associated 
zero order term on vertical fields  is then ${1\over 4}(2-n)\Ric_{\pi(u)}^{\#}-{1\over 4}k$, where $k$ is the scalar curvature.
\end{example}

\begin{proof}
First observe that $$\alpha ={1\over 2} \left(d\otimes^2 C\right)\left( \sum_i A_i\otimes A_i\right).$$
Then we use the elementary fact about elementary matrices $\{E_{ij}\}$:
$$E_{ij}CE_{i'j'}=C_{ji'}E_{ij'}$$
and take the basis of $\mathfrak g$ to be $\{\sqrt 2 e_i\wedge e_j, i<j\}$. Recall  that
$e_i\wedge e_j=1/2(E_{ij}-E_{ji})$.
\end{proof}

\begin{remark}  We have seen  in Corollary \ref{co:derivative} that there is zero order operator 
on the associated bundle $\wedge F\to M$ represented by the Weitzenb\"ock curvature\index{curvature!Weitzenb\"ock}  of a given connection. On the other hand given a curvature operator \index{curvature!operator} $\CR$ of a metric connection, or more generally an operator which has the same symmetry properties as a curvature tensor, is there a canonical vertical diffusion operator on $GLM$ which induces zero order operators on differential forms which have the form of the Weitzenb\"ock curvatures of R? A vertical operator with such a zero order term always exist since we can take $\CR$ in a diagonal form:
\begin{equation}
\label{diagonal-form}
\CR(u)=\sum_{n=1}^N  A_n(u)\wedge A_n(u),
\end{equation}
for some $A_n: GLM\to {\mathfrak gl}(n)$ which are $\ad(G)$-invariant, e.g. by taking an isometric embedding (\eg see \cite{Elworthy-LeJan-Li-book}.  In this case let $(e^j)$ be a basis of $E_{\pi(u)}$ define
\begin{equation}\label{symbols-5}
\begin{array}{lll}
\alpha(u)&=& {1\over 2}\sum_{n=1}^N  A_n(u)\otimes A_n(u),\\
\beta(u)&=&-{1\over 2}\sum_{n=1}^N  (A_n(u))^2-{1\over 2}
 \sum_{j=1}^p R(-, e^j)e^j,
\end{array}
\end{equation}
see Remark \ref{re:Weit}(b). Then $\alpha$ is positive and we can define an operator with its coefficients $\alpha$ and $\beta$ given as above. 

For a discussion of the representation of $\CR$ in the form of (\ref{diagonal-form}) see Kobayashi-Nomizu \cite{Kobayashi-Nomizu-II} (Notes 17 and 18). In particular there is a discussion there of the number $N$ required and of a rigidity theorem originating from Chern, See also Berger-Bryant-Griffiths \cite{Berger-Bryant-Griffiths}.

When $M$ is Riemannian with positive semi-definite curvature operator $\CR: \wedge^2TM \to \wedge^2 TM$ there is a canonical construction. For this take the orthonormal  frame bundle $\pi: OM\to M$, with $G=O(n)$. We will use the isomorphism of $\wedge^2\R^n$ with ${\mathfrak so}(n)$ under which $e_p\wedge e_q$ corresponds to ${1\over 2} (E_{[p,q]}-E_{[q,p]})$ for $e_1,\dots, e_n$ a fixed basis of $\CR^n$ and $E_{[p,q]}$ the elementary matrix so $E_{[p,q]}(v)=v_q e_p$. Set 
$A_{[p,q]}={1\over \sqrt 2}[E_{[p,q]}-E_{[q,p]}]$ so $\{A_{[p,q]}: 1\leqslant p<q\leqslant n\}$ forms an orthonormal basis for ${\mathfrak so}(n)$. Define
$$\alpha: OM\to {\mathfrak so}(n)\times  {\mathfrak so}(n)$$
by 
$$\alpha(u)=\sum_{1\leqslant p\leqslant q\leqslant n, 1\leqslant p'\leqslant q'\leqslant n}
\Big\langle \CR(\wedge^2(u)(e_p\wedge e_q)), \wedge^2(u)(e_{p'}\wedge e_{q'})) \Big\rangle_{\pi(u)} A_{[p,q]}\otimes A_{[p',q']}.$$
Our representation $\rho$ is just the identity map and, by (\ref{rho-star})
and Bianchi's identity, the restriction of $\alpha(u): \R^n\otimes \R^n\to \R^n\otimes \R^n$ to $\wedge^2\R^n$ is just $\CR$ itself. In the notation of (\ref{Weit-1}) we see 
$$\langle Z^\rho(v^1), v^2\rangle
=-4 \Ric(v^1, v^2).$$
If we take $\beta=0$, we obtain from (\ref{Weit-1}) that 
$$\lambda^{\wedge^\rho}(u)=
-\sum_{i<k, j<l} R_{jlik} a_i^* a_k^* a_j a_l -2\; (d\wedge){\Ric}^{\#}.$$
To get the full Weitzenb\"ock term, extend $\alpha$ over GLM by equivariance and define $\beta(u)$, for $u\in GLM$,  by
$\beta(u)={3\over 2}u^{-1} {\Ric}^{\#}(u-)$ as in (\ref{symbols-5}).
\end{remark}


\chapter{Projectible Diffusion Processes}
\label{ch:intertwined}

Let $M^+$ be the Alexandrov one point compactification of a smooth manifold $M$. Consider the space $\C_{y_0}M^+$ of processes $(y_t)$ with life time $\zeta$ on $N^+$ such that $t\to y_t$ is continuous with $y_t=\Delta$ when $t\geqslant \zeta$. Let $\Lo$ be a diffusion operator on $M$ and let $\{\PP_{y_0}, y_0\in M^+\}$ be the family of $\Lo$-diffusion measures \index{diffusion!measure} in the sense of \cite{Ikeda-Watanabe}, \ie the solution to the martingale problem on $\C(M^+)$ so the canonical process $(y_t, 0\leqslant t<\zeta)$ with the system of diffusion measures $\{\PP^\Lo_{y_0}, y_0\in N^+\}$ is a strong Markov process on $M^+$. 
 Denote by $\E$ mathematical expectation with respect to the measure $\PP_{y_0}$.  
We may add to these notations the relevant subscripts or superscripts indicating the diffusion operator or the Markov process concerned, \eg  $\{\PP^\Lo_{y_0}\}, \zeta^\Lo$, $\E^{\Lo, y_0}$ or even $\E^{y_0}$.

  For $y_0\in M$ and $f\in \C_c^\infty M$, the space of smooth functions on $M$ with compact support, let
\begin{equation}
M_t^{df}:=M_t^{df,\Lo}:= f(y_{t\wedge \zeta})-f(y_0)-\int_0^{t\wedge \zeta} \Lo f(y_s)ds\end{equation}
Then $(M_t^{df}: 0\leqslant t<\infty)$ is a martingale on the probability space $(\C(M),\PP^\Lo_{y_0})$ with respect to the $\{\F_t^{y_0}\}$, where $\F_t^{y_0}=\sigma\{y_s; 0\leqslant s\leqslant t\}$. Moreover it has bracket
$$\langle M^{df}\rangle_t=2\int_0^{t\wedge \zeta} \sigma^\Lo((df)_{y_s}, (df)_{y_s})ds.$$
This definition extends to the case of $C^2$ functions $f$ but then $M_t^{df}$ is only defined for $0\leqslant t<\zeta^\Lo$ and is a local martingale.

\section{Integration of predictable processes}\label{se:int-pred-proc}
%


\begin{proposition}
\label{pr:mart-1}
Let $\tau$ be a stopping time with $\tau <\zeta$ and let $\{\alpha_t: 0\leqslant t<\tau\}$ be a $\F_*^{y_0}$ predictable process in $T^*M$  
 such that $\alpha_t\in T^*_{y_t}M$ for each $t\in[0,\tau)$, and for each compact subset of $M$ we have
$$ \int_0^{\tau}\chi_ K(y_s) \alpha_s (\sigma^\Lo \alpha_s )\; ds<\infty$$
almost surely.
  
Then there is a unique local martingale 
$\{M_t^\alpha~:~0\le~t<~\tau\}$ such that for all $f\in \C_c^\infty M$,
\begin{equation}
\label{bracket}
\left\langle M^\alpha, M^{df}\right\rangle_t =2\int_0^{t} \sigma^\Lo \left(\alpha_s, (df)_{y_s}\right)ds,   \qquad  t<\zeta.
\end{equation}
\end{proposition}

\begin{proof}
We can write 
\begin{equation}
\label{rep-alpha}
\alpha_t=\sum_{j=1}^m g_t^j\cdot {df_j}(y_t),
\end{equation}
where the functions $g^j$ are predictable real valued processes, \eg by taking $(f_1,\dots, f_m): M\to \R^m$ to be an embedding and
$g_t^j=\alpha_t\circ X^j$, for $X(x)=\sum_{i=1}^m  X^i(x)e_i$ the projection from $\R^m$ to $T_xM$. Using a partition of unity, at the cost of having an infinite, but locally finite sum, we can assume that the $f_j$ in the representation are all in $\C_c^\infty M$.
Define
\begin{equation}
M_t^\alpha:=\sum_j \int_0^t g_s^j dM_s^{df_j}.\end{equation}
Clearly (\ref{bracket}) holds. 
For uniqueness suppose $K$ is a local martingale orthogonal to $M^{df}$ for all $f\in C_c^\infty M$.

Then $K$ 
vanishes since the martingale problem for $\Lo$ is well posed by an argument attributed to Dellach\'erie (see Rogers-Williams \cite{Rogers-Williams-II}, the end of the proof of theorem 2.5.1). In fact it it were not zero we could take a suitable stopping time $\tau$ to ensure $(1+K^0_{\tau\wedge t})\PP_{x_0}^\Lo$ solves the martingale problem up to time $t$ since
$$K^0_{\tau\wedge t} M_s^{df} \equiv K_{\tau\wedge t}^0 \left(f(x_s)-f(x_0)-\int_0^s \Lo f(x_s) ds\right), \qquad 0\leqslant s\leqslant t$$
is a uniformly integrable martingale.
\end{proof}



We will often write
\begin{equation}
M_t^\alpha=\int_0^t \alpha_s \, d\{y_s\}
\end{equation}
bringing out the fact it is the martingale part of the Stratonovitch integral 
$\int_0^t \alpha_s\circ dy_s$ of $(\alpha_t)$ along the diffusion process $(y_t)$ when that integral is defined \eg when $(\alpha_t)$ is a continuous semi-martingale. Indeed

\begin{lemma}
\label{le:mart-1}
Let $\alpha$ be a $C^2$ 1-form then
\begin{equation}
M_t^\alpha=\int_0^t \alpha_{y_s} \circ dy_s -\int_0^t  \big(\delta^\Lo \alpha\big) (y_s) ds,
\qquad 0\leqslant t<\zeta.
\end{equation}
\end{lemma}
\begin{proof}
This is clear for an exact 1-form. Suppose $\lambda: M\to \R$ is $C^2$ and $\alpha$ is exact, then for $t<\zeta$,
\begin{eqnarray*}
M_t^{\lambda \alpha}
&=&\int_0^t \lambda(y_s) dM_s^\alpha
=\int_0^t \lambda(y_s)\circ dM_s^\alpha-{1\over 2} \left\langle \int_0^\cdot d\lambda(y_s)dy_s, M^\alpha_\cdot \right\rangle_t\\
&=& \int_0^t \lambda(y_s) \,\alpha_{y_s}\circ dy_s -\int_0^t \lambda(y_s)  \big(\delta^\Lo \alpha\big) (y_s) ds
-{1\over 2} \langle M^{d\lambda}_\cdot, M^\alpha_\cdot\rangle_t\\
&=& \int_0^t \lambda(y_s)\,\alpha_{y_s}\circ dy_s -\int_0^t   \delta^\Lo (\lambda \alpha) (y_s) ds
\end{eqnarray*}
since $M^{d\lambda}$ is the martingale part of $\lambda(y_s)$ and 
$$\langle M^{d\lambda}, M^\alpha\rangle_t=2\int_0^t\sigma^\Lo (d\lambda_s, \alpha_s)ds.$$ This proves the result for general $\alpha$ by taking a suitable representation.
\end{proof}

Let $S_x$ be the image of $\sigma_x^\Lo$ in $T_xM$ and let $S:=\cup_x S_x$. By a predictable $S^*$-valued process $(\alpha_t)$ over $(y_t: 0\leqslant t<\zeta)$ we mean a process $(\alpha_t: 0\leqslant t)$ such that
\begin{enumerate}
\item [(i)] $\alpha_t\in S_{y_t}^*$ for all $0\leqslant t<\zeta$
\item[(ii)] $(\alpha_t\circ \sigma_{y_t}^\Lo,  0\leqslant t<\zeta)$ is a predictable process in $TM$, canonically identified with $T^{**}M$.
\end{enumerate}
Note that condition (ii) is equivalent to
\begin{enumerate}
\item[(ii)'] there exists a predictable $(\bar \alpha_t)$ in $T^*M$ over $(y_t)$ such that $\bar \alpha_t|_{S_{y_t}}=\alpha_t$ for all $0\leqslant t<\zeta$.
\end{enumerate}
That (ii') implies (ii) is immediate. To see (ii) implies (ii') first note that $\alpha_t\circ \sigma_{y_t}^\Lo\in S_{y_t}$ for each $t$ since $\alpha_t\circ \sigma^\Lo_{y_t}=\sigma_{y_t}^\Lo(\tilde \alpha_t)$ for any extension $\tilde \alpha_t$ of $\alpha_t$ to $T_{y_t}^*M$. We can then choose a measurable selection $\bar \alpha_t$ in $T_{y_t}^*M$ with $\sigma_{y_t}^\Lo(\bar \alpha_t)=\alpha_t\circ \sigma_{y_t}^\Lo$. This process $\bar \alpha_t$ will satisfy the requirements of (ii') since
\begin{equation}
\label{selection}
\bar \alpha_t\sigma^{\Lo}_{y_t}=\sigma^{\Lo}_{y_t}\bar \alpha_t
=\alpha_t\sigma^{\Lo}_{y_t}.
\end{equation}

In fact (\ref{selection}) is a reflection of the fact that $\sigma^\Lo_y$ extends to a linear isomorphism $\sigma^\Lo_y: S_y^*\to S_y$ canonically. In particular $\sigma^\Lo_{y_t}(\alpha_t)$ is well defined. 
\begin{definition}
If $(\alpha_t)$ satisfies (i) and (ii) we will say it  is in $L_\Lo^2$ if $$\int_0^t\alpha_s\sigma_{y_s}^\Lo(\alpha_s)ds~<~\infty$$
for all $t\geqslant 0$,
and will say it is in $L^2_{\Lo,loc}$ if for any compact subset $K$ of $M$ 
$$\E \int_0^{t\wedge \zeta}\chi_ K(y_s)\alpha_s (\sigma^\Lo \alpha_s )\; ds<\infty$$
for all $t\geqslant 0$.
\end{definition}
\begin{remark}\label{re:Pchange}
Suppose the processes associated to diffusion operators $\Lo$ and $\Lo+\LL_b$ are  both non-explosive, where $b$ is a locally bounded measurable vector field on $M$. Assume that there exists a $T^*M$- valued process $b^\#_.$ defined on the canonical probabilty space $\C_{y_0}M$ such that $\PP^{\Lo}$-almost surely:
\begin {enumerate}
\item $2\sigma^{\Lo}(b^\#_s)=b(y_s)$ 
\item $ \int_0^tb^\#_s\sigma^{\Lo}(b^\#_s)ds<\infty$ 
\end{enumerate}
Then, by the GMCM-theorem, as in the Appendix section \ref{se-GMCM theorem}, we have  on $\C([0,T];M)$, 
$$\PP^{\Lo+\LL_b}=Z_t\PP^{\Lo}$$
where         $ Z_t= \exp\{M^{b^\#}_t-\int_0^tb^\#_s\sigma^{\Lo}(b^\#_s)ds\}.$
In an obvious notation, for suitable $\alpha$, as canonical processes we have, almost surely, 
$$\int_0^t\alpha_s d\{y_s\}^{\Lo}=\int_0^t\alpha_s d\{y_s\}^{\Lo+\LL_b}-\int_0^t\alpha(b(u_s))ds.$$
\end{remark}
\begin{lemma}
\label{le:mart-2}
Suppose $\sigma^\Lo$ has image in a subset $S$ of $TM$. Then $(M^\alpha_t)$ depends only on the restriction of $\alpha_s$ in $\Lo(T_{y_s}M; \R)$ to $S_{y_s}$, $0\leqslant s<\zeta$. In particular (\ref{bracket}) defines uniquely a local martingale for each predictable $S^*$-valued process $(\alpha_t)$ over $(y_t)$for which the right hand side of (\ref{bracket}) is always finite almost surely.
\end{lemma}
\begin{proof}
For $T^*M$-valued $\F_*^{y_0}$ predictable processes $(\alpha^1_t, 0\leqslant t<\zeta)$ and $(\alpha^2_t, 0\leqslant t<\zeta)$ over $(y_t, 0\leqslant t<\zeta)$ which agree on $S$ we see $$\langle M^{\alpha^1}-M^{\alpha^2}, M^{df} \rangle_t
=2\int_0^{t\wedge \zeta}  \sigma (\alpha_s^1-\alpha_s^2, (df)_{y_s}) \,ds=0$$
for all $f\in \C_c^\infty M$. Therefore $M^{\alpha^1}=M^{\alpha^2}$.
On the other hand this also shows that if $\alpha_s\in S_{y_s}^*$ for all $s$, we can use condition (ii)' above to choose a predictable process $\{\bar \alpha_s: 0\leqslant s<\zeta\}$ with values in $T^*M$ over $(y_t)$ and set $M_\cdot^\alpha=M^{\bar \alpha}_\cdot$ without ambiguity.
\end{proof}

\begin {example} {\it Canonical Brownian motion associated to a cohesive diffusion.}\label{ex:BM} For simplicity assume that our $\Lo$-diffusion from a given point $y_0$ is non-explosive.
If $\Lo$ is cohesive with sub-bundle $E$ of TM, take a metric connection $\Gamma$ for $E$, using the metric determined by $2\sigma^{\Lo}$. Let $$\alpha_s(\sigma):=(\paral^\sigma_s)^{-1}: E_{\sigma(s)}\to E_{y_0}$$ be the inverse of parallel translation, $\paral^\sigma_s$,   along $\sigma$ from $E_{\sigma(0)}$ to $E_{\sigma(s)}$, for $\PP^{y_0}$ almost all paths  
$\sigma$ in $M$. Each component of this with respect to an orthonormal basis for $E_{y_0}$ clearly  lies in $L^2_{\Lo}$. With the obvious extension of our notation to the vector space valued  case define an $E_{y_0}$-valued process $B_t: t\geqslant0$ by $$B_t=M^\alpha_t=\int_0^t(\paral)^{-1}d\{y_s\}.$$
It is easy to check from its quadratic variation that it is a Brownian motion on the inner product space $E_{y_0}$. Moreover (as described in \cite{Elworthy-LeJan-Li-book}) it has the same filtration as the canonical process on $\C_{y_0}M$ up to sets of measure zero.  It is the martingale part of the stochastic anti-development $\int_0^t(\paral_s)^{-1}dy_s$of our $\Lo$-diffusion from $y_0$. The use of a different metric connection would  change it by a random rotation, so  this process is  defined on the canonical probability space $\big\{\C_{y_0}M, \F^{y_0}, \PP^{y_0}\big\}$ and up to such rotations depends only on it.
We have,  for $\alpha$ as usual:
  \begin{equation}\label{eq:mart-6}
\int_0^t\alpha_sd\{y_s\}=\int_0^t\left(\alpha_s \circ \paral_s \right)dB_s.
\end{equation}
Using the definitions in the Appendix \ref{se:semi-mart} we see that if our diffusion process $y_.$ is a $\Gamma$-martingale then \begin{equation}
\int_0^t\alpha_sd\{y_s\}=\big(\Gamma\big)\int_0^t\alpha_s dy_s.
\end{equation}
Note that there is always some metric connection $\Gamma$ on $E$ for which a cohesive diffusion process is a $\Gamma$-martingale, by section 2.1 of \cite{Elworthy-LeJan-Li-book}.
\end {example}

\section{Horizontality and filtrations}

  We can characterise horizontality of a diffusion operator or process in terms of filtrations using the following lemma:
\begin{lemma}
\label{le:cohesive}
Suppose $p:N\to M$ is a smooth map, $\B$ a smooth diffusion operator over a smooth diffusion operator $\A$, and also 

\begin{enumerate}
\item [(i)] $\sigma^\A$ and $\sigma^\B$ have constant rank and 
\item[(ii)] the filtration generated by $u_\cdot$ and $p(u_\cdot)$ agree up to sets of  $\PP_{u_0}^\B$-measure zero for some $u_0\in N$. \end{enumerate}
Then $\rank \sigma_u^\B=\rank \sigma_{p(u)}^\A$, all $u\in N$.
\end{lemma}
\begin{proof}
Set $p=\rank \sigma^\A_x$ and $\tilde p=\rank \sigma_u^\B$. By assumption $p$
and $\tilde p$ do not depend on $x\in M$ and $u\in N$. Take connections on $\Image \sigma^\B$ and $\Image \sigma^\A$ which are metric for the metrics induced by the symbols. Extend these connections to $TN$ and $TM$. The martingale part of the stochastic anti-development of $(u_\cdot)$ will be a Brownian motion stopped at $\zeta^\B$ of dimension $\tilde p$ and that of $( p(u_\cdot) )$ will be one of dimension $p$. By (ii) these have the same filtration up to sets of measure zero. But this implies $p=\tilde p$ by the martingale representation theorem, as required.
\end{proof}
\begin{proposition}
\label{pr:cohesive-diffusion}
The following are equivalent for $\B$ over $\A$ when $\A$ is cohesive:
\begin{enumerate}
\item [(a)] $\B=\A^H$
\item[(b)] $\B$ is cohesive and the filtration generated by its associated diffusion $(u_\cdot)$ agrees with that of $p(u_\cdot)$ up to sets of $\PP_{u_0}^\B$-measure zero for given $u_0$ in $N$.
\end{enumerate}
\end{proposition}
\begin{proof}
If (b) holds, Lemma \ref{le:cohesive} shows that $\Image [\sigma_u^\B]=H_u$ for each $u \in N$, since by (\ref{h-lift-1}) we always have $H_u\subset \Image [\sigma_u^\B]$. Thus (b) implies criterion (ii) of Proposition \ref{pr:cohesive}. Also (b) follows from (iii) of Proposition \ref{pr:cohesive} by considering the stochastic differential equation driven by horizontal lifts $\tilde X^0, \dots, \tilde X^m$.
\end{proof}

\section{The Filtering Equation}
\label{se-filt}
Let $p:N\to M$ be a smooth surjective map. Suppose that $\B$ is over $\A$. However we do not assume $\sigma^\A$ of constant rank. Let $\{\PP^\B_{u_0}\}$  and  $\{\PP_{x_0}^\A\}$ be, respectively,  the solutions to the martingale problem for $\B$ and $\A$ on the canonical spaces $\C(M^+)$ and $\C(N^+)$. Denote by $(u_t)$ and $(x_t)$ the corresponding canonical processes with life time $\zeta^N$ and $\zeta^M$ respectively. 
Note that $\zeta^\B\leqslant \zeta^\A\circ p$ almost surely with respect to $\PP_{u_0}^\B$. We shall assume that the paths of the diffusion on $N$ do not explode before their projections on $M$ do, more precisely  $\zeta^M\circ p=\zeta^N$ almost surely with respect to $\PP_{u_0}^\B$ for each $u_0$, equivalently,
\begin{itemize}
\item  {\bf Assumption S. }
$$\C_{u_0}^pM^+:=\{\sigma:[0,\infty)\to M^+: \lim_{t\to \zeta^\B}p(u_t)=\Delta \hbox { when } \zeta^\B<\infty\}$$
has full  $\PP_{u_0}^\B$ measure for each $u_0\in N$.
\end{itemize}
Denote by the following the filtrations induced by the processes indicated:
\begin{eqnarray*}
\F_t^{u_0}& =&\sigma (u_s, 0\leqslant s\leqslant t), \qquad \F^{u_0}=\sigma (y_s, 0\leqslant s< \infty) \\
\F_t^{x_0}&=&\sigma (x_s, 0\leqslant s\leqslant t), \qquad\F^{x_0}=\sigma (x_s, 0\leqslant s< \infty)\\
\F_t^{p(u_0)}&=&\sigma (p(u_s), 0\leqslant s\leqslant t), \qquad\F^{x_0}=\sigma (p(u_s), 0\leqslant s< \infty).
\end{eqnarray*}
\begin{proposition}
\label{pr:mart-mea}
Under Assumption S, $p_\ast(\PP^\B_{u_0})=\PP^\A_{p(u_0)}$ and
 $P_t^\B(f\circ p)=P_t^\A(f\circ p)$ for all $f\in \C_c^\infty (M)$.
\end{proposition}
\begin{proof}
If $p(u_0)=x_0$, $f\in \C_K^\infty(M)$, we only need to show that $M_t^{df,\A}$ is a martingale with respect to $p^*(\PP^\B_{u_0})$. Using Assumption S,
\begin{eqnarray*}
M_t^{df, \A}(p(u))&=&f(p(u_t))-f(p(u_0))-\int_0^t \A f\circ p(u_s)) ds\\
&=&f(p(u_t))-f(p(u_0))-\int_0^t \left(\B (f\circ p) \right) (u_s) ds\\
&=&M_t^{d(f\circ p), \B}
\end{eqnarray*}
is a martingale with respect to $(\F_t^{u_0})$ and $\PP^\B_{u_0}$. Take $s\leqslant t$ and let $G$ be a $\F_s^{x_0}$-measurable function. Then
\begin{eqnarray*}
&&\E^{p_\ast(\PP^\B_{u_0})}\left\{M_t^{df,\A} G\right\}
=\E^{\PP^\B_{u_0}}\left\{M_t^{df,\A}(p(u_\cdot))  G(p(u_\cdot))\right\}\\
&=&\E^{\PP^\B_{u_0}}\left\{M_t^{d(f\circ p),\B}  G\circ p\right\}
=\E^{\PP^\B_{u_0}}\left\{M_s^{d(f\circ p),\B}  G\circ p)\right\}\\
&=&\E^{p_\ast(\PP^\B_{u_0})}\left\{M_s^{df,\A}  G\right\}
\end{eqnarray*}
and the result follows from the uniqueness of the martingale problem.\end{proof}

We will need the following elementary lemma:

\begin{lemma}
\label{le:mart-4}
Let $(\Omega, \F, \F_t, \PP\}$ be a filtered probability space and $\G_*$ a sub-filtration of $\F_*$ with the property that for all $s\geqslant  0$, 
\begin{equation}\label{cond-1}
\E\{A|\G_s\}=\E\{\E\{A|\F_s\}| \G\}, \qquad \forall A\in \F,
\end{equation}
where $\G=\vee_s \G_s$.
Then 
\begin{description}
\item[(i)] $(\E\{M_t|\G\}, t\geqslant  0)$ is a $\G_*$-martingale whenever $(M_t: t\geqslant  0)$ is an $\F_*$-martingale;
\item[(ii)] For all $\G$-measurable and integrable $H$ $$\E \big\{H|\F_s\}=\E\{H|\G_s\};$$
\item[(iii)] $\E\big\{\E\{A|\F_s\}| \G\big\}=\E\big\{\E\{A|\G\}| \F_s\big\}, \qquad \forall A\in \F$.
\end{description}
\end{lemma}

\begin{proof}
For (i) set $N_t=\E\{M_t |\G\}$, $0\leqslant t<\infty$. By (\ref{cond-1}), $(N_t)$ is $\G_t$ measurable.
 For $s\leqslant t$ suppose that $f$ is $\G_s$-measurable and bounded. Then
$\E (N_tf)=\E(M_t f)=\E(M_s f)=\E(N_s f)$.
For (ii), let $H$ and $F$ be bounded measurable functions
with $\G$-measurable and $\F_s$-measurable representations. Then
$$\E\big\{H|F\big\}=\E\big\{H|\E\big\{F|\G\big\} \big\}
=\E\big\{H|\E\big\{F|\G_s\big\} \big\} =\E\big\{H|\E\big\{F|\G\big\} \big\}
$$
using (\ref{cond-1}). Thus
$\E \big\{H|\F_s\}=\E\{H|\G_s\}$ as required.
Part (iii) follows from (ii) on taking $H=\E\big\{\E\{A|\G\}$ and using equation (\ref{le:mart-4}).
\end{proof}
Part (ii) of the following proposition says that the filtration $\F_*^{p(u_0)}$ is \emph{immersed in} the filtration $\F_*^{u_0}$ in the terminology of Tsirelson \cite {Tsirelson-01}.

\begin{proposition}
\label{pr:mart-3}
\begin{description}
\item [(i)]
 For fixed $t>0$ let $f$ be a bounded $\F_t^{u_0}$-measurable function. Then $$\E \left\{f | \F^{p(u_0)}\right\}=\E \left\{f |\F^{p(u_0)}_t\right\}.$$
\item[(ii)]
  All $\F_*^{p(u_0)}$
martingales are $\F_*^{u_0}$ martingales. In fact if $f=G\circ p$ for $G$ an integrable functional on $C(M^+)$ with respect to $\PP^\A$, we have
$$\E^{u_0}\{f| \F_t^{p(u_0)}\}=\E^{u_0}\{f | \F_t^{u_0}\}.$$
\end{description}
\end{proposition}

\begin{proof}
(i) Write $f=F(u_s: 0\leqslant s\leqslant t)$ for $F$ a bounded measurable function on $\C(N^+)$. Let $G$
be bounded measurable functions of $\{p(u_s): 0\leqslant s\leqslant t \}$ 
 and $g^1,\dots, g^k$ bounded Borel functions on $M$, with $h^1,\dots, h^k$ positive real numbers. By the Markov property of $u_\cdot$ and of $p(u_\cdot)$,
  \begin{eqnarray*}
&&\E \left(F(u_s: 0\leqslant s\leqslant t)   \, G\,g^1\circ p(u_{t+h^1})\cdot \dots
\cdot g^k\circ p(u_{t+h^1+\dots+h^k}) \right)\\
=&&\E\left(F(u_s: 0\leqslant s\leqslant t)    G \; P_{h^1}^\A  \left( g^1 P_{h^2}^\A (g^2\dots P_{h^k}^\A g^k)\right) (p(u_t))\right).
\end{eqnarray*}
Therefore,
$$\E \left\{F(u_s: 0\leqslant s\leqslant t) | \F^{p(u)}\right\}=\E \left\{F(u_s: 0\leqslant s\leqslant t) | \F^{p(u)}_t\right\}$$
as required.

Part (ii) is immediate from (i) by Lemma \ref{le:mart-4}.
\end{proof}

As in \S\ref{se:deco-1-1} set $E_x=\Image \sigma_x^\A$ with 
$h_u:E_{p(u)}\to T_uN$ the horizontal lift defined by ({\ref{h-lift-1}), although now we have no constant rank assumption and so no smoothness of $\h$. Also let $E_u^\B=\Image \sigma_u^\B$.
For an $\F_*^{x_0}$ -predictable $E^*$-valued process $\phi_t:=\phi_t(\sigma_\cdot)$, $0\leqslant t<\zeta^\A$ along $(\sigma_t: 0\leqslant t<\zeta^\A)$ let $(p^*(\phi_t): 0\leqslant t<\zeta^\B)$ be the pull back restricted to be an $(E^\B)^*$-valued process along $(u_t: 0\leqslant t<\zeta^\B)$ defined by 
$$p^*(\phi_t)(u_\cdot)=\phi_t(p(u))\circ T_{u_t} p: {E^\B}_{u_t} \to \R.$$
Since $\phi_t$ has a predictable extension $\bar \phi_t$ so does $p^*(\phi_t)$ and so the latter is predictable. Moreover $p^*(\phi_t)\sigma^\B(p^*\phi_t)=\phi_t\sigma^\A(\phi_t)$ by Lemma \ref{le:commutative}
showing $\phi_\cdot$ is in $L^2_\A$ if and only if $p^*(\phi_\cdot)$ is in $L^2_\B$.
For such $\phi$ we have the following intertwining:

\begin{proposition}
\label{pr:mart-2}
Let $\phi$ be a predictable $L^2_\A$-valued process.
\begin{enumerate}
\item[(1)]  For $\PP_{u_0}^\B$ almost surely all sample paths, $M_t^{\A, \phi}\circ p=M_t^{\B, p^*(\phi)}$ for $t<\zeta^\B$. 
\item[(2)] If $\alpha\in L^2_\B$ with $\alpha_t\circ h_t=0$ almost surely, then $\langle M_t^\alpha, M_t^{df\circ Tp}\rangle=0$ and  $\E^{\B,u_0} \{M_t^\alpha | \F^{p(u_0)}\}=0$ for all $C^1$ functions $f$ on $M$ .
\end{enumerate}
\end{proposition}

\begin{proof}
For $\phi=df$, (1) follows from $p^*(df)_u=d(f\circ p)_u$ as in the proof of Proposition~\ref{pr:mart-mea}.
For general $\phi$, taking a predictable extension if necessary, write $\phi_t(x)=\sum_1^m g_t^i(x_\cdot)(df^j)_{x_t}$ for smooth functions $f^j: M\to \R$ and real valued predictable $\{g_t^j: 0\leqslant t<\zeta^\A\}$. Therefore
$$M_t^{\A,\phi}\circ p=\sum_{j=1}^m \int_0^t g_s^j \left(p(u_\cdot)\right) \;dM_t^{\B, p^*(df^j)}=M_t^{\B, p^*(\phi)}$$ for all $t<\zeta^\B$, giving (1).
For (2) let $F:N\to \R$ be a smooth measurable function with respect to $\F^{p(u_0)}$. Then $F=f(p(u_\cdot))$ for some measurable function $f: M\to \R$. 
\begin{eqnarray*}
\E^{\B,u_0}\left(M_t^\alpha f(p(u_\cdot))\right)&=&{1\over 2} \E^{\B,u_0}
\langle M_t^\alpha, M_t^{df\circ Tp}\rangle
={1\over 2}\E^{\B,u_0} \int_0^t \sigma^\B(\alpha_s, df\circ Tp(u_s))ds.
\end{eqnarray*}
If $\alpha_t h_{u_t}=0$ almost surely for all $t$, we apply (\ref{h-lift-1}) to see
$$\sigma^\B(\alpha_s, df\circ Tp(u_s))=\alpha_s\sigma^{\B}_{u_s}\Big(T^*p(df)\Big) =\alpha_s h_{u_s}\sigma^\A_{p(u_s)}df=0$$ 
and thus $\E^{\B,u_0}(M_t^\alpha f(p(u_\cdot)))=0$ giving (2).
\end{proof}

For $\alpha\in L^2_\B$ define $\beta_s\equiv \E^{\B,u_0}\{\alpha_s\circ h_{u_s} |p(u_\cdot)=x_\cdot\}, 0\leqslant s<\zeta$ to be the unique, up to equivalence, element of $L^2_\A$ such that
\begin{equation}
\label{conditional}
\E^{\B,u_0}\Big(\alpha_a\circ h_{u_s}\sigma^\A(\phi_s(p(u_\cdot)))\Big)
=\E^{\A, p(u_0)} \Big(\beta_s\sigma^\A(\phi_s)\Big).
\end{equation}
for any $\phi\in L^2_\A$. To see such an element exists and is unique recall that 
$$\alpha_s\circ h_{u_s}\sigma_{p(u_s)}^\A=\alpha_s \sigma^\B_{u_s}(t_{u_s}p)^*$$
which is an $\F^{u_0}_*$-predictable process with values in $E_{p(u_s)}\subset T_{p(u_s)}M$ at each time $s$, and by Proposition~\ref{pr:mart-3}, (\ref{conditional}) is equivalent to
\begin{equation}
\label{conditional-2}
\beta_s(p(u_\cdot))\sigma^\A_{p(u_s)}
=\E^{\B,u_0}\Big\{\alpha_s\sigma_{u_s}^\B(T_{u_s}p)^* |\F^{p(u_0)}\Big\}
\end{equation}
in the sense of Elworthy-LeJan-Li \cite{Elworthy-LeJan-Li-book}. The predictable projection theorem and the results of \cite{Elworthy-LeJan-Li-book} shows that there is a unique, up to indistinguishability, $\F^{p(u_\cdot)}$-predictable $TM$ versiob $\{\gamma_t: 0\leqslant t<\zeta\}$ say, over $\{p(u_t): 0\leqslant t<\zeta\}$, of the right hand side of (\ref{conditional-2}). By applying the uniqueness part of this projection theorem to $\{\phi_s(\gamma_s): 0\leqslant s<\zeta\}$ when $\phi_\cdot$ is $\F^{p(u_\cdot)}_*$-predictable, $T^*M$-valued over $p(u_\cdot)$ and $\phi_t$ vanishes on $E_{p(u_t)}$ for all $0\leqslant t<\rho$ with probability $1$, we see $\gamma_t\in E_{p(u_u)}$ for all $0\leqslant t<\zeta$ almost surely. Now set $\beta_s(p(u_\cdot))=[\sigma_{p(u_s)}^\A]^{-1}\gamma_s$ in $E^*_{p(u_s)}$.

\begin{proposition}
\label{pr:mart-4}
For any $\alpha_\cdot$ in $L^2_\B$ we have
$$\E^{\B, u_0}\left\{M_t^\alpha \, |\, p(u_\cdot)=x_\cdot \right\}=
\int_0^T \E ^{\B, u_0}\left\{\alpha_s\circ h_{u_s}\, | \, p(u_\cdot)=x_\cdot\right\} d\{x_s\}.$$
\end{proposition}

\begin{proof}
Set $N_t=\E\{M_t^\alpha\, |\, \F^{p(u_0)}\}$ and write $N_t(u)=\bar N_t(p(u))$ for $\{\bar N_t\}$ a $\F_t^{x_0}$-measurable function. By Proposition \ref{pr:mart-3}, $(N_t)$ is an $\F_*^{p(u_\cdot)}$-martingale and we see $(\bar N_t)$ is an $\F_*^x$ martingale. Take
$g\in \C_c^\infty M$ then by Proposition \ref{pr:mart-2},
\begin{eqnarray*}
\langle \bar N, M^{\A, dg}\rangle_t \circ p(u)
&=& \E^{\B, u_0} \left\{\langle M^{\alpha}, M^{d(g\circ p)}\rangle_t |\F_t^{p(u_0)}\right\}(u)\\
&=&\E^{\B, u_0} \left\{\sigma^\B_{u_t} \left(\alpha_t, (T_{u_t}p)^*(dg)\right) | \F_t^{p(u_0)}\right\}\\
&=&\E^{\B, u_0} \left\{\alpha_t \circ h_{u_t}\sigma^\A_{p(u_t)}(dg)| \F^{p(u_0)}_t\right\}
\end{eqnarray*}
by equation (\ref{h-lift-1}). By Proposition \ref{pr:mart-1} and the definition above of the conditional expectation,
$\bar N_t(p(u_\cdot))=M^{\A,\beta}$ for $\beta\circ p(u_\cdot)=\E^{\B, u_0} \left\{\alpha_t \circ h_{u_t} | \F^{p(u_0)}\right\}$ and so
$$\bar N_t(x_\cdot)=\int_0^t \E \left\{\alpha_s\circ h_{u_s} |p(u_\cdot)=x_\cdot\right\} d\{x_s\}$$
as required.
\end{proof}

\section{A family of Markovian kernels}
For a probability measure $\mu_0$ on $N^+$ let the measures $\mu_t$ on $N^+$ be the flow of $u_t$ under $\PP_{\mu_0}^\B$ and set $\nu_t=p_*(\mu_t)$ on $M^+$. Let $\eta_{\mu_0}$ be the law of $u_\cdot\mapsto (p(u_\cdot), u_0)$ on $\C(M^+)\times N^+$ under $\PP_{\mu_0}^\B$ so
$$\eta_{\mu_0}(A, \Gamma)=\int_{y\in M^+} \PP_y^{\A}(A)\,\rho^y_{\mu_0}(\Gamma)\,\nu(dy), \qquad A\in \B(M^+), \Gamma\in\B(N^+)$$
where $\rho_{\mu_0}^y$ arises from a disintegration of $\mu_0$
$$\mu_0(\Gamma)=\int_{y\in M^+} \rho^y_{\mu_0}(\Gamma)\,\nu(dy), \qquad \Gamma\in \B(N^+).$$

For a measurable $f:N^+\to \R$, integrable with respect to $\mu_t$ set
\begin{equation}
\pi_t^{\mu_0,\sigma}f(v)=\E_{\mu_0}^\B\{f(u_t)| p(u_0)=\sigma, u_0=v\}.
\end{equation}It is defined for $\eta_{\mu_0}$ almost all $(\sigma,v)$ in $C(M^+)\times N^+$. In particular for $\PP_{\nu_0}^\A$-almost all $\sigma$ it is defined for $\rho_{\mu_0}^{\sigma(0)}$-almost all $v\in N^+$. We could use the convention that
$$\pi_t^{\mu_0,\sigma}f(v)=0$$ if $p(v)\not =\sigma(0)$.
With this convention, if we define $\theta_t\sigma(s)=\sigma(t+s)$ we see that for $\PP_{v_0}^\A$-almost all $\sigma$ the map $y\mapsto \pi_t^{\mu_t, \theta_t\sigma}f(y)$ is defined for $\mu_t$-almost all $y$ in $N^+$. 

Further for $u_0\in N$ and $f: N^+\to \R$ bounded measurable define $$\pi_tf(u_0): \C_{p(u_0)}M^+\to \R,$$
$\PP^\A_{p(u_0)}$-almost surely, by
\begin{equation}
\pi_tf(u_0)(\sigma)=\E\Big\{f(u_t)|p(u_\cdot)=\sigma\Big\}=\pi_t^{\delta_{u_0},\sigma}f(u_0).
\end{equation}
This can be extended, as in \cite{Elworthy-LeJan-Li-book}, to the case of predictable process in vector bundles over $N$, and to define
$$\pi_t(\alpha\circ h_{u_\cdot})(u_0): \C_{p(u_0)}M^+\to \R$$
as $\E^{\B, u_0}\{\alpha_s h_{u_s} |p(u_\cdot)=x_\cdot\}$, defined above.

\section{The filtering equation}

\begin{theorem}
\label{th:mart-2}
\begin{enumerate}
\item [(1)]If $f$ is $\C_c^2N$, or more generally if $f$ is $C^2$ with $\B f$ and $\sigma^\B(df,df)\circ \h$ bounded, then
\begin{equation}
\pi_tf(u_0)=f(u_0)+\int_0^t \pi_s(\B f)(u_0) ds+
\int_0^t  \pi_s(df\circ h_{u_\cdot}) (u_0)d\{x_s\}.\end{equation}
In particular $\{\pi_t f(u_0): t\geqslant  0\}$ is a continuous $\F^{p(u_0)}_*$ semi-martingale.
\item[(2)] For bounded measurable $f: M^+\to \R$ and $\PP^\A_{v_0}$ almost all $\sigma$ in $C(M^+)$, for each $s,t\geqslant  0$
\begin{equation}
\label{cocycle-1}
S^{\mu_0,\sigma}_{t+s}f(v)
=\pi_t^{\mu_0,\sigma}\pi_s^{\theta_t\sigma,\, \mu_t}f(v)
\end{equation}
for $\rho_{\mu_0}^{\sigma(0)}$ almost all $v$ in $N^+$.
\item[(3)] Moreover there exists a family of probability measures $Q_\nu^{\mu_0, \sigma}$ on $C(N^+)$ define for $\eta_{u_0}$-almost surely all $(\sigma, v)$ such that if
$F: \C(N^+)\to \R$ is of the form
$$F(u_\cdot)=f_1(u_{t_1})\dots f_n(u_{t_n})$$
some $0\leqslant t_1<t_2<\dots t_n$ and bounded measurable $f_j: N^+\to \R$, $j=1,2,\dots, n$ then
\begin{eqnarray*}
\int_{u\in\C(N^+)}F(u)Q_v^{\mu_0, \sigma}(du)
&=&S_{t_1}^{\mu_0, \sigma}\Big(f_1S_{t_2-t_1}^{\mu_{t_1}, \theta_{t_1}\sigma}\Big(f_2\dots S_{t_n-t_{n-1}}^{\mu_{t_n}, \theta_{t_{n-1}}\sigma}f_n\Big)(v)\\
&=&\E^\B_{\mu_0}\{F(u_\cdot)| p(u_\cdot)=\sigma_0\}
\end{eqnarray*}
$\eta_{\mu_s}$-almost surely in $(\sigma, v)$.
\end{enumerate}

\end{theorem}

\begin{proof}
(1). By definition of $M^{df}$ we have 
$$f(u_t)=f(u_0)+\int_0^t \B f(u_s) ds+M_t^{df}$$
so 
\begin{equation}
\label{filt-32}
\pi_t f(u_0)=f(u_0)+\int_0^t \pi_s \B f (u_0) ds+\E \left\{M_t^{df,\B} \,| \, p(u_\cdot)=x\right\}
\end{equation}
and part (1) follows from Proposition \ref{pr:mart-4}.

 (2). We observed above that the right hand side of (\ref{cocycle-1}) is well defined
for $\PP^\A_{\mu_0}$ almost all $\sigma$. The equation then follows from the Markov property.

(3). The existence of regular conditional probabilities in our situation implies the existence of the probabilities $\Q_\nu^{\mu_0, \sigma}$ as required, together with a standard use of the Markov property.
\end{proof}

\begin{remark}
 A description of the $\Q^{\mu_0, \sigma}_v$ is given in the next section, in the case where $\A$ is cohesive.
\end{remark}

Recall we have the decomposition $F_u=H_u+VT_uN$ for each $u\in N$, and $F=\sqcup F_u$. If $\ell\in F_u^*$ there is a corresponding decomposition 
$$\ell=\ell^H+\ell^V\in F_u^*,$$
where $\ell^H$ vanishes on $VT_uN$ and $\ell^V$ on $H_u$.
For $\ell \in T_u^*N$ write $\ell^V=(\ell|{F_u})^V$ and $\ell^H=(\ell|{F_u})^H$.

\begin{corollary}
\label{co:8.14}
Suppose $\A$ is cohesive.  If $f$ is $\C_c^3N$ then there is the
Stratonovitch equation
\begin{equation}
\pi_tf(u_0)(x_\cdot)
=f(u_0)+\int_0^t \pi_s(\B^Vf)(u_0)ds+\int_0^t
\pi_s( df_{u_0} \circ h_{u_o})\circ dx_s.\qquad\qquad
\label{filt-31}
\end{equation}
\end{corollary}

\begin{proof}
We use (\ref{filt-32}). By Proposition~\ref{pr:mart-4}, $$\E\{M_t^{df}\, | \,p(u_\cdot)=x_\cdot\}=\E\{M_t^{df^H}\, | \,p(u_\cdot)=x_\cdot\}.$$   Note that
$$M_t^{df^H}=\int_0^t (df^H)_{u_s} \circ du_s
-\int_0^t \delta^\B (df^H)(u_s) ds$$
by Lemma \ref{le:mart-1}. Furthermore
$$\delta^\B (df^H)=\delta^{\B^V} (df^H)+\delta^{\A^H} (df^H)
=\delta^{\A^H} (df^H)=\delta^{\A^H} (df)=\A^H(f)$$
since $df^H$ vanishes on vertical vectors and $df=df^H+df^V$ while $df^V$ vanishes on horizontal vectors, so $\delta^{\A^H} (df^V)=0$. This gives 
$$\pi_tf(u_0)(x_\cdot)
=f(u_0)+\int_0^t \pi_s(\B^V f)(u_0)(x_\cdot) ds+
\E\left\{\int_0^t (df)_{u_s}^H \circ du_s\; \big| \;p(u_\cdot)=x_\cdot\right\} $$
Finally (\ref{filt-31}) follows since $df^H_u=p^*(df\circ h_u)=df\circ  h_u \circ T_up$ and $T_up\circ du_t=\circ dx_t$.
\end{proof}

\section{Approximations}
Assume now that the law of $u_t$ under $\PP_{u_0}^\A$ is given by
$$P_t^\A(u_0, A)=\int_A P_t^\A(u_0, v)dv, \qquad A\in \B(M)$$
for $p_t^\A(u_0,v)$ a smooth density with respect to some fixed, smooth, strictly positive measure on $M$ to which `$dv$' refers. This is the case if  $\A$ is hypoelliptic. 

Consider the conditional probability 
$$q_t^{u_0,b}(V)=\PP_{u_0}^\B\{u_t\in V|p(u_t)=b\},\qquad V\in \B(N)$$
defined for $p_t^\A(u_0,-)$ almost sure all $b$ in $M$. There is the disintegration of $p_t^\B(u_0,-)$
$$p_t^\B(u_0,V)=\int _{b\in M} q_t^{u_0, b}(V) p_t^\A(p(u_0), db)$$
and the formula
$$\mu_t^{\mu_0,b}(V)=\lim_{\epsilon\downarrow 0}
(p_t^\A(p(u_0),b)]^{-1}
\int_V p_{t-\epsilon}^\B (u_0, dv) p_\epsilon^\A(p(v),b).$$ 
Take a nested sequence $\{\Pi^\ell\}_{l=1}^\infty$ of partitions of $[0,t]$
$$\Pi^l=\{0=t_0^l<t_1^l<\dots < t_{k_l}^l=t\},$$
say, with union dense in $[0,t]$. For any continuous bounded $f: N^+\to \R$ there is the following approximation scheme to complete $\pi_tf(u_0)$:
\begin{proposition}
\begin{eqnarray*}
\pi_tf(u_0)(\sigma)
&=&\lim_{l\to \infty}
\int q_{t_1^l}^{u_0, \sigma(t_1^l)}(dv_1)q_{t_2^l-t_1^l}^{v_1, \sigma(t_2^l)}(dv_2)
\dots q_{t_{l_k}-t_{l_k-1}} ^{v_{k_l-1}, \sigma(t)}(dv_{k_l})f(v_{k_l})
\\
&=&\lim_{l\to \infty} \E_{u_0}^\B
\big\{ f(u_t) \;|\; p(u_{t_j}^l)=\sigma(t_j^l), \qquad 1\leqslant j\leqslant k_l\}.
\end{eqnarray*}

\end{proposition}
\begin{proof}
The two versions of the right hand sides are equal before taking limits.
For $\l=1,2,\dots,$, set
$$S^l f(\sigma)=\E_{u_0}
\big\{f(u_t) \;|\;p(u_{t_j^l})=\sigma(t_j^l), \quad 1\leqslant j\leqslant k_l\}.$$
It is defined for $\PP_{x_0}^\A$-almost all $\sigma$ in $C(M^+)$, where $x_0=p(u_0)$. Let $Q^l$ be the $\sigma$-algebra on $C(M^+)$ generated by $\sigma\mapsto (\sigma(t_1^l), \dots, \sigma(t_j^l))$. Directly from the definitions we see
$$\pi_t^lf=\E \{\pi_t(f)(u_0) \;|\; Q^l\},$$
and so $\{S^lf\}_{l=1}^\infty$ is an $Q^*$-martingale. it is bounded and so converges $\PP_{x_0}^\A$-almost surely. Since $\vee_l Q^l$ is the Borel $\sigma$-algebra the limit is $\pi_tf(u_0)$ as required.
\end{proof}

\section{Krylov-Veretennikov Expansion}
\label{se:expansion}
Suppose $\A=\sum_{j=1}^m \LL_{X^j} \LL_{X^j}+ \LL_{A}$ for smooth vector fields $\{X^j\}_{j=1}^m$ and $A$. We will now take $\{x_t: 0\leqslant t<\zeta\}$ to be the solution to the stochastic differential equation
\begin{equation}
\label{sde}
dx_t=X(x_t)\circ dB_t+A(x_t)dt,
\end{equation}
with $x_0$ given, for a Brownian motion $B_\cdot$ on $\R^m$, rather than the canonical process. Here $X(x): \R^m\to T_xM$ is the map given by
$$X(x)(a^1,\dots, a^m)=\sum_{j=1}^m  a^j X^j(x), \quad x \in M.$$

\noindent 
Let $\{P_t: t\geqslant  0\}$ be the sub-Markovian semi-group generated by $\B$. Let $f\in \C_c^\infty N$. Assume $P_t f\in \C^\infty N$.

\noindent  As in the proof Theorem \ref{th:mart-2}, from
$$P_{t-s}f(u_s)=P_t f(u_0)+\int_0^s d(P_{t-r} f)_{u_r} d\{u_r\}, \quad 0\leqslant s\leqslant t$$ we obtain 
\begin{eqnarray*}
\pi_s P_{t-s} (f) (u_0) (x_\cdot) &=&P_t f(u_0)+\int_0^s \E \left\{ d(P_{t-r} f)_{u_r}\circ h_{u_r} \, |\, p(u_\cdot)=x_\cdot \right\} d\{x_r\} \\
&=& P_tf(u_0)+\int_0^s \E \left\{ d(P_{t-r} f)_{u_r}\circ h_{u_r} \, |\, p(u_\cdot)=x_\cdot \right\}X(x_r) dB_r
\end{eqnarray*}
so that $\pi_s P_{t-s}f (u_0), 0\leqslant s\leqslant t, $ is a continuous $\F^{x_0}_*$ semi-martingale.
Therefore
\begin{eqnarray*}
\pi_tf(u_0)-P_tf(u_0)
&=&\int_0^t d_s (\pi_s P_{t-s}f(u_s))\\
&=&\int_0^t  \E \left\{ d(P_{t-r} f)\circ h_{u_r} \, |\, p(u_\cdot)=x_\cdot \right\}X(x_r) dB_r\\
&=&\int_0^t S_r\big[d(P_{t-r}f)\circ h_{-} \circ X^k(p(-))\big] (u_0)dB_r^k
\end{eqnarray*}
giving a `Clark-Ocone' formula for $\pi_t f(u_0)$. Iterating this procedure formally, 
\begin{eqnarray*}
&&\pi_tf(u_0)=P_tf(u_0)+\int_0^t S_r\big[d(P_{t-r}f)\circ h_{-} \circ X(p(-))\big] (u_0)dB_r\\
&&+\int_0^t\int_0^r \pi_s\Big[dP_{r-s}\big[ d(P_{t-r}f)\circ h_-\circ X^k(p(-))\big] h_- \circ X^j(p(-)\Big] dB^j_s dB^k_r\\
&&=\dots,
\end{eqnarray*}
we obtain the Wiener chaos expansion of $\pi_tf(u_0)(x_\cdot)$.

\section{Conditional Laws}
\label{se:cond-law}
It will be convenient to extend the notation of section \ref{se-filt}.
For $0\leqslant l<r<\infty$ let $\C(l,r;N^+)$ and $\C(l,r;M^+)$ be respectively the space of continuous paths $u: [l,r]\to N^+$ and $x: [l,r]\to M^+$ which remain at $\Delta$ from the time of explosion; and $\C_{u_0}(l,r;N^+)$ and $\C_{x_0}(l,r;M^+)$ the paths from $u_0\in N^+$ and $x_0\in M^+$ respectively,
 Let $\{\PP^{(l,r),\B}_{u_0}\}$ and $\{\PP^{(l,r),\A}_{x_0}\}$ be the associated diffusion measures.

\noindent  
The conditional law of $\{u_s: l\leqslant s\leqslant r\}$ given $\{p(u_s): l\leqslant s\leqslant r\}$ will be given by probability kernels
$\sigma\mapsto \Q_{\sigma, u_0}^{l,r}$ defined $\PP^{(l,r);\A}$ almost surely from $\C_{p(u_0)}(l,r;M^+)$ to $\C_{u_0}^p(l,r;N^+)$ for each $u_0\in N$, where 
$\C_{u_0}^p(l,r;N^+)$ is the subspace of $\C_{u_0}(l,r;N^+)$ whose paths satisfy Assumption S. The defining property is that for integrable $f: \C_{u_0}(l,r;N^+)\to \R$ 
\begin{equation}
\label{cond-law}
\E\left\{ f(u_\cdot)\; |\; p(u_s)=\sigma_s, l\leqslant s \leqslant r\right\}
=\int_{y\in \C_{u_0}(l,r;N^+)}  f(y) d\Q_{\sigma, u_0}^{l,r} (y).
\end{equation}
To obtain the conditional law take the decomposition $\B=\A^H+\B^V$ of Proposition \ref{th:deco}. Represent the diffusion corresponding to $\A$ by a stochastic differential equation
\begin{equation}
\label{sde-1}
d x_t'=X(x_t')\circ dB_t+X^0(x_t')dt.
\end{equation}
Take a connection $\nabla^V$ on $VTN$ and let
\begin{equation}
(\nabla^V) \qquad  dz_t=V(z_t)dW_t+V^0(z_t)dt
\end{equation}
be an It\^o equation whose solutions are $\B^V$-diffusions. Here $(W_t)$ is the canonical Brownian motion on $\R^m$ for some $m$, independent of $(B_\cdot)$, the map $V: M\times \R^m\to TM$ takes values in $\ker[Tp]$, and $V$ and $V^0$ are locally Lipschitz. For such a representation of $\B^V$ diffusions see the Appendix B. Let $\tilde X: N\times \R^m\to H$ and $\tilde X^0: N\to H$ be the horizontal lifts of $X$ and $X^0$ respectively using Proposition \ref{pr:h-map}. The solution to
\begin{eqnarray*}
(\nabla^V) \qquad dy_t&=&\tilde X(y_t)\circ dB_t+\tilde X^0(y_t)dt+V(y_t)dW_t+V^0(y_t)dt, \\
y_l&=&u_0, \qquad u_0\in N, \quad l\leqslant t\leqslant r.
\end{eqnarray*}
has law $\PP_{u_0}^{(l,r),\B}$. Noting that $\tilde X(u)=h_u X(p(u))$ for $u\in M$,

\begin{eqnarray}
\label{sde-5}
(\nabla^V) \qquad dy_t&=&h_{y_t}\circ dx_t'+V(y_t)dW_t+V^0(y_t)dt,\\
y_l&=&u_0,  \quad  l\leqslant t\leqslant r,
\nonumber
\end{eqnarray}
where $x_t'=p(y_t)$ so that $(x_t')$ is a solution to (\ref{sde-1}) starting from $p(u_0)$ at time $l$. Without changing the law of $y_\cdot$ we can replace $x'$ by the canonical process $x_\cdot$. Then 
\begin{theorem}
\label{pr:regcond-1}
Consider the solution $(y_t)$ as a process defined on the probability space  $\C_{p(u_0)}(l,r; M^+)\times \C_0\R^m$ with product measure,
$$y:[l,r]\times \C_{p(u_0)}(l,r; M^+)\times \C_0\R^m\to N^+,$$
and define $\Q^{l,r}_{\sigma, u_0}$ to be the law of $y(\sigma,-): \C_0\R^m\to \C_{u_0}(l,r;N^+)$. For bounded measurable $f: \C_{u_0}(l,r;N^+)$, 
$$\E\left\{ f(u_\cdot) \;| \;p(u_s)=\sigma_s, l\leqslant s \leqslant r\right\}
=\int_{y\in \C_{u_0}(l,r;N^+)}  f(y) d\Q_{\sigma, u_0}^{l,r} (y).$$
\end{theorem}
\begin{proof}
Take a measurable function $\alpha: \C_{p(u_0)}(l,r;M^+)\to \R$. Then
\begin{eqnarray*}
&&\E^{\PP^\B_{u_0}} \left(\alpha(p(u)) \int_{y\in \C_{u_0}(l,r;N^+)}  f(y) \,d\Q_{p(u), u_0}^{l,r} (y)\right)\\
&=&\E^{\PP^\A_{p(u_0)}}  \left(\alpha(x) \int_{y\in \C_{u_0}(l,r;N^+)}  f(y) \,d\Q_{x, u_0}^{l,r} (y)\right)\\
&=&\E^{\PP^\A_{p(u_0)}}  \left(\alpha(x) \int_{\C_0\R^m}  f(y(x,\omega)) \,d\PP(\omega))\right)\\
&=&\int_{C_{p(u_0)}(l,r;M^+)\times C_0\R^m}
  \left(\alpha(x)   f(y(x,\omega)) \, d\PP^{\A}_{p(u_0)}d\PP(\omega))\right)\\
  &=&\E f(u)\alpha(p(u)),
\end{eqnarray*}
 as required.
\end{proof}

Note that Theorem \ref{pr:regcond-1} is equivalent to the statement that $\omega\mapsto\Q^{l,r}_{p(\omega), u_0}$, $\omega\in\C_{u_0}(l,r;N^+)$, is a regular conditional probability of $\PP^{(l,r), \B}_{u_0}$ given $p$. 

\begin {remark}
Let $(\xi^l _t (\cdot,\cdot), l\leqslant t<\infty)$ be a measurable flow for (\ref{sde-1}) and $(\eta^l _t (\sigma,\cdot,),  0 \leqslant t<\infty)$  one for (\ref{sde-5}) with $x'$ replaced by $\sigma\in \C_{p(u_0)}(l,r;M^+)$. For $\omega \in \Omega$, the underlying probability space for the Brownian motion $\B$, define $\Q^{l,r}_\omega$, from the space of bounded measurable functions on $N^+$ to itself, by 
$$\Q^{l,r}_\omega(f)(u_0)=\E f\left(\eta^l _r (\xi ^l_r (p(u_0),\omega), u_0)\right).$$  
A direct calculation shows that 
$$\Q^{l,r} _\omega \Q^{r,s}_\omega =\Q^{l,s}_\omega$$ for 
$0\leqslant l\leqslant r\leqslant s<\infty$. Thus their adjoints  on a suitable dual space would form an evolution.
\end {remark} 

\noindent  More generally, letting
 $\Borel(X)$ stand for the Borel $\sigma$-algebra of a topological space $X$:
\begin{proposition}
\label{th:regcond-2}
Let $\varphi$ be a measurable map from  $C_{x_0}(l,r;M^+)$ to some measure space, and let $$\PP^{(l,r),\varphi}_{x_0}: \C_{x_0}(l,r;M^+)\times \Borel(\C_{x_0}(l,r;M^+))\to [0,1]$$ be a regular conditional probability for $\PP^{(l,r)}_{x_0}$ given $\varphi$. For $u_0$  with $p(u_0)=x_0$ set 
$$\ Q^{l,r,\varphi\circ p}_{ u_0}(\omega,A)
=\int_{C_{x_0}(l,r;M^+)}\Q^{l,r}_{\sigma, u_0}(A)\PP^{(l,r),\varphi}_{x_0}(p(\omega),d\sigma)$$
for $\omega\in  \C_{u_0}(l,r; N^+)$ and $A \in \Borel\left(\C_{u_0}(l,r;N^+) \right)$. Then$\ Q^{l,r,\varphi\circ p}_{u_0}$ is a regular conditional probability of $\PP^{(l,r),\B}_{u_0}$ given $\varphi\circ p$.
\end {proposition}
\begin {proof} By definition
\begin {eqnarray*}
\Q^{l,r,\varphi\circ p}_{ u_0}(\omega,A)&=&\E^{(l,r),\A, x_0}\left\{\Q^{l,r}_{p(-), u_0}(A)| \varphi\right\}p(\omega)\\
&=&\E^{(l,r),\A, x_0}\left\{\E^{(l,r),\B, u_0}\{\chi_A|p=-\}|\varphi\right\}p(\omega)\\
&=&\E^{(l,r),\B, u_0}\{\chi_A|\varphi \circ p\}(\omega).
\end {eqnarray*}
\end {proof}

\begin{corollary}
\label{cor:regcond-2}
For $\varphi$ as in Theorem \ref{th:regcond-2} suppose that the canonical process on $M^+$ with law $\PP^{(0,T),\varphi}_{x_0}(\sigma,-)$ is a semi-martingale for almost all $\sigma$, in its own filtration $\F^{x_0}_t,0\leqslant t\leqslant T$,  for $\PP^{(0,T),\A}_{x_0}$ almost all $\sigma$. Then the solution $y(\sigma,-)$ to the equation
\begin{eqnarray}
\label{sde-6}
(\nabla^V) \qquad dy_t&=&h_{y_t}\circ d\sigma_t+V(y_t)dW_t+V^0(y_t)dt, \\
y_l&=&u_0, \quad  0\leqslant t\leqslant T
\nonumber
\end{eqnarray}
where $\sigma_t , 0\leqslant t\leqslant T$ is run with law $\PP^{(0,T),\varphi}_{x_0}(\sigma,-)$,
is a version of the $\B$-diffusion from $u_0$ conditioned by $\varphi\circ p$.
\end{corollary}
\begin {proof}
That the law of the solution is as required follows from the discussion at the beginning of this section together with Proposition \ref{th:regcond-2}  and Fubini's theorem.
\end {proof}

Conditions under which conditioned processes are semi-martingales are discussed by
Baudoin ~\cite{Baudoin04}. In particular bridge processes derived from elliptic diffusions are, so we obtain the following version of Carverhill's result ~ \cite{Carverhill-88}:

\begin {corollary}
\label{cor:regcond-3}
Suppose $\A$ is elliptic and let $b_t: 0\leqslant t\leqslant T$ be a version of the $\A$-bridge going from $x_0$ to $z$ in time $T$, some $z\in M$. Then the solutions to
\begin{eqnarray}
\label{sde-7}
(\nabla^V) \qquad dy_t&=&h_{y_t}\circ db_t+V(y_t)dW_t+V^0(y_t)dt, \\
y_0&=&u_0,  \quad  0\leqslant t\leqslant T
\nonumber
\end{eqnarray}
give a version of the $\B$ diffusion from $u_0$ conditioned on $p(u_T)=z$.
\end{corollary}

\section{Equivariant case: skew product decomposition}

In the equivariant case, when $N$ is the total space $P$ of a principal bundle \linebreak  $\pi:P\to M$ as in \S5, a version of Theorem \ref{pr:regcond-1}
is given in \cite{Elworthy-LeJan-Li-principal} which reflects the additional structure. In particular the following is proved there: 

\begin{proposition}
\label{rose}
Let $\B$ be an equivariant diffusion operator on $P$ which induces a cohesive diffusion operator $\A$ on $M$. Let $\{y_t: 0\leqslant t <\zeta\}$ be a $\B$-diffusion on $P^*$. Then $$y_t=\tilde x_t\cdot g_t^{\tilde x_\cdot},$$
where \begin{enumerate}
\item[(i)] $\{\tilde x_t: 0\leqslant t<\zeta\}$ is the horizontal lift of $p(y_\cdot)$, starting at $y_0$, using the semi-connection induced by $\B$
\item[(ii)] $\{g_t^\sigma: 0\leqslant t<\zeta(\sigma)\}$ is a diffusion independent of $\{p(y_t): 0\leqslant t<\zeta\}$ on $G$ starting at the identity with time dependent generator $\Lo_t^{\sigma}$
given by
$$\Lo_t^\sigma f(g)=\sum_{i,j}\alpha^{ij} (\sigma(t)\cdot g)\LL_{A_i^*}\LL_{A_j^*} f(g)+\sum \beta^k(\sigma(t)g)\LL_{A_k^*} f(g), $$
for any  $\sigma\in GP^+$, $0\leqslant t <\zeta(\sigma)$, where $A_1^*,\dots, A_k^*$ are the left invariant vector fields on $G$ corresponding to a basis of $\g$ and the $\alpha^{ij}$ and $\beta^k$ are the coefficients for $\B^V$ as in Theorem \ref{th:op-deco}.
\end{enumerate}
\end{proposition}

\noindent
Note that for each $t$ the operator $\Lo_t^\sigma$ is conjugate to the restriction of $\B^V$ to the fibre through $\sigma(t)$ by the map
\begin{eqnarray*}
\g& \mapsto & p^{-1}(p(\sigma(t))) \\
\g&\mapsto &\sigma(t) g.
\end{eqnarray*}
It is a right invariant operator.

\begin{remark}
Note that by the equivariance of $\Lo^\sigma_\cdot$ there will be no explosion of the process $(g_t^\sigma)$ before that of $\sigma_\cdot$.
Consequently Assumption S of \S\ref{se-filt} holds automatically.
\end{remark}

Below we give the equivariant version of Proposition \ref{pr:regcond-1}. We shall use the notation of \S \ref{se:cond-law}.
However we replace the one point compactification $P^+$ of $P$ by
$\bar P=P\cup \Delta$ with the smallest topology agreeing with that of $P$ and such that $\pi: \bar P \to M^+$ is continuous. Also let $G^+$
be the one point compactification $G\cup \Delta$ of $G$ with group multiplication and action of $G$ extended so that 
$$u\cdot \Delta=\Delta, \Delta \cdot g=g\cdot\Delta=\Delta, \qquad \forall u\in \bar P, g\in \Bar G.$$

\noindent 
For $0\leqslant l<r<\infty$ if $y\in \C(l,r;\bar P )$, we 
write $l_y=l$ and $r_y=r$. Let $\C(*,*;\bar P )$ be the union of such
spaces $\C(l,r; \bar P )$. It has the standard additive structure under concatenation:
if $y$ and $y'$ are two paths with $r_y=l_{y'}$ and
$y(r_y)=y'(l_{y'})$ let $y+y'$ be the corresponding element in
$C(l_y,r_{y'}; \bar P)$.
The {\it basic} $\sigma$-algebra of $C(*,*,\bar P )$ is defined to be the pull back by $\pi$ of the usual Borel $\sigma$-algebra on $C(*,*;M^+)$.

 \noindent
   Given an equivariant diffusion operator $\B$ on $P$ consider the laws $\displaystyle{\{\PP_a^{(l,r),\B}: a\in P\}}$ as a kernel from $P$ to $\C(l,r;\bar P )$. The right action $R_g$ by $g$ in $G^+$
extends to give a right action, also written $R_g$, of $G^+$ on $\C(*,*,\bar P )$. Equivariance of $\B$ is equivalent to
$$\PP_{ag}^{(l,r),\B}=(R_g)_*\PP_a^{(l,r),\B}$$
for all $0\leqslant l\leqslant r$ and $a\in P$. Therefore $\pi_*(\PP_a^{(l,r),\B})$
depends only on $\pi(a)$, $l$, $r$ and gives the law of the
 induced diffusion $\A$ on $M$.
We say that such a diffusion $\B$ is {\it basic} if for all $a\in P$ and
 $0\leqslant l< r<\infty$ the basic $\sigma$-algebra on $\C(l,r;\bar P )$ contains
all Borel sets up to $\PP_a^{(l,r), \B}$ negligible sets, i.e. for all
$a\in P$ and Borel subsets $B$ of $\C(l,r;\bar P )$ there exists a Borel
subset $A$ of $\C(l,r,M^+)$ s.t. $\PP_a^{(l,r), \B}\big(\pi^{-1}(A)\Delta B\big)=0$.

For paths in $G$ it is more convenient to consider the space
$\tilde\C_{\id}(l,r;G^+)$ of cadlag paths $\sigma: [l,r]\to G^+$ with
$\sigma(l)=\id$ such that $\sigma$ is continuous until it leaves $G$ and stays at $\Delta$ from then on. It has a multiplication
$$\tilde \C_{\id}(s,t;G^+)\times \tilde\C_{\id}(t,u;G^+)\longrightarrow \tilde\C_{\id}(s,u;G^+)$$
$$(g, g')\mapsto g\times g'$$
where $(g\times g')(r)=g(r)$ for $r\in [s,t]$ and
 $(g\times g')(r)=g(t)g'(r)$ for $r\in [t,u]$.

Given probability measures $\Q$, $\Q'$ on $\tilde\C_{\id}(s,t;G^+)$ and  $\tilde\C_{\id}(t,u;G^+)$ respectively this determines a convolution
$\Q*\Q'$ of $\Q$ with $\Q'$ which is a probability measure on
 $\tilde\C_{\id}(s,u;G^+)$.

\begin{theorem}\label{theorem-kernel-decomposition}
Given the laws  $\{\PP_a^{(l,r),\B}: a\in P, 0\leqslant l<r<\infty\}$ of an
equivariant diffusion $\B$ over a cohesive $\A$ 
there exist probability kernels $\{\PP_a^{H,l,r}: a\in P\}$ from
$P$ to $\C(l,r;\bar P )$, $0\leqslant l < r <\infty$ and $y\mapsto\Q_y^{l,r}$,
 defined $\PP^{l,r}$  a.s. from $\C(l,r; \bar P )$ to $\tilde\C_{\id}(l,r;G^+)$ such
that
\begin{enumerate}
\item[(i)]
$\{\PP_a^{H,l,r}: a\in P\}$ is equivariant, basic and determining a 
 cohesive generator.
\item[(ii)]
$y\mapsto \Q_y^{l,r}$ satisfies
$$\Q_{y+y'}^{l_y, r_{y'}}=\Q_y^{l_y, r_y}*\Q_{y'}^{l_{y'},r_{y'}}$$
for $\PP^{l_{y}, r_y}\otimes \PP^{l_{y'},r_{y'}}$ almost all
$y$, $y'$ with $r_y=l_{y'}$.
\item[(iii)]
For $U$ a Borel subset of $\C(l,r; \bar P )$,
$$\PP_a^{l,r}(U)
=\int_{\C(l,r;\bar P)} \int_{\tilde \C(l,r;G^+)}\chi_{U}(y_\cdot\cdot  g_\cdot) \Q_y^{l,r}(dg)
\PP_a^{H,l,r}(dy).$$
\end{enumerate}
The kernels $\PP_a^{H,l,r}$ are uniquely determined as are the
$\{\Q_y^{l,r}: y\in \C(l,r; \bar P )\}$, $\PP_a^{H,l,r}$ a.s. in $y$ for all $a$
in $P$. Furthermore $\Q_y^{l,r}$ depends on $y$ only through its
projection $\pi(y)$ and its initial point $y_l$.
\end{theorem}
The proof of this theorem is as  that of
Theorem 2.5 in \cite{Elworthy-LeJan-Li-principal} (although there the processes are assumed to have no explosion).

Stochastic differential equations can be given for $(\tilde x_t)$ and $(g_t^\sigma)$ as in \S\ref{se:cond-law}, from which the decomposition can be proved via It\^o's formula; see Theorem \ref{th:flow-decom} below for details of a special case.

Proposition \ref{rose} extends results for Riemannian submersions by Elworthy-Kendall \cite{Elworthy-Kendall} and related results by Liao\cite{Liao89}. A rich supply of examples of skew-product decomposition of Brownian motions, with a general discussion, is given in Pauwels-Rogers\cite{Pauwels-Rogers}.

For a special class of derivative flows, considered as $GLM$-valued process as in \S\ref{se:deri-flow} there is a different decomposition by Liao \cite{Liao00}, see also Ruffino \cite{Ruffino}.


\section{Induced processes on vector bundles}
\label{se:VB2}
In the notation of \S\ref{se:Weitz} let $\rho: G\to L(V,V)$ be a $C^\infty$ representation with $\Pi^\rho: F\to M$ the associated bundle. A $\B$-diffusion $\{y_t: 0\leqslant t<\zeta\}$ on $P$ determines a family of $\{\psi_t: 0\leqslant t<\zeta\}$ of random linear map $W_t$ from $F_{x_0}\to F_{x_t}$,
where $x_t=\pi(y_t)$. By definition,
$$\psi_t[(y_0,e)]=[(y_t,e)].$$
Assuming $\A$ is cohesive we have the parallel translation $\parals_t: F_{x_0}\to F_{x_t}$ along $\{x_t: 0\leqslant t<\zeta\}$ determined by
our semi-connection. This is given by
$$\parals_t[(y_0, e)]=[(\tilde x_t, e)]$$
where $\tilde x_\cdot$ is the horizontal lift of $x$, starting at $y_0$.
\bigskip

\noindent
When taken together with Corollary \ref{co:derivative} the following extends results for derivative flows \index{flow!derivative} in Elworthy-Yor\cite{Elworthy-Yor}, Li\cite{flow},  Elworthy-Rosenberg \cite{Elworthy-Rosenberg}, and Elworthy-LeJan-Li\cite{Elworthy-LeJan-Li-book}.

\begin{theorem}
\label{lilac}
Let $\rho: G\to L(V;V)$ be a representation of $G$ on a Banach space $V$ and $\Pi^\rho:F\to M$ the associated vector bundle. Let $\{y_t: 0\leqslant t<\zeta\}$ be a $\B$-diffusion for an equivariant diffusion operator
$\B$ over a cohesive diffusion operator $\A$. Set $x_t=p(y_t)$ and let
$\Psi_t: F_{x_0}\to F_{x_t}, 0\leqslant t<\zeta$ be the induced transformations on $F$. Then the local conditional expectation
$\{\bar \Psi_t: 0\leqslant t<\zeta\}$, for $\bar \Psi_t=\E\{\Psi_t| \sigma\{x_s: 0\leqslant s <\zeta\}$ exists and is the solution of the covariant equation along $\{x_t: 0\leqslant t<\zeta\}$:
$${D\over \partial t}\bar \Psi_t=\Lambda^\rho\circ \bar\Psi_t$$
with $\Psi_0$ the identity map, $\Lambda^\rho: F\to F$ given by $\lambda^\rho$ in Theorem \ref{th:comp} and where ${D\over \partial t}$ refers to the semi-connection determined by $\B$.
\end{theorem}

\begin{proof}
From above and Proposition \ref{rose} we have
\begin{eqnarray*}
\Psi_t[(y_0,e)]=[(\tilde x_t \circ g_t^{\tilde x}, e)]
=[(\tilde x_t, \rho(g_t^{\tilde x})^{-1}e]
\end{eqnarray*}
and so $\parals_t^{-1}\psi_t[(y_0,e)]
=[(y_0,\rho(g_t^{\tilde x})^{-1}e)]$.
Now from the right invariance of $\G_t^\sigma$, for fixed path $\sigma$ and time $t$, we can apply Baxendale's integrability theorem for the right action 
\begin{eqnarray*}
G\times L(V;V)&\to&L(V;V)\\
(g,T)&\mapsto& \rho(g_t^\sigma)^{-1}\circ T
\end{eqnarray*}
to see $\E|\rho(g_t^\sigma)^{-1}|_{L(V;V)}<\infty$ for each $\sigma$,
$t$ and we have $\Epsilon(\sigma)_t\in L(V;V)$ given by
$$\Epsilon(\sigma)_te=\E \rho(g_t^\sigma)^{-1} e.$$
By considering 
$(1+\E |\rho(g_t^\sigma)^{-1}\parals_t^{-1} \psi_t$ for $\sigma=x_\cdot$. We see the local conditional expectation $\bar \Psi_t$ exists in $L(F_{x_0}; F_{x_t})$ and
$$\bar \Psi_t[(y_0, e)]=[(\tilde x_t, \Epsilon(x_\cdot)_te)].$$
The computation in Theorem \ref{th:comp} shows that
$${d\over dt}\parals_t^{-1}\bar \Psi_t[(y_0, e)]
= {d\over dt } [(y_0, \Epsilon(x_\cdot)_te)]
=[(y_0, \lambda^\rho(\tilde x_t)\Epsilon(x_\cdot)_t e)]$$
giving
$${D\over dt}\bar \Psi_t[(y_0, e)]
=[(\tilde x_t, \lambda^\rho(\tilde x_t)\Epsilon(x_\cdot)_t e)]
=\Lambda^\rho(x_\cdot)\bar \Psi_t[(y_0, e)]$$
as required.
\end{proof}

\begin{remark}
Theorem \ref{lilac} could also be used to identify the generator of the operator induced 
on sections of $F^*$, reproving Theorem \ref{th:comp}, since if $\phi\in \gamma F^*$ then 
$\E \phi\circ \Psi_t \chi_{t<\zeta}
=\E \phi\circ \bar \Psi_t\chi_{t<\zeta}$
if the expectations exist, by Corollary 3.3.5 of \cite{Elworthy-LeJan-Li-book}.
The extra information in Theorem \ref{lilac} is the existence of the conditional expectation. Baxendales' integrability theorem used for
this applies in sufficiently generality to give corresponding results for infinite dimensional $G$, for example in the situation arising in chapter \ref{ch:flow} below.
\end{remark}

\chapter{ Filtering with non-Markovian Observations}
\label{ch:nonlinear-filtering}
So far we have considered smooth maps $p:N\to M$with a diffusion process $u_.$ on $N$ mapping to a diffusion process $x_.=p(u_.)$  on $M$. From the point of view of filtering we have considered $u_.$ as the \emph{signal}  and $x_.$ as the \emph{observation process}. However the standard set up for filtering does not assume Markovianity of the observation process. Classically we have a signal $z_.$, a diffusion process on $\R^d$ or a more general space, and an observation process $x_.$ on some $\R^n$ given by an SDE of the form \begin{equation}\label{obs-1}
dx_t=a(t,x_t,z_t)dt+b(t,x_t,z_t)dB_t
\end{equation}
where $B_.$ is a Brownian motion independent of the signal.
To fit this into our discussion we will need to assume that the noise coefficient of the observation SDE does not depend on the signal other than through the observations, as well as the usual cohesiveness assumptions. We can take $N=\R^d\times\R^n$ and $M=\R^n$ with $p$ the projection and $u_t=(z_t,x_t)$. To reduce to our Markovian case we can use the standard technique of applying  the Girsanov-Maruyama theorem. Here we first carry this out in the general context of diffusions with basic symbols, as discussed in Section \ref {se-basic symbol} and then show how it fits in with the classical situation.  For simplicity we shall assume that the signal is a time homogeneous diffusion, and that the coefficients in the observation SDE are also independent of time.  The state spaces are taken to be smooth manifolds and the standard non-degeneracy assumptions on the observation process somewhat relaxed.

For other discussions about filtering with processes which have values in a manifold see \cite{Duncan}, \cite{Pontier-Szpirglas}, and \cite {Estrade-Pontier-Florchinger}.

\section{Signals with Projectible Symbol}
\label{se:basic-filtering}

Using the notation and terminology of Section \ref{se-basic symbol} suppose that our diffusion operator $\B$ on $N$ is conservative and descends cohesively over $p:N\to M$
so that for a horizontal vector field $b^H$ on $N$ the diffusion operator $\tilde{\B}:=\B-b^H$ lies over some  cohesive $\A$.  Choose such an $\A$ so that $\tilde{\B}$, and so $\A$, is also conservative: \emph{ we assume that this is possible}.  Also choose a locally bounded one-form $b^\# $ on $N$ with 2$\sigma^{\B}(b^\# )=b^H$. This is possible since $b^H$ is horizontal, and we can, and will, choose $b^\# $ to vanish on vertical tangent  vectors and satisfy 
\begin{equation}\label{eq-bsharp}
  b^\#_y (b^H(y))=  2\sigma^\B_y( b^\# _y,b^\# _y)=|b^H(y)|^2_y \qquad y\in N
\end{equation}
where $|b^H(y)|_y $ refers to the Riemannian metric on the horizontal tangent space induced by $2\sigma^{\A^H}$. This can be achieved by first  choosing some smooth $\tilde{b}:N\to T^*M$ such that, in the notation of equation (\ref{b1}), 
$\sigma^\A_{p(y)}(\tilde{b}(y))=b(y)$ for $y\in N$; and then taking $b^\#$ to be the pull back of $\tilde{b}$ by $p$:
$$b^\#_y(v)=\tilde{b}(y)(T_yp(v))\qquad y\in N$$


Now set $$ Z_t= \exp\{-M^\alpha_t-\frac{1}{2}\left\langle M^\alpha\right\rangle_t\} $$ for $\alpha_t(u_.)=b^\# _{u_t}$ where $u\in \C([0,T];N)$, our canonical probability space furnished with  measures $\mathbf{ P}:=\mathbf{ P}^\B$ and $\tilde{\mathbf{ P}}:=\mathbf{ P}^{\tilde{\B}}$ and corresponding expectation operators $\E$ and $\tilde{\E}$.

Here and below we are using the notation of  proposition \ref{pr:mart-1} with $M^\alpha$ etc referring to taking martingale parts with respect to $\mathbf{P}$ while $\tilde{M}^\alpha$ and $\int_0^t\alpha_s d\{y_s\}^{~}$ are with respect to $\tilde{\mathbf{P}}$.

\noindent
 From the Girsanov-Maruyana-Cameron-Martin theorem (see the Appendix, Section \ref{se-GMCM theorem}), we know that $Z_.$ is a martingale under $\mathbf{ P}$ and the two measures are equivalent with 
$$ \frac{d\mathbf{ P}_{y_0}^{\tilde\B}}{d\mathbf{ P}_{y_0}^{\B}}=Z_T.$$ 

\noindent
Suppose $f:N\to \R$ is bounded and measurable. We wish to find 
$\pi_t(f):N\to \R, 0\leqslant t\leqslant  T$ where 
$$\pi_t(f)(y_0)=\E_{y_0}\big\{f(u_t)|p(u_s),0\leqslant s\leqslant t\big\}.$$

\noindent
Following the approach due to Zakai, consider the {\it unnormalised filtering process}
$\hat \pi_t(f):N\to \R$ given by 
$$\hat \pi_t(f)(u_0)=\tilde{\E}_{u_0}\big\{f(u_t)Z_t^{-1}\; | \;p(u_s),0\leqslant s\leqslant  t\big\}.$$
\bigskip

\noindent
For completeness we state and prove the Kallianpur-Striebel formula
\index{Kallianpur-Striebel formula}, a version of Bayes' formula:

\begin{lemma} 
\label{le:Kall-Strieb}
$$\pi_t(f)(u_0)=\frac{\hat \pi_t(f)(u_0)}{\hat \pi_t(1)(u_0)} \qquad \mathbf{ P}_{u_0}-as.$$
\end{lemma}

\begin{proof}
Set $x_0=p(u_0)$. Let $g:\C_{u_0}([0,T];N)\to \R$ be $\F^{x_0}_t$-measurable. Then\begin{eqnarray}
\E_{u_0}\{f(u_t)g(u_.)\}&=&\tilde{\E}\{\frac{1}{Z_t}f(u_t)g(u_.)\}\nonumber\\
&=&\tilde{\E}\{\tilde{\E}\{\frac{1}{Z_t}f(u_t)|\F^{u_0}_t\}g(u_.)\}\nonumber\\
&=&\E\{Z_t\tilde{\E}\{\frac{1}{Z_t}f(u_t)|\F^{u_0}_t\}g(u_.)\}.
\end{eqnarray}
Thus $$\pi_t(f)(u_0)=\E\{Z_t|\F^{u_0}_t\} \hat \pi_t(f)(u_0).$$ Taking $f$ constant shows that $\E\{Z_t|\F^{u_o}_t\}\hat \pi_t(1)(u_0)=1$ and the result follows.
\end {proof}

 \noindent 
 We can now go on to obtain the analogue of the Duncan-Mortensen-Zakai (DMZ) equation \index{Duncan-Mortensen-Zakai equation}%
  for the unnormalized filtering process, using the results of Section \ref{se:cond-law} on conditional laws:
 \begin{theorem}
 \label{th:DMZ}
 For any $C^2$ function $f:N\to \R$, under $\tilde{\PP}$,
 
 \begin{equation}
 \label{eq:DMZ}
\begin{array}{ll}
 \hat \pi_t f(u_0)
=&f(u_0)+\int_0^t \hat \pi_s \big(\B f \big)(u_0)\;ds
 +\int_0^t \hat \pi_s \big(fb^{\#} (-)h_{-}\big) (u_0)d \{x_s\} \\
&+\int_0^t \hat \pi_s  \big(df_{-}h_{-}\big) (u_0) d\{x_s\};
\end{array} 
\end{equation}

 \begin{equation}
\begin{array}{ll}
\hat \pi_t f(u_0)=&f(u_0)+\int_0^t \hat \pi_s \big(\B f \big)(u_0)\;ds
 +\int_0^t \langle \hat \pi_s (fb) (u_0),  d \{x_s\}\rangle_{x_s} \\
&+\int_0^t \hat \pi_s  \big(df_{-}h_{-}\big) (u_0) d\{x_s\}.
\end{array}
\end{equation}
where $x_s=p(u_s), 0\leqslant s\leqslant \infty$ is  the projection to $M$ of the canonical process from $u_0$ on $N$, and $h$ the horizontal lift map for the induced semi-connection. 

\noindent
Using  an alternative notation:
\begin{equation}
\label{eq:DMZ-2}
\hat{ \pi}_t f=\hat {\pi}_0f +M_t^{ \pi \big( fb^\#\circ h_{u_.}\big), \A} +M_t^{ \hat \pi_. \big( df \circ h_{u_.}\big),\A}+\int_0^t \hat{ \pi}_s (\B f)ds.
\end{equation}

\end{theorem}

\begin{proof}
Since we are working with $\tilde{\PP}$ we will write $M^{b^{\#}}$ for $M^{b^{\#},\tilde{\B}}$, etc. Also $Z_t^{-1}$ satisfies:
\begin{equs}
dZ_t^{-1}=Z_t^{-1}dM^{b^{\#}}_t
\end{equs}
while
$$df(u_t)=dM^{df}_t+\tilde{\B}(f)(u_t)dt $$
giving
\begin{eqnarray*}
d\big(Z_t^{-1}f(u_t)\big)=&&Z_t^{-1}dM^{df}_t+Z_t^{-1}\tilde{\B}(f)(u_t)dt\\&&+f(u_t)Z_t^{-1}dM^{b^{\#}}_t+Z_t^{-1}+df_{u_t}(b^H(u_t))dt
\end{eqnarray*}
since $ dM^{df}_tdM^{b^{\#}}_t =\sigma^{\tilde{\B}}\big(df_{u_t},b^{\#}\big)=df_{u_.}(b^H(u_t)).$
Thus
\begin {equs}
d\big(Z_t^{-1}f(u_t)\big)=Z_t^{-1}dM^{df}_t+Z_t^{-1}\B(f)(u_t)dt+f(u_t)Z_t^{-1}dM^{b^{\#}}_t+Z_t^{-1}.
\end{equs}
We can now take conditional expectations using proposition \ref{pr:mart-4} since $\B-\LL_{b^H}$ is over the cohesive operator $\A$ to complete the proof.
\end{proof}



%
\begin{lemma}
There are the following formulae for angle brackets:
\begin{equation}
d\langle\hat{\pi}(1)\rangle_t=\langle \hat{\pi}_t ( b)  ,\hat{ \pi}_t ( b)\rangle_{x_t}^E dt
\end{equation}
\begin{equation}
d\langle\hat{\pi}(1),\hat{\pi}(f)\rangle_t=\langle \hat{\pi}_t ( fb)  ,\hat{ \pi}_t ( b)\rangle_{x_t}^E dt +\hat{\pi}_t ( df \circ h_{u_.})  \circ \hat{ \pi}_t ( b(u_.))dt
\end{equation}
\end{lemma}
\begin{proof}
From the previous theorem
$$
\langle\hat{\pi}(1),\hat{\pi}(f)\rangle dt=\big(dM_t^{\hat \pi ( fb^\#\circ h), \A} +dM_t^{ \hat \pi_. ( df \circ h_{u_.}),\A}\big)dM_t^{ \pi ( b^\#\circ h), \A}$$
$$=2\sigma^{\A}\big(\hat{ \pi}_t ( fb^\#\circ h), \hat{\pi_t} ( b^\#\circ h)\big)dt+2\sigma^\A\big(  \hat{ \pi}_t ( df \circ h_{u_.})  ,\hat{ \pi}_t ( b^\#\circ h)\big)dt$$
$$=\langle \hat{\pi}_t ( fb)  ,\hat{ \pi}_t ( b)\rangle_{x_t} dt +\hat{\pi}_t ( df \circ h_{u_.})  \circ \hat{ \pi}_t ( b(u_.)) dt$$
since for any one form $\phi$ on $M$ we have:
\begin{eqnarray*}
\sigma^\A\big( \phi ,\hat{\pi}_t ( b^\#\circ h)\big)&=&\hat{\pi}_t\big(\langle \phi|_E, b^\#\circ h\rangle^{E^*}_.\big)\\
&=&\frac{1}{2}\hat{\pi}_t\big( \phi( b) \big)\\
&=&\frac{1}{2}\phi(\hat\pi_t(b)).
\end{eqnarray*}
This gives the second formula, from which comes the first.
\end{proof}

We can now give a version of Kushner's formula in our context:
\begin{theorem} In terms of the probability measure $\tilde{\PP}$

\begin{equs} 
\label{th:kushner}
\pi_tf=\pi_0f+\int_0^t\pi_s\B(f)ds+\int_0^t\pi_s\big(df\circ h_{u_.}\big)\left[d\{x_s\}-\pi_s(b(u_.))_s ds\right]\\
+\int_0^t\langle\pi_s(fb)-\pi_s(f) \pi_s(b),d\{x_s\}-\pi_s(b)\rangle_{x_s}.
\end{equs}
\end{theorem}
\begin{proof}
From the definition and then Ito's formula:

\begin{eqnarray*}
d \pi_t(f)&=&d\left(\frac{\hat{\pi}_t(f)}{\hat{\pi}_t(1)}\right)\\
&=& \frac{d\hat{\pi}_t(f)}{\hat{\pi}_t(1)}-\frac{\hat{\pi}_t(f)d\hat\pi_t(1)}{(\hat{\pi}_t(1))^2}-\frac{d\hat{\pi}_t(f)d\hat\pi_t(1)}{(\hat{\pi}_t(1))^2}\\
&&+\frac{\hat{\pi}_t(f)d\hat\pi_t(1)d\hat\pi_t(1)}{(\hat{\pi}_t(1))^3}.
\end{eqnarray*}
Now substitute in the second formula of Theorem \ref{th:DMZ} and use the previous lemma.
\end {proof}

Note that $\hat \pi_t(f)$, $b$,  and $\tilde{\mathbf{ P}}$, depend on the choice of $\A$. We would like to have a version of formula \ref{th:kushner} which is independent of such choices. First note  that if $\B-b^H_1$ is over $\A_1$, 
and $\B-b^H_2$ is over $\A_2$, then the difference of the
two vector fields on $N$ descends to a vector field on $M$: if $g:M\to \R$ is smooth and $\tilde{g}=g\circ p:N\to\R$ then
$$(b^H_2-b^H_1)\tilde g=(\B-b^H_1)\tilde f-(\B-b^H_1)\tilde g
=(\A_1-\A_2)g.$$
Therefore if we set $b_0(z)=T_yp(b^H_2(y)-b^H_1(y))$ for $p(y)=z$, $z\in M$ then $\A_1=\A_2+\LL_{b_0}$, and by Remark \ref{re:Pchange}
 \begin{equation}\label{eq:Achange}
d\{x_s\}^{\A_2}=d\{x_s\}^{\A_1}+b_0ds
\end{equation}
From this we see immediately that the symbols  $d\{x_s\}-\pi_s(b)ds$, and $\pi_s(fb)-\pi_s(f) \pi_s(b)$ in  formula \ref{th:kushner} are in fact independent of the choice we made of $\A$. To relate to now classical concepts we  next discuss the first of these in more detail.

\section{Innovations and innovations processes}
Keeping the notation above, for $\alpha\in L^2_\A$, so $\alpha_t\in T^*_{x_t}M$ for $0\leqslant t<\infty $, define a real valued process $I^\alpha_t:0\leqslant t<\infty$, the  $\alpha$-\emph{innovations process}\index{innovations process}
 by
 \begin{equation}
I^\alpha_t=\int_0^t\alpha_s\left(d\{x_s\}^\A-\pi_sb(u_.)ds\right)
\end{equation}
A generalisation of a standard result about innovations processes is:
\begin{proposition}\label {pr:innovations}
The process $I^\alpha_.$ is independent of the choice of $\A$. Under $\PP^{\B,u_0}$ it is an $\F^{x_0}_*$ martingale.
\end{proposition}
\begin{proof}
The observations just made show it is independent of the choice of $\A$. It is clearly also
adapted to $\F^{x_0}_*$.  To prove the martingale property note first that by Proposition \ref{pr:mart-2}  and formula (\ref{eq:Achange})
\begin{eqnarray*}
\int_0^t\alpha_s d\{x_s\}^\A&=&\int_0^tp^*(\alpha_s)d\{u_s\}^{\B-\LL_{b^H}}\\
&=&\int_0^tp^*(\alpha_s)d\{u_s\}^{\B}-\int_0^t p^*(\alpha_s)b^H(u_s)ds\\
&=&\int_0^tp^*(\alpha_s)d\{u_s\}^{\B}-\int_0^t \alpha_s(b(u_s))ds.
\end{eqnarray*}

From this we see that if $0<r<t$ and $Z\in \sigma\{x_s:0\leqslant s\leqslant r\}$ then 
\begin{eqnarray*}
&&\E^\B\chi_Z\Bigg\{\int_r^t\alpha_s\left(d\{x_s\}^\A-\pi_sb(u_.)\right)ds\Bigg\}
\\&&=\E^\B\chi_Z\Bigg\{\int_r^t\alpha_s\left(b(u_s)-\pi_sb(u_.)\right)ds\Bigg\}=0
\end{eqnarray*}
giving the required result.
\end {proof}
If we fix a metric connection, $\Gamma$, on $E$, as described in Example \ref{ex:BM} we can take the canonical Brownian motion, $B^{\Gamma,\A}$ say,  on $E_{x_0}$ determined by $\A$ and $\Gamma$. Then, by equation (\ref{eq:mart-6}), we can write
$d\{x_s\}^\A-\pi_s(b(u_.))ds=\paral_s dB^{\Gamma'\A}-\pi_s(b(u_.))ds$. In terms of the the $\PP$ Brownian motion,  $B^\Gamma$, on $E_{x_0}$, which is the martingale part under $\PP$ of the $\Gamma$- stochastic anti-development of $x_.$ we can  define an $E_{x_0}$-valued process,  $z^\Gamma_t:0\leqslant t<\infty$, by 
\begin{equation}
z^\Gamma_t=B^{\Gamma}_t+\int_0^t(\paral_s)^{-1} (b(u_s)-\pi_s(b(u_.))ds.
\end{equation}

A candidate for the \emph{innovations process} of our signal -observation system is the stochastic development , $\nu^\Gamma_.$ say, of $z^\Gamma_.$.  under $\Gamma$. This can be defined by using the canonical sde on the orthonormal frame bundle of $E$,
namely $$d\tilde{\nu}_t=X( \tilde{\nu}_t)(\tilde{\nu_0})^{-1}\circ dz_t$$
for a fixed frame  $\nu_0$ for $E_{x_0}$. Here $$X(\mu)(e)=h^\Gamma_\mu(\mu(e)).$$
for $\mu:\R^p\to E_m$ a frame  in at some point $m\in M$, and $e\in \R^p$, for $p$ the fibre dimension of $E$. The process $\nu^\Gamma_.$ is then the projection of $\tilde{\nu}_.$ on $M$. For example see \cite {Elworthy-Stflour}. It will satisfy the Stratonovich equation
\begin{equation}
d\nu^\Gamma_t=\paral_t\circ dz_t
\end{equation}
where the parallel translation is now along the paths of $\nu^\Gamma_.$. 
Let $\Theta:C_0(M)\to C_0(M)$ be the map given by $\Theta_(\sigma)_t=\nu^\Gamma(\sigma)_t$, treating $z^\Gamma_.$ as defined on $C_0(M)$.  Let $\mathcal{D}=\mathcal{D}^\Gamma: C_0(T_{x_0}M\to C_{x_0}M$ be the stochastic development using $\Gamma$ with inverse $\mathcal{D}^{-1}$. We will continue to assume that there is no explosion so that these maps are well defined. For example,
$$z(x_.)=\mathcal{D}^{-1} \Theta (x_.).$$
We  define a semi-martingale, on $M$ to be a $\Gamma$\emph{-martingale}\index{$\Gamma$-martingale}
 if it is the stochastic develoment using $\Gamma$ of a local martingale, see the Appendix, Section\ref{se:semi-mart}.
\begin {theorem}
For each metric connection $\Gamma$ on $E$ the innovations process $\nu^\Gamma$ is a $\Gamma$-martingale. If $\Gamma$ is chosen so that the $\A$-diffusion process is a $\Gamma$-martingale under $\PP^{\A}$ then for $\alpha:[0,\tau)\times C_{x_0}M\to T^*M$ which is predictable and lives over $x_.$, provided the integrals exist, 
\begin{equation}
I^\alpha\circ\Theta(x_.)=\big(\Gamma\big)\int_0^.\alpha(\nu^\Gamma(x_.)_.)_s \!d\nu^\Gamma(x_.)_s-\int_0^.\alpha(x_.)_s \overline{b}(x_.)_s ds
\end{equation}
where $\overline{b}(-)_s: C_{x_0}\to TM$ is the conditional expectation, 
$$\overline{b}_s=\E\{b(u_s)| p(u_.)=x_.\},$$
and has $\overline{b}(x_.)_s\in T_{x_s}$ almost surely for all $s$.
\end{theorem}
\begin{proof}
The fact that $\nu^\Gamma$ is a $\Gamma$-martingale is immediate from the definition and Proposition
\ref{pr:innovations}.  To prove the claimed identity note that our extra assumption on $\Gamma$ implies that  $\parals_s^{-1}d\{x_s\}^\A=d(\mathcal{D}^{-1}(x_.))_s$. Therefore\begin{equation}
I^\alpha(x_.)=\int_0^.\alpha_s (x_.) \parals_s d\mathcal{D}^{-1}(x_.)_s-\int_0^.\alpha_s(x_.)\bar{b}(x_.)_s ds
\end{equation}
while by definition
\begin{equation}
\big(\Gamma\big)\int_0^.\alpha_sd\nu^\Gamma_s(x_.)=\int_0^. \alpha_s(\nu^{\Gamma}_s(x_.)) \parals_s^{\nu^\Gamma_.(x_.)} d\big(\mathcal{D}^{-1}(\nu^\Gamma(x_.))\big)_s
\end{equation}
where the superscript on the parallel translation symbol indicates that it is along the paths $\nu^\Gamma_.(x_.)$. Our identity follows.
\end{proof}
\begin {remark}\label{re:innov}
\begin{enumerate}
\item[{(1)}]  For $\Gamma$ such that the $\A$-process is a $\Gamma$ martingale we can easily see that 
$\Theta$ has an adapted inverse. Indeed its inverse is defined  almost surely by $$ \Theta^{-1}=\mathcal{D}\circ {\mathrm Mart}^{\PP^A}\circ\mathcal{D}^{-1}$$
where ${\mathrm Mart}^{\PP^A}$ denotes the operation of taking the martingale part under the probability measure $\PP^\A$. 
\item[{(2)}]  If we are given a  connection $\Gamma$ on $E$ we could make our choice of  $\A$ so that its diffusion process gives a $\Gamma$ martingale. This specifies $\A$ uniquely and  might be more natural sometimes, for example in the classical case with  $M=\R^n$.
\item[{(3)}] The results and earlier discussion still hold if $\Gamma$ is not a metric connection. However then $B_.^{\Gamma,\A}$ cannot be expected to be a Brownian motion. The connection could even be on $TM$ rather than on $E$ in which case  $B_.^{\Gamma,\A}$ will be a local martingale  in $T_{x_0}M$. This will be a natural procedure when $N=\R^n$, using the standard flat connection.
\end{enumerate}
\end{remark}

\section{Classical Filtering} 
For an example of the situation treated above consider a signal process $(z_t, 0\leqslant t\leqslant T)$ on $\R^d$ satisfying an SDE \begin{equation}
dz_t=V(z_t, x_t)dW_t+\beta(z_t, x_t)dt 
\end{equation}
with $(x_t, 0 \leqslant t\leqslant T)$, the observation process, taking values in $\R^n$ and satisfying:
\begin{equation}
dx_t= X^{(1)}(x_t)dB_t+ X^{(2)}(x_t)dW_t+b(z_t, x_t)dt.
\end{equation}
Here $B_.$ and $W_.$ are independent Brownian motions of dimension $q$ and $p$ respectively. We then take $N=\R^d\times \R^n$ and $M=\R^n$, with $p:N\to M$ the projection. We set $u_t=(z_t,x_t)$ so that 
\begin{equation}
\begin{array}{ll}
\B f(z,x)=&\frac{1}{2}D^2_{1,1}f(V^i(z,x),V^i(z,x)) +D_1f(\beta(z,x))  \\ 
&+\frac{1}{2}D^2_{2,2}f(X^{(1),i}(x),X^{(1),i}(x))+\frac{1}{2}D^2_{2,2}f(X^{(2),j}(x), X^{(2),j}(x))\\
&+ D_2f(z,x)(b(z,x))+D^2_{1,2}f(z,x)(V^i(z,x),X^{(1),i}(z, x))
\end{array}
\end{equation}
using the repeated summation convention where $i$ goes from $1$ to $p$ and $j$ from $1$ to $q$, with the $V^j$ referring to the components of $V$ and similarly for $X^{(1),i}$ and $X^{(2),j} $. Also $D^2_{l,m}$  refers to the second partial Frechet derivative, mixed if $l\not=m$, etc.

The filtering problem would be to find $\E\{g(z_t)\;|\; x_s:0\leqslant s\leqslant t\}$ for suitable $g:\R^d\to \R$. This would fit in with the discussion above by defining $f:\R^d\times \R^n\to \R$ by
$f(z, x)=g(z)$. Note that we have allowed feedback from the signal to the observation; usually only the special case where $V$ and $\beta$ are independent of $x$ is considered. Also we have allowed the noise driving the signal to also affect the observations (``correlated noise"). This can give a non-trivial connection,  in which case the terms involving horizontal derivatives of $f$ will not vanish even for $f$ independent of $x$. This vanishing would occur otherwise (i.e. for uncorrelated noise) so that in that case
the formula in Theorem \ref{th:DMZ} reduces to the usual DMZ equation, for example as in \cite{Pardoux-Nato} or \cite {Pardoux-Stflour}.

Our basic assumptions are smoothness of the coefficients, non-explosion (for simplicity of exposition), and the cohesiveness of our observation process. By the latter we mean that  for all $x\in \R^n$ and $z\in \R^d$ the image of  the map $(e^1,e^2)\mapsto X^1(x)(e^1)+X^2(X)(e^2)$ from $\R^q\times \R^p$ to $\R^n$ contains $b(z,x)$ and has  dimension independent of $x$. Some bounds are needed on $b$ to ensure the existence of its conditional expectations.

To carry out the procedure for the signal and observation given above we must first identify the horizontal lift operator  determined by  $\B$. For this for each $x\in M$ let $Y_x:\R^n\to\R^{ p+q}$ be the inverse of the restriction of the map $(e^1,e^2)\mapsto X^1(x)(e^1)+X^2(X)(e^2)$, from $\R^q\times \R^p$ to $\R^n$ ,  to the orthogonal complement of its kernel. Then from Lemma \ref{factorization-connection} we see that the horizontal lift $h_u: \R^n\to \R^d\times \R^n$ is given by 
\begin{equation}
h_u(v)=(V(z,x)\circ Y_x(v), v) \quad u=(z,x)\in \R^d\times \R^n.
\end{equation}

\noindent
A natural choice of $\A$ is 
$$\A(f)(x)=\frac{1}{2}D^2_{2,2}f\left(X^{(1),i}(x),X^{(1),i}(x)\right)
+\frac{1}{2}D^2_{2,2}f\left(X^{(2),j}(x), X^{(2),j}(x)\right).$$
Having done that  the `$b$' of our general discussion is just the drift $b: \R^d\times \R^n\to \R^n$ of our observation's stochastic differential equation. Moreover for suitable $T^*M$-valued processes $\alpha_.$ we have the $\alpha$-innovations process
$$ I^\alpha_t= \int_0^t \alpha(x_s)\left(X^{(1)}(x_s)dB_s+X^{(2)}(x_s)dW_s\right)
+\int_0^t \alpha(x_s)\left(b(z_s,x_s)-\bar{b}(x)_s\right)ds,$$
where $\bar{b}(\sigma)=\E^\B\{b(z_s,x_s)\;|\;x_r=\sigma_r \quad 0\leqslant r\leqslant s\}$.

From Theorem \ref{th:kushner} , Kushner's formula , given smooth $g:\R^d\to \R$, one has \begin{eqnarray*}
\pi_t g=&g(z_0)+\int_0^t\big( \pi_s \frac{1}{2}D^2_{1,1}g(-)(V^i(-),V^i(-)) +\pi_s D_1g(-)(\beta(-))\big)ds\\
&+I^{\pi_s\big(dg(-)V(-)\circ Y\big)}_t+I^{\langle \bar{gb}_s- \bar{g}_s \bar{b}_s,-\rangle_{x_s}}_t .
\end{eqnarray*}
This can be compared, for example, with the formula given in the remark on page 85 of \cite{Pardoux-Stflour}, following the proof of Proposition 2.2.5 there. Alternatively see \cite {Pardoux-Nato}.

\noindent
Using the standard flat connection of $R^n$ we get the innovations process given by 
$$\nu_t= x_0+\int_0^t \left(X^{(1)}(x_s)dB_s+X^{(2)}(x_s)dW_s\right)+\int_0^t \left(b(z_s,x_s)-\bar{b}(x)_s\right)ds.$$

\pagebreak

\section{Examples}
Consider the stochastic partial differential equation \index{stochastic partial differential equation}%
 on $L^2([0,1];\R^p)$:
$$du_t(x)=\Delta u_t(x)+\sum_{i=1}^m\Phi_i(x,u_t(x))dB_t^i$$
where $(B_t^i)$ are independent Brownian motions. For $p>1$ it can be considered as a system of equations.  One natural question is
to find the law of $u_t$ given that of $u_s(x_0),0\leqslant s\leqslant t$ for some given point $x_0$,
or to find the conditional law of $u_t$ given $u_s(x_0), 0\leqslant s \leqslant t$. Here we indicate briefly how the approach we have been following may sometimes be applied to this or similar problems. For simplicity we take $p=1$, so our ``observations" process is one dimensional; $M=\R$. 

Let $y_t=u_t(x_0)$. It satisfies:
$$dy_t=(\Delta u_t)(x_0)dt+\sum_{i=1}^m\Phi_i(x_0, y_t)dB_t^i.$$
Because of the drift term we cannot expect this to be Markovian so we will have to remove  the term $(\Delta u_t)(x_0)dt$ by a Girsanov transformation. 

 Let $(e_i)$ be the standard  orthonormal base of 
$\R^m$. Define $$ \Phi: L^2([0,1];\R)\times \R^m\to L^2([0,1];\R)$$ and 
$$\tilde{\Phi}:
R\to \mathcal{L}(\R^m;\R)$$ by $$\Phi(u)(e)(x)= \sum_{i=1}^m\Phi_i(x,u(x))\langle e,e_i \rangle$$
and $$\tilde \Phi(z)(e):=\sum_{i=1}^m\Phi_i(x_0,z)\langle e, e_i\rangle,$$
 respectively.
Consider $T_z\R $, identified with $\R$ and furnished with the metric induced by $\tilde \Phi(z)$:
$$\langle v_1, v_2\rangle_z={v_1v_2\over \sum_{i=1}^m (\tilde\Phi_i(z))^2}.$$
To have cohesivity and to be able to apply the Girsanov-Maruyama-Cameron-Martin theorem this must be well defined, i.e. the denominator must never vanish, and it must determine a non-explosive Brownian motion.  If these conditions hold, we still  have to be sure that the Girsanov transformed S.P.D.E has solutions existing for all time and that we can apply the martingale method approach used in the proof of \ref {th-GMCM}. Alternatively we can try to apply one of the standard tests to show that the local  martingale which arises is a true martingale. 
First we apply Lemma \ref{factorization-connection} to obtain the horizontal lift map. For this we need the dual map 
$\tilde\Phi^*(z): \R\to \R^m$ is given by: $$\Phi^*(z)(1)={1\over \sum_{i=1}^m (\Phi_i(x_0, z))^2} \sum_{j=1}^m \Phi_j(x_0, z)e_j.$$
Then from equation (\ref{eq:hor-lift}) the horizontal lift $h_u: T_{u(x_0)}\R \to L^2([0,1];\R)$ at a function $u$ is given by 
$$ h_u(1)(x)=\Phi(x,u(x))\circ \tilde\Phi^*(u(x_0)).$$
In particular a natural choice of drift $b^h$ to remove by the Girsanov-Maruyama-Cameron-Martin theorem, namely  $ b^h(u)=h_u(\triangle u(x_0))$, is given by 
\begin{equation}
b^h(u)(x)={\sum_{j=1}^m \Phi_j(x_0, u(x_0))\Phi_j(x_0,u(x)) \over \sum_{k=1}^n (\Phi_k(x_0, u(x_0)))^2} \triangle u(x_0).
\end{equation}
Making the change of probability to $\tilde{\PP}$ we see that our SPDE becomes
$$du_t(x)=\Delta u_t(x)- {\sum_{j=1}^m \Phi_j(x_0, u_t(x_0))\Phi_j(x,u_t(x)) \over \sum_{k=1}^n (\Phi_k(x_0, u_t(x_0)))^2} \triangle u_t(x_0)               +\sum_{i=1}^m\Phi_i(x,u_t(x))d\tilde{B}_t^i$$
for new, independent Brownian motions $\tilde {B}^1,...\tilde {B}^m$ and has the decomposition
 \begin{eqnarray*}
&&du_t(x)
=[{\sum_{j=1}^m \Phi_j(x_0, u_t(x_0))\Phi_j(x,u_t(x)) \over \sum_{k=1}^n (\Phi_k(x_0, u_t(x_0)))^2}\Phi_i(x_0, u_t(x_0))\tilde{B}_t^i]\\
&& + [\left(\triangle u_t(x)-{\sum_{j=1}^m \Phi_j(x_0, u_t(x_0))\Phi_j(x,u_t(x)) \over \sum_{k=1}^n (\Phi_k(x_0, u_t(x_0)))^2} \triangle u_t(x_0)\right)dt\\
&&+\sum_{i=1}^m\left( \Phi_i(x, u(x)_t)-{\sum_{j=1}^m \Phi_j(x_0, u_t(x_0))\Phi_j(x,u_t(x)) \over \sum_{k=1}^n (\Phi_k(x_0, u_t(x_0)))^2}\Phi_i(x_0, u_t(x_0))\right )d\tilde{B}_t^i].
\end{eqnarray*}
In this decomposition the term in the first square brackets relates to the horizontal lift of the $\A$-process , while that in  the second is the vertical component. They are independent (under $\tilde{\PP}$), given $u$ at $x_0$.

We could continue by applying the Kallianpur -Striebel formula, Lemma \ref{le:Kall-Strieb} or go directly to our version of Kushner's formula, Theorem \ref{th:kushner}. In that formula the operator $\B$ will be the infinite dimensional diffusion operator on $L^2([0,1];\R)$ which is the generator of the solution of our SPDE, so there are extra analytical problems.  However there are cases where the situation is fairly straightforward. For example:

\begin{enumerate}
\item[{(1)}] $\Phi_i(z,u)=\phi_i(z)$, where the vector $\{\phi_1(z), \dots, \phi_m(z)\}$ never vanishes for any $z$. In this case $y_t$ is basically Gaussian. 
\item[{(2)}]  $\Phi(z,u)=u$ with one dimensional noise $B_t$, in which case the solution of the SPDE is $u_t(x)={1\over \sqrt{2\pi}t} e^{-{x^2\over 2t}}e^{B_t- {t^2\over 2}}$.
\end{enumerate}

\chapter{The Commutation Property}
\label{ch:comm}

In certain cases the filtering is in a sense trivial: the process decomposes into the observable
and an independent process. From the geometric point of view this means the commutation of the vertical operator $B^V$ and the horizontal operator $\A^H$. See  
Theorem \ref{th:equivs} below.

\noindent
For $p$ a Riemannian submersion (defined in Chapter \ref{ch:Riem.sub}  below) with totally geodesic fibres and $\B$ the Laplacian, 
Berard-Bergery \& Bourguignon \cite{Berard-Bergery-Bourguignon} show that $\A^H$ and $\B^V$ commute. Their proof is based on the result of R.Hermann \cite{Hermann}

\begin{theorem}\label{th:hermann}[R.Hermann]
 A Riemannian submersion $p:N\to M$ has totally geodesic fibres iff the Laplace-Beltrami operator of $N$ commutes with all  Lie derivations by horizontal lifts of vector fields on $M$.
 \end{theorem}
 From this, and the H\"ormander form\index{H\"ormander!form} representation of $\A^H$, it follows immediately that $\A^H$ with $\B^V$ will commute in their situation. In this section we consider some extensions of this and their consequences. 

First, for $p:N\to M$ with a diffusion operator $\B$ over a cohesive $\A$, as usual, we will say that a vector field on $N$ is \emph{basic} if it is the horizontal lift of a section of $E$.
From our H\"ormander form\index{H\"ormander!form} representation of  $\A^H$ we  get the following extension of  Berard-Bergery\& Bourguignon's result:
\begin{theorem}
\label{theorem:comm1}
For a diffusion operator $\B$ over a cohesive diffusion operator $\A$ the following are equivalent:
\begin{itemize}
\item{[i]}
$\B^V$ commutes with all Lie derivations by smooth basic vector fields of $N$;
\item{[ii]}the operators $\B$, $\B^V$, and $\A^H$ commute (on $C^4$ functions);
\item{[iii]} the operator  $\B^V$ commutes with the horizontal lifts of the vector fields which appear in one H\"ormander form\index{H\"ormander!form} representation of $\A$.

\end{itemize}
\end{theorem}

\begin{proof}
It is clear that [i] implies [iii], and [iii] implies [ii]. 
To show [ii] implies [i] observe that every
section of $E$ has the form $ \sigma^{ \A}\left( \sum_1^m \lambda^jdf_j\right)$ since every one form on $M$ has can be written as  $\sum_1^m \lambda^jdf_j$ for $\lambda^j:M\to \R$ and $f_j:M\to\R$ and some integer $m$. By definition of the connection this shows that every basic vector field on $N$ has the form
$\sum_1^m\lambda^j\dot p\sigma^{ \A^H}(p^*df_j)$. It will therefore suffice to show that if [ii] holds then $ \B^V$ commutes with Lie differentiation by $ \lambda  \dot p\sigma^{ \A^H}(p^*df)$ for all smooth $\lambda, f:M\to \R$.

For this assume [ii] holds and take a smooth $g:N\to \R$. By definition of the symbol and Remark \ref{re:vertical}:
\begin{eqnarray*}
2\B^Vdg\left(\lambda  \dot p\sigma^{ \A^H}(p^*df)\right)
&=&2\lambda\dot p \B^V dg\left(\sigma^{ \A^H}(p^*df)\right)\\
&=&\lambda\dot p \B^V \left(\A^H(f\dot p g)-f\dot p\A^H(g)-g\A^H (f\dot p)\right)\\
&=&\lambda\dot p \left( \A^H (f \dot p)\B^Vg-f\dot p\A^H \B^Vg -(\A f)\dot p \B^Vg\right )\\
&=& 2 \lambda\dot p d(\B^Vg)\sigma^{ \A^H}(p^*df)\\
&=&2d(\B^Vg)\sigma^{ \A^H}(\lambda \dot p p^*df)
\end{eqnarray*}
as required.
\end{proof}
For the special case of an equivariant diffusion on a principal bundle as considered in Chapter \ref{ch:deco-2} we can obtain a working criterion for commutativity: see also
Example \ref{ex:parallel-R}.

\begin{corollary}\label{co:equ-comm}
In the notation of Theorem \ref{th:op-deco} commutativity of $\B^V$ and $\A^H$ holds if  and only if both $\alpha$ and $\beta$ are constant along all horizontal curves. This holds if and only if $\A^H(\alpha^{i,j})=0$ and $\A^H(\beta^k)=0$ for all $i,j,k$.
\end{corollary}
\begin{proof}
First note that each vector field $A^*_k$ commutes with all basic vector fields. Indeed if $V$ is basic it is equivariant and so $$ (R_{\exp tA_k})_\star (V)=V \qquad t>0.$$ 
Differentiating in $t$ at $t=0$ gives the required commutativity. Thus the operators $\Lo   _{A_k^*}$ are invariant under flows of basic vector fields and so for $\B^V$ to commute with basic vector fields  the coefficients $\alpha$ and $\beta$ must be constant along their flows. By the theorem this gives the first result since any horizontal curve can be considered as an integral curve of a (possible time dependent) basic vector field.

Clearly, from the H\"ormander form of $\A^H$, if this holds  both $\alpha$ and $\beta$ are $\A^H$-harmonic. The converse holds since from above $\A^H$ commutes with all of the vertical vector fields $\Lo^*_{A_k}$.
\end{proof}
The Corollary is applied to derivative flows in
Example \ref{ex:parallel-R} of Section \ref{se:horizontal-lift} below.

Hermann proved that a Riemannian submersion with totally geodesic fibres has the natural structure of a fibre bundle with group the isometry group of a typical fibre. 
\begin{theorem}[Hermann] \label{hermann2}
If N is a complete Riemannian manifold and $\phi: N\to M$ is a $C^\infty$ Riemannian submersion then $\phi$ is a locally trivial fibre space. If in addition the fibres of $\phi$ are totally geodesic submanifolds of $N$, $\phi$ is a fibre bundle with structure group the Lie group of isometries of the fibre.
\end{theorem}
An analogous result given the hypothesis of theorem \ref{theorem:comm1} together with some completeness and hypoellipticity conditions is proved in Theorem \ref{th:equivs} below.

Before that we consider when the associated semi-groups commute.

\section{Commutativity of Diffusion Semigroups}

It is well known that in general the commutativity of two diffusion generators (on $\C^4$ functions) does not imply that of their associated semi-groups. One reference is \cite{Reed-Simon} page 273 where an example they ascribe  to Nelson is given. Here is a minor modification of that construction:

Cut  $\R^2$ along the positive $x$-axis. Take a copy $A$, say, of $(0,\infty)\times (-\infty,0]$ and glue it along the cut to the upper part of the cut plane, identifying $(0,\infty)\times\{0\}$ in $A$ with the positive $x$-axis. Similarly glue a copy ,$B$, of $(0,\infty)\times (0,\infty)$ along the cut to the lower part of the cut plane. This gives a version of the plane  but with two copies of the upper and lower quadrants, and with the origin missing.  On this we have naturally defined vector fields $X^1$ given by $\frac{\partial}{\partial x}$ and $X^2$ given by $\frac{\partial}{\partial y}$. These certainly commute. However their associated semi-groups do not, as can be seen by starting at the point $(-1,-1)$ moving along the $X_1$-trajectory for time $2$ and then along the $X^2$ trajectory for the same amount of time. We end up at the point $(1,1)$ of copy $B$. However if we had changed the order of the vector fields we would be at $(1,1)$ of copy $A$. 
A more geometrically satisfying construction would be, as Nelson, to use the double covering of the punctured plane as state space with similarly behaved vector fields.
Here is an easy positive result: 

\begin {proposition}
\label{theorem:comm2}
 Let $A_1$ and $A_2$ be diffusion operators with associated semi-groups $\{P^1_t\}_{t>0}$ and $ \{P^2_t\}_{t>0}$ acting as strongly continuous semi-groups on a Banach space $E$ of functions which contains the $C^2$ functions with compact support. Let $\mathcal{G}_1$ and $\mathcal{G}_2$ be the corresponding generators , (closed extensions of the restrictions of $A_1$ and $A_2$ to the space of $C^2$ functions with compact support).    Assume there is a core $\mathcal{C}_2$  for $\mathcal{G}_2$ consisting of bounded $C^\infty$ functions such that for $f\in \mathcal{C}_2$:
\begin{itemize}
\item[{[i]}] For all $t>0$ the function $P^1_tf$ is $C^4$.  

\item[{[ii]}]  
$A_2\frac{1}{t}(P^1_t f-f)$ is uniformly bounded in $t\in (0,1)$ and in space, and it converges pointwise to $A_2A_1P_t^1f$as $t\to0+$. 
\item[{[iii]}]  
$A_2P^1_tf$ is uniformly bounded in $t\in (0,1)$ and in space.
\end{itemize}
Then commutativity of $P^1_t$ with $P^2_s$, $0\leqslant s$, $0\leqslant t$ follows from commutativity of $A_1$ with $A_2$ on $C^2$ functions. Moreover if this holds the semi-group $\{P^{A_1+A_2}_t\}_{t>0}$ 
associated to $A_1+A_2$ satisfies $$P^{A_1+A_2}_t=P^1_tP^2_t.$$
\end{proposition}
\begin{proof}
Let $f:M\to \R$ be in $\mathcal{C}_2$.

We show first that \begin{equation} \label{equation:commi}
A_2P_t^1f=P_t^1A_2f
\end{equation}
For this set $V_t=A_2P_t^1f$. Then, by hypothesis [ii],
\begin{eqnarray}
\frac{\partial}{\partial t}V_t&=& A_2 A_1 P_t^1f  \nonumber\\
&=&A_1V_t
\label{equation:commii}
\end{eqnarray}
by commutativity.
By assumption [ii] we know $V_s$ is bounded uniformly in $s\in [0,t]$ for any $t>0$. However there is a unique $C^2$ and uniformly bounded solution, $P^1V_0$, to any diffusion equations such as  (\ref{equation:commii}) with given smooth bounded  initial condition $V_0$ (as is easily seen by the standard  use of It\^o's formula applied to $V_{t-s}$ acting on a diffusion process with generator $A_1$).      This gives
\begin{equation*}
 A_2P^1_tf=P^1_tV_0=P^1_t A_2f
\end{equation*}
as required. 
Now suppose $f\in \dom(\mathcal{G}_2)$. By assumption there is a sequence $\{f_n\}_n$
of functions in $\mathcal{C}_2$ converging in $\mathcal{G}_2$-graph norm to $f$. Then $P^1_tA_2f_n \to P^1_t\mathcal{G}_2f$  and $P^1_tf_n\to P^1_tf$. Equation (\ref{equation:commi}) therefore shows that $P^1_tf\in \dom(\mathcal{G}_2)$ and we have
 \begin {equation}
  \label{equation:comm2}
 \mathcal{G}_2P^1_t\supset P^1_t\mathcal{G}_2.
\end{equation}

\noindent
Next, for $f\in \dom( \mathcal{G}_2)$, and our fixed $t>0$ set $W_s=P^1_tP^2_sf $. Since the convergence of $\frac{1}{\epsilon}\{P^2_{s+\epsilon}f-P^2_sf\}$ to $\mathcal{G}_2P^2_sf$ is in $E$ we see, using equation (\ref{equation:comm2}),
$$\frac{\partial}{\partial s}W_s=P_t^1\mathcal{G}_2P^2_sf=\mathcal{G}_2P^1_tP^2_sf=\mathcal{G}_2W_s$$
since $P^2_sf\in \dom(\mathcal{G}_2)$. In particular $W_s\in \dom( \mathcal{G}_2)$.

\noindent
Although now it is not clear that $W$ is $C^2$  we see from this that $\frac{\partial}{\partial u}P^2_uW_{s-u}=0$ for $0<u<s$, giving
$$P_t^1 P^2_sf=P^2_0W_s=P^2_sW_0=P^2_sP^1_tf$$
for $0\leqslant s\leqslant t$. For $s>t$ it is now only necessary to use the semigroup property of $P^2$, to commute with $P^1_t$ portion by portion.

\noindent
Finally since $P^2_tf\in \dom(\mathcal{G}_2)$ the above gives
\begin{eqnarray*}
\frac{\partial}{\partial t}P^1_tP_t^2 f& = & A_1P^1_tP^2_tf+P^1_tA_2P^2_tf \\
 & = & (A_1+A_2)P_t^1P^2_t f
\end{eqnarray*}
and we can repeat the second arguement showing uniqueness of solutions of the diffusion equation to obtain $P_t^{A_1+A_2}f=P^1_tP^2_tf$.
\end{proof}
\begin{remark}

 Condition [i] does not always hold. A simple example is when the state space is
$\R^2-\{(0,1)\}$ and the operator is $\frac{\partial^2}{\partial x^2}$. The standard positive result for  degenerate operators on $\R^n$ is due to  Ole\v{i}nik, \cite{Oleinik-66}. 

\end{remark}

\section{Consequences for the Horizontal Flow}
\label{se:horizontal-lift}
For our standard set up of $p:N\to M$ with diffusion operator $\B$ over a cohesive $\A$,
let $P^V$ and $P^H$ denote the semi-groups generated by the vertical and horizontal components of $B$, and let $p_t^V(u,-),t\leqslant 0,u\in N$, be the transition probabilities of $P^V$. If we set $N_x=p^{-1}(x)$ for $x\in M$ then $p^V_t(u,-)$  will be a probability measure on $N^+_{p(u)}$, the union of $N_{p(u)}$ with $\Delta$.
For and 
For $\PP^\A_{x_0}$-almost all $\sigma\in \C_{x_0}M^+$ for  each $x_0\in M$ there are measurable maps 
$$ \paral^\sigma_t:N^+_{x_0}\to N^+_{\sigma_t}$$
 such that for each $u\in N_{x_0}$ the process $(t,\sigma)\mapsto \paral^\sigma_t(u)$ is an $\A^H$-diffusion and is over $\sigma$. These can be obtained, for example, by taking a stochastic differential equation, as equation (\ref{sde-1}), 
$$ dx_t=X(x_t) \circ dB_t +A(x_t) dt$$
for our $\A$-diffusion. Let $Y_x:E_x\to \R^m$ be the adjoint (and right inverse) of $X(x)$, each $x\in M$. Then consider the SDE on N 
$$dy_t=\tilde{X}(y_t)Y(\sigma_t)\circ d\sigma_t$$ 
and let $(t,\sigma)\mapsto \paral_t^\sigma$ be the restriction of its flow to $N_{x_0}$ , augmented by mapping the coffin state, $\Delta$,  to itself. This SDE is canonical since it can be rewritten as $$dy_t= h_{y_t}\circ d\sigma_t$$ for $h$ the horizontal lift map of  Proposition \ref{pr:h-map}. 

 We will often need to assume that the lifetime of this diffusion is the same as that of its projection on $M$:
\begin{definition}
The semi-connection induced by $\B$ is said to be \emph{stochastically complete }\index{connection!stochastically complete} if 
$$\C_{u_0}^pM^+:=\{\sigma:[0,\infty)\to M^+: \lim_{t\to \zeta}p(u_t)=\Delta \hbox { when } \zeta(u)<\infty\}$$
has full  $\PP_{u_0}^{\A^H}$ measure for each $u_0\in N$ or equivalently if the lifetimes satisfy $$\zeta(u)=\zeta(p(u))$$
for  $\PP_{u_0}^{\A^H}$-almost all paths $u$.

The semi-connection is said to be \emph{strongly stochastically complete}\index{connection!strongly stochastically complete}  if also we can choose a version of $\paral^\sigma_t: N_{\sigma (0)}\to N_{\sigma(t)}$ which is a smooth diffeomorphism whenever $\sigma(0)$ is a regular value of $p$ and $t<\zeta(\sigma)$.
\end{definition}

Note that strong stochastic completeness of the connection will hold whenever the fibres of $p$ are compact by the basic properties of the domains of local flows of SDE, \cite {Kunita-book}, \cite {Elworthy-book}. This also holds if the  stochastic horizontal differential equation
is strongly $p$-complete in the sense of Li \cite{flow}  for $p=\dim(N)-\dim(M)$.
\begin {proposition}
\label{pr:horflow} Suppose the semi-groups $P^V$ and $P^H$ commute and stochastic completeness of the connection holds. Then the horizontal flow preserves the vertical transition probabilities in the sense that for all  positive $s$ and $0<t<\zeta(\sigma)$,
\begin{equation}  
      (\paral^\sigma_t)_\ast p^V_s(u_0,-)=p^V_s(\paral^\sigma_t(u_0,-)
\end {equation}  
for all $u_0\in N_{\sigma}$ for $\PP^\A$-almost all $\sigma$.
Equivalently for any bounded measurable $h:N\to \R$ we have $\PP^\A$-almost surely;
\begin {equation}
\label{eq:horflow1}
P^V_s\left(h\circ\paral^\sigma_t\right)(u_0)=P^V_sh(\paral^\sigma_t(u_0))
\end{equation}
\end {proposition}
\begin {proof}  
It suffices to show that given any finite sequence $0\leqslant t_1\leqslant  t_2\leqslant \dots \leqslant  t_k<t $, bounded measurable $f_j:M\to \R$, $j=1,...,k$ and bounded measurable $h:N\to \R$, if $u_0\in N_{x_0}$ then 
\begin{equs}
&\E_{x_0}\{f_1(\sigma_{t_1})...f_k(\sigma_{t_k})\chi_{t<\zeta(\sigma)}P^V_s(h\circ \paral^\sigma_t)(u_0)\}\\
&=\E_{x_0}\{f_1(\sigma_{t_1})...f_k(\sigma_{t_k})\chi_{t<\zeta(\sigma)}P^V_s(h)(\paral_t^\sigma(u_0))\}.
\label{eq:comm3}
\end {equs}
where $\chi_Z$ denotes the indicator function of a set $Z$.
To see this set $\tilde{f_j}=f_j\circ p:N\to \R$. Then the  left hand side of (\ref{eq:comm3}) is 
\begin{eqnarray*}
&&\E_{x_0}\{\tilde{f_1}(\paral^\sigma _{t_1}(u_0))...\tilde{f}_k(\paral_{t_k}(u_0))\chi_{t<\zeta(\sigma)}P^V_s(h\circ \phi_t )(u_0)\}\\
&&=\E_{x_0}\{P^V_s\left(\tilde{f_1}(\paral ^\sigma_{t_1}(u_0))...\tilde{f}_k(\paral_{t_k}^\sigma(u_0))\chi_{t<\zeta(\sigma)}h(\paral^\sigma_t(u_0))\right)\}\\
&&=P^V_s\left(P^H_{t_1}\tilde{f_1}...P^H_{t_k-t_{k-1}}\tilde{f_k}P^H_{t-t_k}h\right)(u_0)\\
&&=\left(P^H_{t_1}\tilde{f_1}...P^H_{t_k-t_{k-1}}\tilde{f_k}P^H_{t-t_k}P^V_s h\right)(u_0)
\end{eqnarray*}
which reduces to the right hand side of (\ref{eq:comm3}).
\end {proof}
\begin{remark}
Assuming strong stochastic completeness of our semi-connection let $\{z_t:0\leqslant t<
\zeta(p(u_.)\}$ be a semi-martingale in $N$ with $ p(z_t)=x_t:=p(u_t): 0\leqslant t<\zeta(p(u_.))$. If $x_0$ is a regular value of $p$ we have the Stratonovich equation:
\begin{equation} \label{eq:pull-back}
 d\parals_t^{-1}z_t=T\parals_t^{-1}\circ T\parals_t^{-1}\left(h_{z_t}\circ dx_t\right)
\end{equation}
where $\parals_t$ refers to $\parals_t^{x_.}$. To see this, for example set $b_t= \parals_t^{-1}z_t$ and observe that $$ dz_t=d(\parals_t b_t)=T\parals_t\circ db_t+h_{\parals_tb_t}\circ dx_t.$$
\end{remark}

Now assume that our induced semi-connection is strongly stochastically complete. For a regular value $x_0$ of $p$ and $u_0\in N_{x_0}$ define a process $\alpha^{u_0}: [0,\infty)\times \C_{u_0}N^+\to N_{x_0}^+$ by 
\begin{equation}
\alpha_t^{u_0}(u)=\alpha_t(u)=(\paral_t^{p(u)})^{-1}u_t
\end{equation}
if $u\in \C_{u_o}N$ with $t<\zeta(u)$ and define $\alpha_t^{u_0}(u)=\triangle$ if $t\geqslant  \zeta(u)$. 
Note that $\alpha_t$ may not go out to infinity in $N_{x_0}$  as $t$ increases to its extinction time.

Also define 
$$ \parals_s^*(\B^V)(f)= \B^V(f\circ \parals_s)\circ \parals_t^{-1}$$
to obtain a random time dependent diffusion operator $ \parals_s^*(\B^V)$ on each fibre
over a regular value of $p$.

\begin{lemma}\label{le:pull-back}
In the notation of equation (\ref{sde-5}) we have the It\^{o} equation for 
$\alpha_t:=\alpha^{u_0}_t$:\begin{equation}
\label{sde-alpha}
\nabla ^V\qquad d\alpha_t=T\parals_t^{-1} V(\parals_t\alpha_t)dW_t+T\parals_t^{-1}V^0(\parals_t\alpha_t)dt.
\end{equation}
In particular for $f:N\to\R$ in $C^2$
\begin{equation} \label{eq:mart}
M^{df,\alpha}_t:=f(\alpha_t)-\int_0^t\parals_s^*(\B^V)(f)(\alpha_s)ds
\end{equation}
is a local martingale.
\end {lemma}
\begin{proof}
Formula ( \ref{sde-alpha}) is immediate from equations (\ref{sde-5}) and (\ref{eq:pull-back}). That $ M^{df,\alpha}_.$ is a local martingale follows immediately using the properties of pull-backs under diffeomorphisms of Lie derivatives when $V$ is $C^1$, and by going to local co-ordinates otherwise.
\end{proof}
\begin {lemma}\label {le:Lie}  At all points above regular values of $p$ we have:
 $$\frac{d}{ds}\E\{ \parals_s^*(\B^V)\} |_{s=0}=[\A^H, \B^V] $$
\end {lemma}
\begin {proof}
This is an exercise in the use of Ito's formula. For example write $\A$ in the H\"ormander form\index{H\"ormander!form}
$$\A=\frac{1} {2} \sum_{j=1}^m \LL_{X^j}\LL_{X^j}+\LL_X^0$$
so that $\parals_s$ is the flow of  the SDE $$ d z_s=\sum_{j=1}^m \widetilde{X^j}(z_s)dB^j+\widetilde{X^0}(z_s)$$
using the horizontal lifts of the vector fields $X^j$. From the Ito formula in lemma 9B Chapter VII of \cite {Elworthy-book} we have 
$$\frac{d}{ds}\E\{ \parals_s^*(\B^V)\}|_{s=0}=\frac{1}{2}\sum_{j=1}^m\frac{d^2}{ds^2}(\parals_s^j)^*(\B^V)\}|_{s=0}+\frac{d}{ds}(\parals_s^0)^*\B^V|_{s=0}$$
where $\paral^i$ is the flow of the vector field $\widetilde{X^j}$. Since $$\frac{d}{ds}(\parals_s^j)*(\B^V)=[\LL_{\widetilde{X^j}},(\parals_s^j)^*\B^V]$$
 we have the result.
 \end{proof}

\begin {definition}\label{de:stoch-hol-inv}
For a regular value $x_0$ of $p$. We say $\B^V$ is \emph{stochastically holonomy invariant at $x_0$} if  on $N_{x_0}$ we have $\parals_t^*(\B^V)=\B^V$ for all $0\leqslant t<\zeta^{x_.}$ with probability one.  If this holds for all all regular values $x_0$ then we say $\B^V$ is \emph{ stochastically holonomy invariant}.
Similarly we say $B^V$ is \emph{holonomy invariant at $x_0$} if the corresponding result holds for parallel translation along any piecewise $C^1$ curve starting at $x_0$ in $M$, and is \emph{holonomy invariant} if this holds for all regular values $x_0$.
\end {definition}
\begin {remark}
\label{re:stoch-hol-inv}
\begin{enumerate}
  \item If the $\A$-diffusion on $M$ is represented by a stochastic differential equation we can lift that equation  to $N$ and obtain a local flow $\eta^H_t: 0\leqslant t< \zeta^H(-)$ where 
  $\zeta^H(y): y\in N$ gives its explosion times; so that with probability one  $\eta^t$ is defined and smooth on the open set $\{y\in N : t\leqslant \zeta^H(y)$, see \cite {Kunita-book} or
  \cite{Elworthy-book}. We can say that $\B^V$ is invariant under the horizontal flow if for all $C^2$ functions $f:N\to \R$ we have 
  
  $$ \B^V(f)\circ \eta_t=\B^V(f\circ \eta_t)$$
   on $\{y\in N : t\leqslant \zeta^H(y)$,  almost surely, for all $t>0$.
   This does not require strong stochastic completeness of the semi-connection, nor do we have to restrict attention to fibres over regular values. On the other hand if it holds, and given such strong stochastic completeness, if $x_0$ is a regular value it follows that    $N_{x_0}$ lies in $\{y\in N : t\leqslant \zeta^H(y)$ for all $t< \zeta^M(x_0)$ and that we have stochastic holonomy invariance at $x_0$. 
  \item Assume completeness of the semi-connection. If $\A$ satisfies the standard H\"ormander condition\index{H\"ormander!condition}, or more generally if the space $\mathcal{D}^0(x_0)$, as in Section \ref{se:topology} is all of $M$, then holonomy invariance at $x_0$ implies holonomy invariance. This follows since concatenation of paths gives composition of the corresponding parallel translations and the conditions imply that any two points can be joined by a smooth path with derivatives in $E$.  Moreover by Theorem \ref {th:topology} every point is a regular value and so given also strong stochastic completeness of the connection from the theorem below we see that holonomy invariance of $B^V$ at one point implies it is invariant under the horizontal flow induced by any SDE on $M$ which gives one point motions with generator $\A$. The same holds for stochastic holonomy invariance: see Theorem \ref{th:equivs} below.

\end{enumerate}
\end{remark}
\begin{theorem}\label{th:equivs}
Suppose the induced semi-connection is complete and strongly  stochastically complete, and $x_0$ is a regular value of $p$.  Then the following are equivalent:
\begin {itemize}
\item[{[i$x_0$]}] For all  $u_0\in N_{x_0}$ and for any $\F^\alpha$-stopping time $\tau$ with $\tau(\alpha(u))<\zeta(p(u))$, the process $\{\alpha_t:0\leqslant t<\tau\}$ is independent of $\F^{x_0}$;
\item[{[ii$x_0$]}] $\B^V$ is stochastically holonomy invariant at $x_0$;
\item[{[iii$x_0$]}] $\B^V$ is  holonomy invariant at $x_0$;

\item[{[iv$x_0$]}]  $\B^V$ and $\A^H$ commute at all points of $\overline{\mathcal D^0(x_0)}$;
\item[{[v$x_0$]}] $P_.^V$ and $P_.^H$ commute at all points of $\overline{\mathcal D^0(x_0)}$.
\end{itemize}
If the above hold at some regular value $x_0$ they hold for all elements in $\mathcal D^0(x_0)$. Moreover $\alpha^{u_0}_.$ is a Markov process on $N_{x_0}$ with generator
$\B^V$.
\end {theorem}

\begin{proof}
We will show that [i$x_0$] is equivalent to  [ii$x_0$] which implies [iv$x_0$]. Then [iv] implies [iii$y$] for all $y\in {\mathcal D^0(x_0)}$ which implies [v]. Finally we show [v] implies [ii$y$] for all $y\in {\mathcal D^0(x_0)}$.

Assume [i$x_0$] holds. Let $f:N_{x_0}\to \R$ be smooth with compact support. Then the local martingale $M^{df,\alpha}$ given by formula (\ref{eq:mart}) is a martingale and from equation (\ref{sde-alpha}) we see that $$ \E\{M^{df,\alpha}|\F^{x_0}\}=f(u_0).$$
Therefore for $\PP^{x_0}$-almost all $\sigma$ in $C_{x_0}M$
\begin{equation} \label{eq:alpha1}
\E\{f(\alpha_t)\}=\E\{f(\alpha_t)|p(u_.)=\sigma\}=f(u_0)+\int_0^t\E\{(\parals_s^{\sigma})^*(\B^V)(f)(\alpha_s)\}ds.
\end{equation}
Also, in the notation of equation  (\ref{sde-alpha}), with the obvious notation for the filtrations generated by our processes, we have $\F^{\alpha_.}_t\subset\F^{W_.}_t\wedge\F^{x_0}_t$ and $\F^{W_.}_t\subset\F^{\alpha_.}_t\wedge\F^{x_0}_t$ so our assumption implies that $\F^{W_.}_t=\F^{\alpha_.}_t$, for all positive $t$, after stopping $W_.$ at the explosion time of $\alpha_.$. From this, and equation (\ref{sde-alpha})  we see that if we set $\bar{M}^{df,\alpha}_t=\E\{M^{df,\alpha}_t|\F^{\alpha_t}\}$ we obtain  a martingale with respect to $\F^{\alpha}_*$ and
 \begin{equation}
f(\alpha_t)=\bar{M}^{df,\alpha}_t +\int_0^t\overline{\parals_s^*\B^V}(f)(\alpha_s) ds
\end{equation}
where $\overline{\parals_s^*\B^V}=\E\{\parals_s^*\B^V\}$. Thus by the usual martingale characterisation of Markov processes we see that $\alpha_.$ is Markov with (possibly time dependent) generator  $\overline{\parals_s^*\B^V}$ at time $s$. However equation (\ref{eq:alpha1}) then implies, for example by \cite {Revuz-Yor}  Proposition(2.2), Chapter VII,  that the generator is given by $(\parals_s^{\sigma})^*(\B^V)$ for arbitrary $\sigma$ in a set of full measure in $C_{x_0}M$.  Thus [i$x_0$] implies the stochastic holonomy invariance [ii$x_0$].

Conversely if [ii$x_0$] holds, equation (\ref{eq:mart}) gives 
$$ f(\alpha_t)=M^{df,\alpha}_t +\int_0^t\B^V(f)(\alpha_s)ds.$$
Then $M^{df,\alpha}_.$ is an $\F^{\alpha_.}_*$-martingale and again we see that $\alpha_.$ is Markov,  with generator $\B^V$. It is therefore independent of $x_.$ giving [i$x_0$]. Moreover, in an obvious notation, if $0\leqslant s\leqslant t$, by the flow  property of parallel translations, on $N_{x_0}$,
$$\B^V=\parals_t^*(\B^V)=\parals_s^*(\parals_t^s)^*(\B^V),$$
and so, almost surely, at all points of $N_{x_s}$ we have 
$$(\parals_t^s)^*(\B^V)=(\parals_s^*)^{-1}\B^V=\B^V.$$
Since $(\parals_t^s)^*(\B^V)$ has the same law as $\parals_{t-s}^*(\B^V)$and is independent of $\F^{x_0}_s$ this shows that  [ii$y$] holds for $p^\A_s(x_0,-)$-almost all $y\in M$ for all $s>0$.

On the other hand [ii$y$] implies that $\B^V$ and $\A$ commute  on $N_y$ by Lemma \ref{le:Lie}. Thus by continuity of $[\B^V,\A^H]$, and the support theorem we see that [ii$x_0$] implies [iv]. 

Furthermore as in Theorem \ref{theorem:comm1} we see that [iv] implies that $\B^V$ commutes with basic vector fields at all points over $\overline{\mathcal D^0(x_0)}$. From this the holonomy invariance [iii$y$] holds for all $y\in {\mathcal D^0(x_0)}$.

Now assume [iii$x_0$] and so by  Remark \ref{re:stoch-hol-inv}(2.) we have [iii$y$] for all $y\in \mathcal D^0(x_0)$. Since $\parals_t^{\sigma}(u_0)$ stays above $\mathcal D^0(x_0)$ for any suitable piecewise smooth $\sigma$ we find
the solution to the martingale problem of $\B^V$ for any point $u_0$ of $N_{x_0}$ is holonomy  invariant at $u_0$, i.e. along  piecewise smooth curves $\sigma$ in $M$ starting at $x_0$,
 $$P^V_t(f \circ \parals_s^{\sigma}-)(u_0)=P^V_t(f)(\parals_s^{\sigma}u_0).$$
 
 By Wong-Zakai approximations  we see that stochastic holonomy invariance of $P^{\B^V}$  holds over  $x_0$ and hence on taking expectations we get [v$x_0$].  As observed  we also get  [v$y$] for all $y\in \mathcal D^0(x_0)$ and hence by continuity for all $y\in \overline{\mathcal D^0(x_0)}$. Thus [iii$x_0$] implies [v].

Finally assuming [v] we can apply Proposition \ref{pr:horflow}, observing that the proof still holds since it only involves points in $\overline{ \mathcal D^0(x_0)}$. Differentiating equation (\ref{eq:horflow1}) in $s$ at $s=0$ gives the stochastic holonomy invariance [ii$y$] for all $y\in {\mathcal D^0(x_0)}$
\end{proof}
\begin {remark}\label{re:deterministic}
From the proof and Theorem \ref{theorem:comm1} we see that the stochastic completeness of the connection is not needed to ensure that [iv$x_0$]  and [iii$x_0$] are equivalent.
\end{remark}

We can now go further than our Theorem \ref{th:topology} in extending  Hermann's result,  Theorem\ref{th:hermann}. For this we will need some extra  hypoellipticity conditions to deal with the case of non-compact fibres. Take a H\"ormander form\index{H\"ormander!form} $\A$
corresponding to a smooth factorisation
 $$\sigma_x^{\A}=X(x)X(x)^*$$
 with $X(x)\in \mathbb{L}(\R^m:T_xM$ for $x\in M$. Let $\HHH$ denote the usual Cameron -Martin space of finite energy paths $H=L^{2,1}_0([0,1];\R^m)$. For $h\in \HHH$ and $x\in M$ let 
 $\phi^h_t(x), 0\leqslant t\leqslant 1$ be the solution at time $t\in [0,1]$ to the ordinary differential equation\begin{equation}
\label{eq:control}
\dot{z}(t)=X(z(t))(\dot{h})
\end{equation}
with $\phi^h_0(x)=x$. In particular we assume such a solution exists up to time $t=1$.  For each $x\in M$ this gives a smooth mapping  $\phi^-_1(x):\HHH\to M$, namely $h\mapsto \phi^h_1(x)$. Let $C^{h,x}:E_x\to E_x$ be the {\bf deterministic Malliavin covariance operator}, see \cite{Bismut-book}, given by 
$$ C^{h,x}=T_h\phi_1^-(x)(T_h\phi_1^-(x))^\ast.$$
Then $\phi^-_1(x)$ is a submersion in a neighbourhood of $h$ if and only if $C^{h,x}$
is non-degenerate. It is shown in \cite{Bismut-book} that this condition is independent of the choice of H\"ormander form\index{H\"ormander!form} for $\A$, and follows from the standard H\"ormander condition\index{H\"ormander!condition} that $X^1,\dots ,X^m$ and their iterated Lie brackets span $T_xM$ when evaluated at the point $x$. A more intrinsic formulation of it can me made in terms of the manifold of $E$-horizontal paths of finite energy, as described in
 \cite{Montgomery-book}.

\begin{theorem}\label{th:generalHermann}
Consider a smooth map $p:N\to M$ with diffusion operator $\B$ on $N$ over a cohesive diffusion operator $\A$. Suppose that the connection induced by $\B$ is complete. Also assume that $\mathcal D^0(x)$ is dense in $M$ for all $x\in M$   and that either the fibres of $p$ are compact or that the solutions to  equation (\ref{eq:control}) exist up to time $1$ and there exists $h_0\in \HHH$ and $x_0\in M$ such that $C^{h_0,x_0}$ is non-degenerate. Then $p:N\to M $ is a locally trivial bundle.

If also $\B$ and $\A^H$ commute we can take $N_{x_0}$, the fibre over $x_0$, to be the model fibre and choose the  local trivialisations $$\tau: U\times N_{x_0}\to p^{-1}(U)$$ to
satisfy $$\tau(x,-)^* (\B^V|N_{x})=\B^V|N_{x_0}.$$
\end{theorem}

\begin{proof}
The local triviality given compactness of the fibres is a special case of Corollary \ref{cor:top} so we will only consider the other case. 

For this set $y=\phi_1^{h_0}(x_0)$. Our assumption on the covariance operator together with the smoothness of $h\mapsto \phi_1^{h}(x_0)$ implies by the inverse function theorem that there is a neighbourhood $U_y$ of $y$ in $M$ and a smooth immersion $s:U_y\to \HHH$ with $s(y)=h_0$ and  $\phi_1^{s(x)}(x_0)=x$ for $x\in U_y$.

We know from Theorem \ref{th:topology} that $p$ is a submersion so all its fibres are submanifolds of $N$.  Define $\tau_{U_y}: U_y\times N_{x_0}\to p^{-1}(U_y)$ by using the parallel translation along the curves $\phi^{s(x)}_t:0 \leqslant t\leq 1$ that is:
\begin{eqnarray}
\label{eq:trivialisation }
\tau_{U_y}(x, v)=\parals_1^{\phi^{s(x)}_.}(v) \hspace{.5in}{(x,v)\in (U_y\times N_{x_0})}.
\end{eqnarray}

For a general point $x$ of $M$ we can find an $x'\in U_y\cap \mathcal D^0(x)$ and argue as in the proof of Theorem \ref{th:topology} to obtain  open neighbourhoods $U_x$ of $x$  in $M$ and $U'_{x'}$ of $x'$ in $U_{x_0}$ and a fibrewise diffeomorphism of $p^{-1}(U'_{x'})$ with $ p^{-1}(U_x)$ obtained from parallel translations. This can be composed
with a restriction of $\tau_{U_{x_0}}$ to give a trivialisation near $x$.
This proves local triviality. The rest follows directly from Remark \ref{re:deterministic} since our trivialisations came from parallel translations.
\end{proof}

\begin{remark}\label{re:group}
 Set \begin{equation}
\label{eq:group }
\mathbb{G}(\B^V_{x_0})=\{\alpha\in \Diff(N_{x_0}):\alpha ^*(\B^V|N_{x_0})=\B^V|N_{x_0}\}.
\end{equation}
Then assuming the commutativity in the theorem we can consider $\mathbb{G}(\B^V_{x_0})$ as a structure group for our bundle though unless the fibres of $p$ are compact it is not clear if we have a smooth fibre bundle with this as group in the usual sense, since this requires smoothness  into $\mathbb{G}(\B^V_{x_0})$ of the transition maps between overlapping  trivialisations. See the next section and Michor \cite{Michor} section 13.

Note that elements of  $\mathbb{G}(B^V_{x_0})$ preserve the symbol of  $\B^V$ and so if that symbol has constant rank preserve the inner product induced on the image of $\sigma^{\B^V}$. In particular if $\B^V$ is elliptic they are isometries of the Riemannian structure induced on  the fibre $N_{x_0}$. This is the situation arising from Riemannian submersions as in Hermann's Theorem \ref{hermann2} and  described in detail in Chapter \ref{ch:Riem.sub} below. The space of  isometries of a Riemannian manifold with compact- open topology is  well known to form a Lie group, for example see \cite{Kobayashi-Nomizu-I}. However there appears to be no detailed proof that the same holds in degenerate cases even when the H\"ormander condition\index{H\"ormander!condition} holds at each point.  When H\"ormander's condition  holds the Caratheodory metric on the manifold determines the standard manifold topology, e.g. see \cite{Montgomery-book} Theorem 2.3, which is locally compact, and the group of isometries of a connected locally compact metric space is locally compact in the compact-open topology, see \cite {Kobayashi-Nomizu-I}, Chapter 1, Theorem 4.7. Thus in this case $\mathbb{G}(B^V_{x_0})$ will be locally compact. 

In general preserving the possibly  degenerate Riemannian structure determined by its symbol will not  be enough to characterise $\mathbb{G}(B^V_{x_0})$. Even in the elliptic case there may be a ``drift vector" which needs to be preserved as well and this may lead to $\mathbb{G}(B^V_{x_0})$ being very small. For example if $N_{x_0}$ is $\R^2$ and 
$\B^V= \frac{1}{2}\triangle -|x|^2\frac{\partial}{\partial x^1}$ the group is trivial.

\end{remark}

\begin{example}\label{ex:parallel-R}
\begin {enumerate}
\item  As an example consider the situation described in Section \ref{se:deri-flow} of the derivative flow of a stochastic differential equation (\ref{SDE-1})  on $M$ acting on the frame bundle $GLM$ to produce a diffusion operator $\B$ on $GLM$. Assume that $M$ is Riemannian and complete, and that
the one point motions are Brownian motions, so that $\A=\frac{1}{2}\triangle$ . Assume also that the connection induced is the Levi-Civita connection. Then if $\B$ and ${\A}^H$ commute, by Corollary {co:equ-comm} , we see that the co-efficients $\alpha$ and $\beta$ of $\B^V$ described in Theorem \ref{th:conn-deri} must be constant along horizontal 
curves. However as pointed out in the proof of Corollary \ref{co:derivative}, the restriction of $\alpha (u)$ for $u\in GLM$ to anti-symmetric tensors is essentially (one half of) the curvature operator. It follows that the curvature is parallel \index{curvature!parallel}, $ \nabla \CR=0$.
In turn this implies, \cite{Kobayashi-Nomizu-I} page 303,  that $M$ is a local symmetric space and so  if simply connected,  a symmetric space. In  Section \ref{se:symmetric} we show how such stochastic differential equations arise on any symmetric space. Also from Example \ref{ex: gbm-spheres} we see that the standard gradient SDE for Brownian motion on spheres also give derivative flows with this property.

\item For the apparently weaker property of commutativity for the derivative flow $T\xi_t$ of our SDE (\ref{SDE-1}) acting directly on the tangent bundle $TM$ recall first that if the generator  $\A$ is cohesive (and even if it just happens that the symbol of $\A$ has constant rank, see \cite{Elworthy-LeJan-Li-book}) then for $v_t =T\xi_t(v_0)$ some $v_0\in T_{x_0}M$ we have the covariant SDE
\begin{equation}
\label{eq:der-flow }
\hat{D}v_t=\breve{\nabla}_{v_t}X dB_t-\frac{1}{2}\breve{\Ric}^{\#} (v_t)dt +\breve{\nabla}_{v_t}A dt.
\end{equation}  

From this we see that  if $\A$ is cohesive the process $ \alpha_.$ defined  by $\alpha_t= \hat{\parals_t}^{-1}T\xi_t(v_0)$ satisfies the SDE
$$ d\alpha_t=\hat{\parals_t}^{-1}\left(\breve{\nabla}_{\hat{\parals_t} \alpha_t }X dB_t-\frac{1}{2}\breve{\Ric}^{\#} (\hat{\parals_t} \alpha_t) dt +\breve{\nabla}_{\hat{\parals_t} \alpha_t}A dt \right).$$

Suppose also that $A=0$. We see that $\alpha_.$ is independent of $\xi_.(x_0)$ if and only if both $\breve{\nabla}_{-}X$ and $\breve{\Ric}^{\#} $ are holonomy invariant. If $M$ is Riemannian and the solutions of the SDE are Brownian motions and the induced connection is the Levi-Civita connection we can deduce, as above, using Theorem \ref{th:equivs}, that commutativity of the the vertical and horizontal diffusions operators on $TM$ holds  only if  $M$ is locally symmetric .
\end{enumerate}
\end{example}

\chapter {Example: Riemannian Submersions \& Symmetric Spaces}
\label{ch:Riem.sub}

\section{Riemannian Submersions}
Recall that when $N$ and $M$ are Riemannian manifolds a smooth surjection $p:N\to M$ is a \emph{Riemannian submersion} if for each $u$ in $N$ the map $T_up$ is an orthogonal projection onto $T_{p(u)}M$, i.e. restricted to the orthogonal complement of its kernel it is an isometry. Note that if $p:N\to M$ is a submersion and $M$ is Riemannian we can choose a  Riemannian structure for $N$ which makes $p$ a Riemannian submersion. If a diffusion operator $\B$ on $N$ which has projectible  symbol  for $p:N\to M$ is also elliptic its symbol induces Riemannian metrics on $N$ and $M$ for which $p$ becomes a Riemannian submersion.  
 A well studied situation is when $p$ is a Riemannian submersion and $\B$ is the Laplacian, or $\frac{1}{2}\triangle_N$, on $N$. The basic geometry of Riemannian submersions was set out by O'Neill in \cite{O'Neill}; he ascribes the term `submersion' to Alfred Gray.  In this section we shall mainly be relating the work of B\'rard-Bergery \& Bourguignon \cite{Berard-Bergery-Bourguignon}, Hermann, \cite{Hermann}, Elworthy\& Kendall, \cite{Elworthy-Kendall}, and Liao, \cite{Liao89}, to the discussion above. The book \cite{Falcitelli-Ianus-Pastore} shows the breadth of geometric structures which can be considered in association with Riemannian submersions.

 A simple example of a Riemannian submersion is the map $p:\R^n-\{0\}\to\infty$
 given by $p(x)=|x|$. Then, for $n>1$, Brownian motion on $\R^n-\{0\}$ is mapped to the Bessel process on $(0,\infty)$ with generator $\A= \frac{1}{2}\frac{d^2}{dx^2}+\frac{n-1}{2x}\frac{d}{dx}$. Thus in this case $\frac{1}{2}\triangle_N$ is projectible but its projection is not  $\frac{1}{2}\triangle_M$. The well known criterion for the latter to hold is that $p$ has \emph{minimal fibres} as we show below. See also \cite{Elworthy-book},and \cite {Liao89}.

To examine this in more detail we follow Liao,\cite{Liao89}.Suppose that $p$ is a Riemannian submersion. The horizontal subbundle on $N$ is just the orthogonal complement of the vertical bundle. Working locally take an orthonormal family of vector fields $X^1,\dots,X^n$ in a neighbourhood of of a given point $x_0$ of $M$. Let $\tilde{X}^1,\dots,\tilde{X}^n$ be their horizontal lifts to a neighbourhood of some $u_0$ above $x_0$, and let $V^1,\dots,V^p$ be a locally defined orthonormal family of vertical vector fields around $u_0$. Then near $u_0$, using the summation convention over $j=1,\dots,n$, $\alpha=1,\dots,p$, we have \begin{equation}
\triangle_N=\tilde{X}^j\tilde{X}^j+V^\alpha V^\alpha-\nabla^N_{\tilde{X}^j}\tilde{X}^j-\nabla_{V^\alpha}V^\alpha
\end{equation}
while
\begin{equation}
\triangle_M={X}^j{X}^j-\nabla^N_{X^j}X^j.
\end{equation} 
Here $\nabla_M$, $\nabla_N$ refer to the Levi-Civita connections on $M$ and $N$, and we are identifying the vector fields with the Lie differentiation in their directions.

Now $\tilde{X}^j\tilde{X}^j$ lies over $X^jX^j$ while $V^\alpha V^\alpha$ is vertical. Also the horizontal component of the sum $\nabla_{V^\alpha}V^\alpha$ at a point $u\in N$is the trace of the second fundamental form of the fibre $N_{p(u)}$ of $p$ through $u$, denoted by $T_{V^\alpha}V^\alpha$ in O'Neill's notation, while $\frac{1}{2}\triangle_N$ lies over $\nabla_{{X}^j}{X}^j$ by Lemma 1 of \cite{O'Neill}.

Thus we see that $\frac{1}{2}\triangle_N$ is projectible if and only if the trace of the second fundamental form, $\trace T$, of each fibre $p^{-1}(x)$ is constant along the fibre in the sense of being the horizontal lift of a fixed tangent vector, $2A(x)\in T_xM$. If so $\frac{1}{2}\triangle_N$ lies over $\frac{1}{2}\triangle_M-A$.
In particular $A=0$, or equivalently $p$ maps Brownian motion to Brownian motion, if and only if $p$ has minimal fibres.

In general to relate to the discussion in Section \ref{se-basic symbol} we can set $b^H(u)=-\frac{1}{2}\trace T(u)$, with $b(u)=T_upb^H(u)$ in $T_{p(U)}M$. Let $\triangle^V$ be the vertical operator on $N$ which restricts to the Laplacian on each fibre, and let $\triangle^H$ be the horizontal lift of $\frac{1}{2}\triangle_M$.  Our decomposition in Theorem \ref{the-extended-deco} becomes
\begin{equation}
\frac{1}{2}\triangle_N=\left(\frac{1}{2}\triangle^H-\frac{1}{2}\trace T\right)+ \frac{1}{2}\triangle^V
\end{equation}
since the vertical part of $\nabla_{V^\alpha}V^\alpha$ is just $\nabla^V_{V^\alpha}V^\alpha$ where $\nabla^V$ refers to the  connection on the vertical bundle which restricts to the Levi-Civita of the fibres, and also
the vertical part of $\tilde{X}^j\tilde{X}^j$ vanishes because by Lemma 2 of \cite{O'Neill} the vertical part of $\tilde{X}^j\tilde{X}^k$ is the vertical part of $\frac{1}{2}[\tilde{X}^j,\tilde{X}^k]$.

\section{Riemannian Symmetric Spaces}
\label{se:symmetric}
Let $K$ be a Lie group with bi-invariant metric and let $M$ be a Riemannian manifold with a symmetric space structure given by a triple $(K,G,\sigma)$. This means that there is a smooth left action $K\times M\to M,(k,x)\mapsto L_k(x)$ of $K$ on $M$ by isometries such that if we fix a point $x_0$ of $M$ and define $p:K\to M$ by $p(k)=L_k(x_0)$ then $p$ is a Riemannian submersion and a principal bundle with group the subgroup $K_{x_0}$of $K$ which fixes $x_0$. Write  $G$ for $K_{x_0}$. Thus $M$ is diffeomorphic to $K/G$.  Moreover if $\g$ denotes the Lie algebra of $G$, and $\k$ that of $K$, (identified with the tangent spaces at the identity to $G$ and $K$ respectively), there is an orthogonal and $ad_G$- invariant decomposition $$ \k=\g+\m$$ where $\m$ is a linear subspace of $T_{\id}K$. Further $\sigma$ is an involution on $K$ and  $\g$ and $\m$ are, respectively, the $+1$ and the $-1$ eigenspaces of the involution on $T_{\id}K$ induced by $\sigma$. See Note 7, page 301, of Kobayashi \& Nomizu Volume I, \cite{Kobayashi-Nomizu-I}, for definitions and basic properties, and Volume II,  \cite {Kobayashi-Nomizu-II}, for a  detailed treatment.

We shall also let $\sigma$ denote the involutions induced by $\sigma$ on $\k$ and on $M$, and by differentiation on $ TM$ and $OM$. On $M$ it is an isometry, so it does act on $OM$. Note that on $T_{x_0}M$ it acts as $v\mapsto-v$.

Since $G$ fixes $x_0$ the derivative of the left action $L_k$ at $x_0$ gives a representation of $G$ by isometries of $T_{x_0}M$. The \emph{linear isotropy representation}. We shall assume it to be faithful, i.e. injective. As a consequence the action of $K$ on $M$ is effective, so that $K$ can be considered as a sub-group of the diffeomorphism group of $M$, and also the action of $K$ on the frame bundle  of $M$ is free, i.e the only element of  $K$ which fixes a frame is the identity element. See page 187 and the remark on page 198 of \cite {Kobayashi-Nomizu-II} for a discussion of this, and how the condition can be avoided.  Taking a fixed orthonormal frame $u_0:\R^n\to T_{x_0}M$ , say, at $x_0$, we can consider $G$ as acting by isometries on $\R^n$ by 
\begin{equation}g \cdot e=u_0^{-1}TL_gu_0(e).\end{equation}
Let $\rho:G\to O(n)$ denote this representation.
 We then have the well known identification of $K$ as a subbundle of the orthonormal frame bundle of $M$:
\begin{proposition}
\label {pr:bundle iso}
Let $\Phi: K\to OM$ be defined by $\Phi(k)(e)=TL_k(u_0e)$ for $e\in\R^n$. Then $\Phi$ is an injective homomorphism of principle bundles. Moreover $\Phi$ is equivariant for the actions of $\sigma$ on $K$ and $OM$.
\end{proposition}
\begin{proof} To see that $\Phi$  is a bundle homomorphism it is only necessary to check that $\Phi$ commutes with the actions of $G$. For this take $e\in \R^n$ and $g\in G$. Then, for $k\in K$, 
\begin{eqnarray*}
\Phi(k \cdot g)(e)&=&TL_k TL_g u_0(e)\\
&=&\Phi(k)TL_g u_0(e)=\Phi(k)u_0(g \cdot e)
\end{eqnarray*}
as required.
 For the equivariance with respect to $\sigma$ observe that by definition, $\sigma(L_kx_0)=L_{\sigma(k)}x_0$ so that acting on the frame $\Phi(k)$ we have
 \begin{eqnarray*}
\sigma(\Phi(k))&=&\sigma (TL_k\circ u_0)  \\
 & = & TL_{\sigma(k)}u_0=\Phi(\sigma(k)).
\end{eqnarray*}
\end{proof}

It is easy to see that $p:K\to M$ has totally geodesic fibres. We can therefore take $\B=\frac{1}{2}\triangle^K$ to have $\B$ lying over $\frac{1}{2}\triangle^M$. Moreover in the decomposition of $\B$ the vertical component
$\frac{1}{2}\triangle^V$ restricts to the one half the Laplacian of $G$ on the fibre $p^{-1}(x_0)$. The induced connection has horizontal subspace $\m$ at the identity element of $K$. It is clearly left K-invariant and so $H_k=TL_k[\m] $ for general $k\in K$. From the equivariance under the right action of  $G$ it is a principle connection: $TR_g[H_k]=H_{kg}$. Since $H_{kg}=TL_kTL_g[\m]= TR_kTL_k ad_g[\m]$ this holds because of the $ad_G$-invariance of $\m$. This is the \emph{canonical connection}.

The connnection on $K$ extends to one on $OM$ as described in Proposition \ref{pr:h-funct-2}.This is known as the \emph{canonical linear connection}.  Since the connection on $K$ is invariant under $\sigma$, by the equivariance of $\Phi$ so is the  canonical linear connection. As in \cite {Kobayashi-Nomizu-II} we have:
\begin{proposition} \label {pr:can-levi} The canonical linear connection is the Levi-Civita connection.
\end{proposition}
\begin{proof} It is only necessary to check that its torsion $T$ vanishes. By left invariance it is enough to do that at the point $x_0$. Let $u,v\in T_{x_0}M$. However by invariance under $\sigma$ we see 
$$ T(u,v)=\sigma T(\sigma(u),\sigma(v))=-T(-u,-v)=-T(u,v),$$
as required.
\end{proof}

Let $k_t,t\geqslant0$ be the canonical Brownian motion on $K$ starting at the identity, $\id$, and let $B_t$ be the Brownian motion on the Euclidean space $\k$ given by the right flat anti-development:
$$B_t=\int_0^tTR_{k_s}^{-1}d\{k_s\}.$$
Define $\xi_t: M\to M$ by $\xi_t(x)=L_{k_t}x$, for $t \geqslant 0$, $x\in M$. 
\begin {proposition} The diffeomorphism group valued process $\xi_t,t\geqslant 0$ is the flow of the sde 
$$ dx_t=X(x_t)\circ dB_t$$
where $$X(x)\alpha=\frac{d}{dt}L_{\exp t\alpha}x|_{t=0} $$
\end{proposition}
\begin{proof} Observe that $k_.$ satisfies the right invariant SDE
$$ dk_t=TR_{k_t}\circ dB_t$$ which is $p$-related to the given SDE on $M$.
\end{proof}

\begin{remark}\label{re:sym-ljw} 
The last two propositions relate to the discussion of connections determined by stochastic flows \index{connections!determined by stochastic flows} in the next section, and to the discussion about canonical SDE on symmetric spaces in \cite{Elworthy-LeJan-Li-book}. In \cite {Elworthy-LeJan-Li-book} it was shown that the connection determined by our SDE is the Levi-Civita connection. In  Proposition \ref{pr:determ-2} below, and in Theorem 3.1 of \cite{Elworthy-LeJan-Li-principal}, it is shown that the connection determined by a flow (in this case the canonical linear connection) is the adjoint of that induced by its SDE. this is confirmed in our special case since the adjoint of a Levi-Civita connection is itself.
\end{remark}

 We can also apply our analysis of the vertical operators and Weitzenb\"ock formulae to our situation, For this it is simplest to assume the symmetric space is \emph{irreducible}. This means that the restricted linear holonomy group of the canonical connection on $p:K\to M$ is irreducible i.e. for every $g\in G$ there is a null-homotopic loop based at $x_0$ whose horizontal lift starting at $\id\in K$ ends at  the point $g$. The definition in \cite {Kobayashi-Nomizu-II} is that $[\m,\m]$ acts irreducibly on $\m$ via the adjoint action, and it is shown there, page 252, that this implies that $\g=[\m,\m]$. As a consequence the linear isotropy representation of $G$ on $T_{x_0}M$ is irreducible, and  equivalently so is our representation $\rho$.
 
 The vertical operators determined by $\B^V$ on the bundles associated to $p$ via our representation $\rho$ and its exterior powers $\wedge^k\rho$ are given in Theorem \ref{th:comp}  by  the function $\lambda^{\wedge^k\rho}:K\to \Lo( \wedge^k\R^n;\wedge^k\R^n)$. By Corollary \ref{co:derivative} and the discussion above they correspond to the Weitzenb\"ock curvatures of the Levi-Civita connection, and so in particular are symmetric.  To calculate them using Theorem \ref{th:comp} first use the fact that $\B^V$ restricts to $\frac{1}{2}\triangle^G$ on $p^{-1}(x_0)$ to represent it as $\frac{1}{2}\displaystyle{\sum}  \LL   _{A_j^*}\LL   _{A_j^*}$ for $A^*_j$ as in Section \ref{se-decomposition-2}.
  The computation in the proof of Corollary \ref{co:} shows that
 \begin{equation}
\lambda^{\wedge^k\rho}(u)=-{(n-2)!\over (k-1)!(n-k-1)!} c_{\wedge^k}(u),
\end{equation}
for $$c_{\wedge^k}(u)=(d\wedge^k)A_l(u)\circ (d\wedge^k) A_l'(u)$$  the Casimir element of our representation $\wedge^k\rho$ of $G$.

If $\wedge^k\rho$  is irreducible then $c_{\wedge^k}(u)$ is constant scalar. As remarked in Corollary \ref{co:} this happens when $G=SO(n)$, given our irreducibility hypothesis on the $\rho$ and then it is just ${1\over 2}n(n-1)/{n(n-1)\dots(n-k+1)\over k!}$. Thus for the  sphere $S^n(\sqrt{2})$ of radius $\sqrt{2}$, considered as $SO(n+1)/SO(n)$ we have
\begin{equation}
\lambda^{\wedge^k\rho}(u)=-\frac{1}{4}k(n-k).
\end{equation}

 \chapter{Example: Stochastic Flows}
\label{ch:flow} \index{flow!stochastic}

Before analysing stochastic flows by the methods of the previous paragraphs we describe some purely geometric constructions which will enable us to identify the semi-connections which arise in that analysis. 

\section{Semi-connections on the Bundle of Diffeomorphisms}
\label{se:diff.bundle}

Assume that $M$ is compact. For  $r\in\{1,2,\dots\}$ and
  $s>r+\dim M/2$ let 
$\D^s=\D ^s(M)$ be the
 space of  diffeomorphisms of $M$ of  Sobolev class $H^s$. See, for example,
 Ebin-Marsden \cite{Ebin-Marsden} and Elworthy \cite{Elworthy-book} for the
 detailed structure of this space. Elements of $\D ^s$ are then
 $C^r$ diffeomorphisms. The space is a topological group under composition,
and has a natural Hilbert manifold structure for which the tangent space 
$T_\theta \D^s$ at $\theta\in \D^s$ can be identified with the space
of $H^s$ maps $v: M\to TM$ with $v(x)\in T_{\theta(x)}M$, all $x\in M$. In particular
 $T_{\mathrm id}\D ^s$
can be identified with the space $H^s\Gamma(TM)$ of $H^s$ vector fields
on $M$.  For each $h\in \D ^s$ the right translation
\begin{eqnarray*}
R_h:\D ^s& \to&\D ^s  \\
R_h(f)&=&f\circ h
\end{eqnarray*}
is $C^\infty$. However the joint map
\begin{equation}
\label{regularity}
\D ^{s+r}\times \D ^s\to \D ^s
\end{equation}
is $C^r$ rather than $C^\infty$ for each $r$ in $\{0,1,2,\dots\}$.

\noindent
For $x_0\in M$ fixed, define $\pi:\D ^s \to M$ by
\begin{equation}\pi(\theta)=\theta(x_0).
\end{equation}
The fibre $\pi^{-1}(y)$ at $y \in M$ is given by: 
$\{\theta\in \D ^s: \theta(x_0)=y\}$.
Set $\D ^s_{x_0}:=\pi^{-1}(x_0)$. Then the elements of
 $\D ^s_{x_0}$ act on the right as $C^\infty$ diffeomorphisms of $\D ^s$. We can consider this as giving a principal bundle structure to $\pi: \D^s\to M$ with group $\D^s_{x_0}$, although there is the lack of regularity noted in equation (\ref{regularity}).

 \noindent 
 A smooth semi-connection on $\pi: \D^s\to M$ over a sub-bundle $E$ of $TM$ consists of a family of linear horizontal lift maps $h_\theta: E_{\pi(\theta)}\to T_\theta \D^s$, $\theta\in \D^s$, which is smooth in the sense that it determines a $C^\infty$ section of $\Lo(\pi^*E; T\D^s)\to \D^s$. In particular we have
$$h_\theta(u): M\to TM$$
with $$h_\theta(u)(y)\in T_{\theta(y)}M,$$
$ u\in E_{\theta(x_0)}$, $\theta\in \D^s$, $y\in M$. 

\noindent 

We shall relate semi-connections on $\D^s\to M$ to certain reproducing 
kernel Hilbert spaces. For this let $E$ be a smooth sub-bundle of $TM$ and
$\HH$ a Hilbert space which consists of smooth section of $E$ such that the
inclusion $\HH\to C^0\Gamma E$ is continuous (from which comes the continuity into $\HH^s\Gamma E$ for all $s>0$). Such a Hilbert space determines and is determined by its reproducing kernel \index{reproducing kernel}
$k$, a $C^\infty$ section of the bundle $\Lo (E^*;E)\to M\times M$ with fibre $\Lo (E_x^*;E_y)$ at $(x,y)$, see \cite{Baxendale76}.
 By definition,
$$k(x,-)=\rho_x^*: E_x^*\to \HH$$
where $\rho_x: \HH\to E_x$ is the evaluation map at $x$, and so
$$k(x,y)=\rho_y\rho_x^*: E_x^*\to E_y.$$

 Assume $\HH$ {\it spans $E$}
 in the sense that for each $x$ in $M$,
$\rho_x: \HH\to E_x$ is surjective. It then induces an inner product
$\langle, \rangle_x^\HH$ on $E_x$ for each $x$ via the isomorphism $\rho_x\rho_x^*: E_x^*\to E_x$.

Using the metric on $E$ the reproducing kernel $k$ induces linear maps 
$$k^{\#}(x,y): E_x\to E_y, \qquad x,y\in M,$$
with $k^{\#}(x,x)=\id$.
\begin{proposition}
\label{pr:determ-1}
A Hilbert space $\HH$ of smooth sections of a sub-bundle $E$ of $TM$ which spans $E$ determines a smooth semi-connection $h^\HH$ on
$\pi: \D^s\to M$ over $E$ by
\begin{equation}
\label{deterministic-1}
h_\theta^\HH(u)(y)=k^{\#}\Big(\theta(x_0), \theta(y)\Big) (u),
\qquad \theta \in \D^s, u\in E_{\theta(x_0)}, y \in M,
\end{equation}
for $k^{\#}$ derived from the reproducing kernel of $\HH$ as above.
In particular the horizontal lift $\tilde \alpha$ starting from $\tilde \alpha(0)=\id$, of a curve $\alpha: [0,T]\to M$, $\alpha(0)=x_0$ with $\dot \alpha(t)\in E_{\alpha(t)}$ for all $t$, is the flow of the non-autonomous ODE on $M$
\begin{equation}
\label{deterministic-2}
\dot z_t =k^{\#} \Big(\alpha(t), z_t\Big)\dot \alpha(t).
\end{equation}
The mapping $\HH\mapsto (h^\HH, \langle, \rangle ^\HH )$ from such Hilbert spaces to semi-connections over $E$ and Riemannian metrics on $E$ is injective.
\end{proposition}
\begin{proof}
From the definition of $k^{\#}$ we see $h_\theta^\HH(u)(y)$, as given by (\ref{deterministic-1}), takes values in $T_{\theta(y)}M$,  is linear in $u\in E_{\theta(x_0)}$ into $T_\theta\D^s$, and is  $\D_{x_0}^s$-invariant. Moreover,
$$T_\theta \pi\circ h_\theta^\HH(u)=h_\theta^\HH(u)(x_0)=
k^{\#}\Big(\theta(x_0), \theta(x_0)\Big) (u)=u$$
for $u \in E_{\theta(x_0)}$ and so $h_\theta^\HH$ is a `lift'.

To see that $h$ is $C^\infty$ as a section of
$\Lo(\pi^*E;T\D ^s)\to \D ^s$
 note that for each  $r\in\{0,1,2,\dots\}$ the composition map
\begin{eqnarray*}
 T_{id}\D ^{r+s}\times\D ^s &\to&  T\D ^s\\
(V,\theta)&\mapsto& T\CR_\theta(V)
\end{eqnarray*}
is a $\displaystyle{ C^{r-1}}$ vector bundle map over $\D^s$, 
being a partial derivative of the composition 
$\displaystyle{\D ^{r+s}\times
 \D ^{s}\to \D ^{s}}$.
Therefore it induces a $C^{r-1}$ vector bundle map
$Z\mapsto TR_\theta\circ Z$, for $Z: E_{\theta(x_0)}\to \HH$ and  for $\underline \HH$ the trivial $\HH$-bundle over
$\D ^s$,
by composition\\
\begin{center}
\begin{picture}(150,80)(0,0)
\put(10,80){$\Lo(\pi^*E; \underline \HH)$}
\put(72,83){\vector(1,0){63}}
\put(140,80){$\Lo(\pi^*E;T\D ^s)$}
\put(85,25){$\D ^s$}
\put(40, 73){\vector({3},-2){47}}
\put(140, 73){\vector({-3},-2){47}}
\end{picture}
\end{center}

\noindent
On the other hand $y\mapsto k(y,-)$ can be considered as a $C^\infty$
 section of  $\Lo(E; H_\gamma)\to M$ and so
$\theta\mapsto k(\theta(x_0),-)$ as a $C^\infty$ section of
 $\Lo(\pi^*E; \underline H_\gamma)$. This proves the
 regularity of $h$.

\noindent
That the horizontal lift $\tilde \alpha$ is the flow of (\ref{deterministic-2}) is immediate. To see that the claimed injectivity holds, given
$h_\theta^\HH$ observe that (\ref{deterministic-1}) determines $k^{\#}$: this is because given any $x$ in $M$ there exists a
$\C^\infty$ diffeomorphism $\theta$ such that $\theta(x_0)=x$ and for such $\theta$
\begin{equation}
\label{deterministic-3}
k^{\#}(x,z)(u)=h_\theta^\HH(u)(\theta^{-1}z).
\end{equation}
\end{proof}

\begin{remark}
We cannot expect surjectivity of the map $\HH\to h^\HH$ into the space of semi-connections on $\pi: \D^s\to M$. Indeed for $k^{\#}$
defined by (\ref{deterministic-3}) to be the reproducing kernel for some 
Hilbert space of sections of $E$ we need
\begin{enumerate}
\item [1)] $h_\theta^\HH(u)(y)\in E_{\theta(y)}$ for $u\in E_{\theta(x_0)}, y \in M$,  and a metric $\langle, \rangle $ on $E$ with respect to which the following holds:
\item[2)] for $x,y\in M$,
$$k^{\#}(x,y)=\Big(k^{\#}(y,x)\Big)^*, $$
\item[3)] For any finite set $S$ of points of $M$ and $\{\xi_a\}\in E_a$,
$a\in S$
$$\sum \Big\langle k^{\#}(a,b)\xi_a, \xi_b\Big\rangle \geqslant  0.$$
\end{enumerate}

\end{remark}
For each frame $u_0:\R^n\to T_{x_0}M$ there is a homomorphism of principal bundles
\begin{equation}
\label{PB-homo}
\begin{array}{llll}
\Psi^{u_0}: &\D^s&\to& GLM\\
&\theta &\mapsto& T_{x_0}\theta\circ u_0.
\end{array}
\end{equation}
As with connections such a homeomorphism maps a semi-connection
on $\D^s$ over $E$ to one on $GLM$. The horizontal lift maps are related by\\
\begin{picture}(150,80)(0,20)
\put(10,80){$T_\theta\D^s$}
\put(65,88){$T_\theta\Psi^{u_0}$}
\put(45,83){\vector(1,0){75}}
\put(130,80){$T_{\Psi^{u_0}(\theta)}{GLM}$}
\put(26,48){$h_\theta$}
\put(125,48){$h_{\Psi^u_0(\theta)}$}
\put(60,20){$E_{\theta(x_0)}$}
\put(75, 35){\vector({-3},2){55}}
\put(82, 35){\vector({3},2){55}}
\end{picture} \\
and if $\tilde \alpha: [0, T]\to \D^s$ is a horizontal lift of $\alpha: [0, T]\to M$ then
$$\Psi^{u_0}(\tilde \alpha(t))=T_{x_0}\tilde \alpha(t)\circ u_0, \qquad 0\leqslant t\leqslant T$$
is a horizontal lift of $\alpha$ to $GLM$.

\begin{theorem}
\label{pr:determ-2} Let $h^\HH$ be the semi-connection on $\pi: \D^s\to M$ over $E$ determined by some $\HH$ as in Proposition \ref{pr:determ-1}. Then the semi-connection induced on $GLM$, and so on $TM$, by the homeomorphism $\Psi^{u_0}$ is the adjoint $\hat \nabla$ of the metric connection which is projected on $(E, \langle, \rangle^\HH)$  by the evaluation map $(x,e)\mapsto \rho_x(e)$ from $M\times \HH\to E$, \cf  (1.1.10) in \cite{Elworthy-LeJan-Li-book}.  In particular every semi-connection on $TM$ with metric adjoint connection arises this way from some, even finite dimensional, choice of $\HH$.
\end{theorem}

\begin{proof}
Let $\alpha: [0,T]\to M$ be a $C^1$ curve with $\dot \alpha (t)\in E_{\alpha(t)}$ for each $t$. By Proposition \ref{pr:determ-1} its horizontal lift $\tilde \alpha$ to $\D^s$ starting from $\theta\in \pi^{-1}(\alpha(0))$ is the solution to
\begin{eqnarray}
{d\tilde \alpha\over dt}&=&k^{\#}\Big(\tilde \alpha(t)(x_0), \tilde \alpha(t)-\Big) \dot\alpha(t)\\
\tilde \alpha(0)&=&\theta.
\end{eqnarray}
The horizontal lift to $GLM$ is $t\mapsto T_{x_0}\tilde \alpha(t)\circ u_0$ and to $TM$ through $v_0\in T_{\theta(x_0)}M$, \ie the parallel translation $\{\parals_t(v_0): 0\leqslant t\leqslant T\}$ of $v_0$ along $\alpha$, is given by
$$\parals_t(v_0)=T_{x_0}\tilde \alpha(t)\circ (T_{x_0}\theta)^{-1}(v_0)
=T_{\alpha(0)}\Big(\tilde \alpha(t)\circ\theta^{-1}\Big)(v_0).$$
However this is $T_{\alpha(0)}\pi_t(v_0)$ for $\{\pi_t: 0\leqslant t\leqslant T\}$ the solution flow of  
$${dz_t\over dt}=k^{\#}\Big( \alpha(t), z(t)\Big) \dot\alpha(t)\\
$$
which by Lemma 1.3.4 of \cite{Elworthy-LeJan-Li-book} is the parallel translation of the adjoint of the associated connection (in \cite{Elworthy-LeJan-Li-book} $k^{\#}$ is denoted by $k$).

The fact that all such semi-connections on $TM$ arise from some finite dimensional $\HH$ comes from Narasimhan-Ramanan
\cite{Narasimhan-Ramanan61} as described in \cite{Elworthy-LeJan-Li-book}, or more directly from Quillen \cite{Quillen}
\end{proof}


\section{Semi-connections Induced by Stochastic Flows}

From Baxendale \cite{Baxendale84} we know that a $C^\infty$ stochastic flow $\{\xi_t:t\geqslant  0\}$ on $M$, i.e. a Wiener process on
 $\D ^\infty:=\cap_s \D ^s$, can be considered as the solution  flow of a stochastic differential equation on $M$ driven by a possibly infinite dimensional noise. Its one point motions form a
 diffusion process on $M$ with generator $\A $, say. The noise
comes from the Brownian  motion $\{W_t: t\geqslant  0\}$ on
$\HH^s\Gamma(TM)$ determined by a  Gaussian
measure $\gamma$ on $\HH^s\Gamma(TM)$. (In our $C^\infty$ case they
lie on $\HH^\infty (TM):=\cap_s\HH^s\Gamma(TM)$.) We will
take $\gamma$ to be mean zero and so we may have a drift  $A$  in
$\HH^\infty (TM)$. The  stochastic flow $\{\xi_t:  t\geqslant  0\}$
 can then be taken  to be the solution of the right invariant stochastic
 differential equation on $\D ^s$
\begin{equation}
\label{diffeo-1}
d\theta_t=TR_{\theta_t}\circ dW_t+TR_{\theta_t}(A)dt
\end{equation}
with $\xi_0$ the identity map $\id$. In particular it determines a right invariant generator $\B$ on $\D^s$.

For fixed $x_0$ in $M$ the one point motion $x_t:=\xi_t(x_0)$ solves
\begin{equation}
\label{diffeo-2}
dx_t=\circ dW_t(x_t)+A(x_t)dt.
\end{equation}
We can write
(\ref{diffeo-2}) as
\begin{equation}
\label{diffeo-3}
dx_t= \rho_{x_t}\circ dW_t+ A(x_t)dt.
\end{equation}

Thus $\pi(\xi_t)=\xi_t(x_0)=x_t$. For a map $\theta$ in $\D ^s$, the solution $\xi_t\circ \theta$ to
(\ref{diffeo-1}) starting at $\theta$ has
$\pi(\xi_t\circ \theta)=\xi_t(\pi(\theta))$, the solution to (\ref{diffeo-3}) starting from $\pi(\theta)$, and we see that the diffusions
are $\pi$-related (\cf \cite{Elworthy-book}), and $\A$ and $\B$ are intertwined by $\pi$.

The measure $\gamma$ corresponds to a reproducing kernel Hilbert space,
 $H_\gamma$ say, or equivalently to an abstract Wiener space
 structure $i: H_\gamma\to \HH^s\Gamma(TM)$ with $i$ the inclusion (although $i$ may not have dense image). 
Then
$$\sigma_\theta^\B: (T_\theta\D ^s)^*
\to T_\theta \D ^s$$
is right invariant and determined at $\theta=id$ by the canonical
isomorphism $H_\gamma^*\simeq H_\gamma$ through the usual map $j=i^*$
$$(\HH^s\Gamma(TM))^* \stackrel {j}{\hookrightarrow}
H_\gamma^*\simeq H_\gamma \stackrel{i}{\hookrightarrow}
\HH^s\Gamma(TM),$$
i.e.
$$\sigma_{id}^\B=i\circ j.$$
This shows $H_\gamma$ is the image of $\sigma_{\id}^\B$ with induced metric.
In this situation our cohesiveness condition on $\A $
becomes the assumption that there is a $C^\infty$ subbundle $E$ of $TM$ such that $H_\gamma$ consists of sections of $E$ and spans $E$, and $A$ is a section of $E$.
Let $\langle, \rangle_y$ be the inner product on $E_y$ induces by $H_\gamma$.

The reproducing kernel \index{reproducing kernel} $k$ of $H_\gamma$  is the covariance of $\gamma$ and :
$$k^{\#}(x,y)v=\int_{U\in\HH^s\Gamma(E)}
\Big\langle U(x),v \Big\rangle_x U(y)\; d\gamma(U),
\hskip 25pt v\in E_x; \hskip 3pt  x,y \in M.$$
Analogously to Lemma \ref{factorization-connection}  we have the commutative diagram

\begin{picture}(200,100)(0,0)
\put(40,70) {$(T_\theta \D ^s)^*$}
\put(4,12){$T_{\theta(x_s)}^*M\to E_{\theta(x_0)}^*$}

\put(90,80){$j\circ (T\CR_\theta)^*$}
\put(100, 27){$k(\theta(x_0),-)$}

\put(80,72){\vector(1,0){85}}
\put(186,72){\vector(1,0){85}}
\put(85,20){\vector(1,0){85}}
\put(186,20){\vector(1,0){85}}

\put(56,24){\vector(0,1){40}}
\put(175,24){\vector(0,1){40}}
\put(272,68){\vector(0,-1){40}}

\put(170,70){$H_\gamma$}
\put(170,15){$H_\gamma$}

\put(205,77){$T\CR_\theta\circ i$}
\put(205, 27){$\rho_{\theta(x_0)}$}

\put(270, 72){$T_\theta\D ^s$}
\put(270,12){$E_{\theta(x_0)}\hookrightarrow T_{x_0}M$}

\put(60,42){$(T_\theta \pi)^*$}
\put(180,42){$\ell_\theta$}
\put(275,42){$T_\theta\pi=\rho_{x_0}$}
\end{picture}\\
with $\ell_\theta$ uniquely determined under the extra condition
$$\ker \ell_\theta={\ker}_{\rho_{\theta(x_0)}}.$$

Writing $K: M\to \Lo(H_\gamma; H_\gamma)$ for the map giving
the projection $K(x)$ of $H_\gamma$ onto $\ker\rho_{x}$ for each $x$
in $M$ and letting $K^\perp (x)$ be the projection onto
$[\ker \rho_x]^\perp$ we have
$$\ell_\theta=K^\perp\big(\theta(x_0)\big),$$
(agreeing with the note following Lemma  \ref{factorization-connection}), and so 
$$\ell_\theta(U)=k^{\#}\Big( \theta(x_0),-\Big)U(\theta(x_0)), \qquad U\in \HH_\gamma.$$ Note that the formula 
$$K^\perp(y)(U)=k^{\#}(y,-)U(y)$$
 for $U$ in $\HH_\gamma$ determines
an extension 
$K^\perp(y): \Gamma E\to \HH_\gamma$. 
We then define $K(y)U=U-K^\perp(y)U$. Note that $\rho_y(K(y)U)=0$ for all $U$
in $\Gamma E$.

  The horizontal lift map determined by $\B$ as in Proposition \ref{pr:h-map} is therefore given by
\begin{equation}
\begin{aligned}
h_\theta : E_{\theta(x_0)}&\to \;\;T\CR_\theta(H_\gamma)\;\;
\qquad \subset \quad T_\theta \D ^s\\
h_\theta(u)&=T\CR_\theta \,\ell_\theta \Big[k^{\#}(\theta(x_0),-)u\Big],
\end{aligned}
\end{equation}
for $\theta\in \D ^s$. Consequently
\begin{equation}\label{diffeo-5}
h_\theta(u)(y)=k^{\#}\Big(\theta(x_0), \theta(y)\Big)(u).
\end{equation}

Comparing this with formula (\ref{deterministic-1}) we have
\begin{proposition}
The semi-connection $h$ determined on $\pi: \D^s\to M$ by the equivariant diffusion operator $\B$ is just that given by the reproducing kernel Hilbert space $H_\gamma$ of the stochastic flow which determines $\B$, \ie $$h=h^{H_\gamma}.$$
\end{proposition}

The horizontal lift $\{\tilde x_t: t\geqslant  0\}$ of the one point motion $\{x_t: t\geqslant  0\}$ with
 $\tilde x_0=\id$ is the solution to
\begin{equation}
\label{lift-10-1}
d\tilde x_t=k^{\#}\Big(\tilde x_t(x_0), \tilde x_t-\Big)\circ dx_t;
\end{equation}
which in a more revealing notation is:
\begin{equation}
\label{lift-10-2}
d\tilde x_t=TR_{\tilde x_t}\Big(K^\perp(\tilde x_t(x_0))\circ dW_t\Big)
+TR_{\tilde x_t}\Big(K^\perp (\tilde x_t(x_0))A\Big).
\end{equation}
Equivalently  $\{\tilde x_t: t\geqslant  0\}$ can be considered as the solution
 flow of the  non-autonomous stochastic differential equation on $M$
\begin{eqnarray*}
dy_t&=&k^{\#}\big(x_t,y_t\big)\circ dx_t
\end{eqnarray*}
\ie
\begin{equation}
\label{lift-10-4}
dy_t=\Big(K^\perp(x_t)\circ dW_t\Big)(y_t)+K^\perp(x_t)(A)(y_t).
\end{equation}

The standard fact that the solution to such equation as
(\ref{lift-10-4}) starting at $x_0$ is just $\{x_t:t\geqslant  0\}$, i.e.
that $\tilde x_t(x_0)=x_t$ reflects the fact that $\tilde x_\cdot$ is a
lift of $x_\cdot$. The lift through $\phi\in \D_{x_0}^s$ is
just $\{\tilde x_t\circ \phi: t\geqslant  0\}$.

\begin{remark}
\label{re:flow}
If our solution flow is that of an SDE
$$dx_t=X(x_t)\circ dB_t+A(x_t)dt$$
for $X(x): \R^m\to TM$ arising, for example, from H\"ormander form\index{H\"ormander!form} representation of $\A$ as in \S\ref{se:expansion} above the relationships with the notation in this section is as follows:
$H_\gamma=\{X(\cdot)e: e\in \R^m\}$ with inner product induced by the surjection $\R^m\to H_\gamma$. If $Y_x=[X(x)|_{\ker X(x)^\perp}]^{-1}$ then $k^{\#}(y,-): E_y\to H_\gamma$ is
$$k^{\#}(y,-)u=X(-)Y_y(u), \qquad u\in E_y.$$
 Also $K^\perp(y):\Gamma E\to H_\gamma$ is $K^\perp(y)U=X(-)Y_y(U(y))$.
\end{remark}

\begin{remark}
The reproducing kernel Hilbert space $H_\gamma$ determines the stochastic flow and so by the injectivity part of Proposition \ref{pr:determ-1} the semi-connection together with the generator
$\A$ of the one-point motion determines the flow, or equivalently
the operator $\B$. This is because the symbol of $\A$ again gives the metric on $E$ which together with the semi-connection determines $H_\gamma$ by Proposition \ref{pr:determ-1}. The generator $\A$ then determines the drift $A$. A consequence is that the horizontal lift $\A^H$ of $\A$ to $\D^s$ determines the flow (and hence $\B$, so $\B^V$ really is redundant). 
\end{remark}

To see this directly note that given any cohesive $\A$ on $M$ and $\D_{x_0}^s$-equivariant $\A^H$  on $\D^s$ over $\A$, with no vertical part, there is at most one vertical $\B^V$ such that $\A^H+\B^V$ is right invariant. This follows from the following lemma

\begin{lemma}
Suppose $\B^1$ is a diffusion operator on $\D^s$ which is vertical and right invariant then $\B^1=0$. 
\end{lemma}

\begin{proof}
By Remark \ref{re:3.4} (i) the image $\Epsilon_\theta$, say, of 
$\sigma_\theta^{\B'}$ lies in $VT_\theta \D^s$ for  $\theta\in \D^s$ and so
 if $V\in \Epsilon_\theta$. On the other hand, by right invariance $\Epsilon_\theta =TR_\theta(\Epsilon_\id)$. 
Therefore if $V\in \Epsilon_\id$ then $V(\theta_{x_0})=0$ all
 $\theta\in \D^s$ and so $V\equiv 0$. Thus $\Epsilon_\id=\{0\}$ and
  by right invariance, $\B^1$ must be given by some vector field $Z$ 
  on $\D^s$. But $Z$ must be vertical and right invariant, so again 
  we see $Z\equiv 0$.
\end{proof}

Proposition \ref{pr:h-funct-2}  applies to the 
homomorphism $\Psi^{u_0}: \D^s\to GL(M)$ of (\ref{PB-homo}).  From this and Theorem \ref{pr:determ-2} we see that the semi-connection $\nabla$ on $GLM$ determined by the generator of the derivative flow in \S\ref{se:deri-flow} is the adjoint $\hat \nabla$ of the connection $\breve \nabla$, so giving an alternative proof of Theorem \ref{th:conn-deri} above.
Proposition \ref{pr:h-funct-2} also gives a relationship between
the curvature and holonomy group of $\hat \nabla$  and those of the connection induced by the flow on $\D^s\stackrel{\rho_{x_0}}{\rightarrow} M$.

We can summarize our decomposition results as applied to these stochastic flows in the following theorem. The skew product decomposition was already described in \cite{Elworthy-LeJan-Li-principal}
for the case of solution flows of SDE of the form (\ref{sde}), and in particular with finite dimensional noise: however the difference is essentially that of notation, see Remark \ref{re:flow} above.

\begin{theorem}
\label{th:flow-decom}
Let $\{\xi_t: t\geqslant  0\}$ be a $C^\infty$ stochastic flow on a compact
 manifold  $M$. Let $\A$ be the generator of the one point motion on $M$
 and $\B$  the generator of the right invariant diffusion on $\D^s$
 determined by  $\{\xi_t: t\geqslant  0\}$. Assume $\A$ is strongly
cohesive. Then there is a unique decomposition $\B=A^H+\B^V$ for
$A^H$ a diffusion operator which has no vertical part in the sense of definition \ref{no-vert}  and $\B^V$ a diffusion operator which is along
the fibres of $\rho_{x_0}$, both  invariant under the right action of
$\D^s_{x_0}$. The diffusion process $\{\theta_t: t\geqslant  0\}$ and
$\{\phi_t: t\geqslant  0\}$ corresponding to $A^H$ and $\B^V$ respectively
 can be represented as solutions to
\begin{equation}
\label{representation-Aoplift}
d\theta_t=TR_{\theta_t}\Big(K^\perp(\theta_t(x_0))\circ dW_t\Big)
+TR_{\theta_t}\Big(K^\perp(\theta_t(x_0))A\Big)
\end{equation}
and
\begin{equation}
\label{B-minus-Aoplift}
d\phi_t=TR_{\phi_t}\Big(K(z_0)\circ dW_t\Big) +TR_{\phi_t}\Big(K(z_0)A\Big)
\end{equation}
for $z_0=\phi_0(x_0)=\phi_t(x_0)$.
There is the corresponding skew-product decomposition of the given stochastic flow
$$\xi_t=\tilde x_t g_t^{x_\cdot}, 0\leqslant t<\infty$$
where $\{\tilde x_t: t\geqslant  0\}$ is the horizontal lift of the one point notion $\{\xi_t(x_0): t\geqslant  0\}$ with $\tilde x_0=\id_M$ and for $\PP_{x_0}^\A$-almost all $\sigma: [0, \infty)\to M$, $\{g_t^\sigma: t\geqslant  0\}$ is a $\D_{x_0}^s$-valued process independent of $\{\tilde x_t: t\geqslant  0\}$ and satisfying 
\begin{eqnarray*}
dg_t^\sigma&=&T\tilde \sigma_t^{-1} \rho(\tilde \sigma_tg_t^\sigma-)\Big(K(\sigma_t)\circ dW_t\Big)
+T\tilde \sigma_t^{-1} \rho(\tilde \sigma_tg_t^\sigma-)\Big(K(\sigma_t)
A\Big)\\
\tilde \sigma_0& = &{\id}_M 
\end{eqnarray*}
where $\tilde \sigma$ is the horizontal lift of $\tilde \sigma$ to $\D^s$ with $\tilde \sigma_\cdot$ in the horizontal life of $\tilde \sigma$ to $\D^s$ with $\tilde \sigma_0=\id_M$.
\end{theorem}
\begin{remark}
As in \cite{Elworthy-LeJan-Li-book} we could rewrite the terms such as
$K(\sigma_t)\circ dW_t$ and $K^\perp (\sigma_t)\circ dW_t$ above as It\^o differentials which can be written as
\begin{eqnarray*}
K(\sigma_t) dW_t& = &\tilde \paral_t(\sigma_\cdot) \, d\beta_t \\
K^\perp(\sigma_t) dW_t& = &\tilde \paral_t(\sigma_\cdot)\, d\tilde B_t
\end{eqnarray*}
where $\tilde {\parals_t}(\sigma_\cdot):H_\gamma\to H_\gamma$, $0\leqslant t<\infty$, is a family of orthogonal transformations mapping 
$\ker \rho_{x_0}\to \ker\rho_{\sigma_t}$ defined for $\PP^\A_{x_0}$-almost all $\sigma: [0,\infty)\to M$ and $\{\beta_t: t\geqslant  0\}$, 
$\{\tilde B_t: t\geqslant  0 \}$ are independent  Brownian
motions, ($\beta_t$ could be cylindrical), on $\ker \rho_{x_0}$ and $[\ker \rho_{x_0}]^\perp$ respectively.

\end{remark}
\begin{proof}
Our general result give the decomposition $\B=A^H+\B^V$ into
horizontal and vertical parts. We have just proved the representation
(\ref{representation-Aoplift}) for $A^H$. To show that
$\B-A^H$ corresponds to (\ref{B-minus-Aoplift}) take an
orthonormal base $\{X^j\}$ for $H_\gamma$. Then, on a suitable domain,
\begin{equation}
\B=\frac{1}{2}\sum_j \LL_{\X^j} \LL_{\X^j} +L_\Ab
\end{equation}
for $\X^j(\theta)=TR_{\theta}(X^j)$ and $\Ab=TR_{\theta}(A)$, while,
by (\ref{representation-Aoplift}),
\begin{equation}
A^H=\frac{1}{2}\sum_j\LL_{\Y^j}\LL_{\Y^j}+\LL_\Bb
\end{equation}
for $\Y^j(\theta)=TR_\theta \left(K^\perp(\theta(x_0)X^j\right)$,
$\Bb=TR_\theta(K^\perp(\theta(x_0))A)$.

Define vector fields $\Z^j$, $\C$ on $\D^s$ by
\begin{eqnarray*}
\Z^j(\phi)&=&TR_{\phi}\left(K(\phi(x_0))X^j\right), \hskip 20pt \hbox{and}\\
\C(\phi)&=&TR_{\phi}\left(K(\phi(x_0))A\right), \hskip 20pt \hbox{for }
\phi\in \D^s.
\end{eqnarray*}
Then $\Ab=\Bb+\C$ and $\X^j=\Y^j+\Z^j$ each $j$. Moreover
 $$\sum_j\LL_{\Y^j}\LL_{\Z^j}+\sum_j\LL_{\Z^j}\LL_{\Y^j}=0$$
by Lemma \ref{lemma-flows} which follows below. This shows that
\begin{equation}
\B^V=\frac{1}{2} \sum_j\LL_{\Z^j}\LL_{\Z^j}+\LL_\C.
\end{equation}
Thus the diffusion process from $\phi_0$ corresponding to $\B^V$
can be represented by the solution to
\begin{equation}
d\phi_t=TR_{\phi_t}\left(K(\phi_t(x_0)\circ dW_t)\right)+
TR_{\phi_t}\left(K(\phi_t(x_0)A)\right)dt.
\end{equation}
If we set $z_t=\rho_{x_0}(\phi_t)=\phi_t(x_0)$. We obtain, via
It\^o's formula
$$z_t=\rho_{z_t}\left(K(z_t)\circ dW_t\right)
+\rho_{z_t}\left(K(z_t)A\right),$$
i.e. $dz_t=0$. Thus $\phi_t(x_0)=z_0$ and (\ref{B-minus-Aoplift})
 holds.

The skew product formula is seen to hold by calculating the stochastic differential of $\tilde x_t g_t^{\tilde x}$ using (\ref{lift-10-2}) to see it satisfies the SDE (\ref{diffeo-1}) for $\{\xi_t: t\geqslant  0\}$.
\end{proof}

\begin{lemma}
\label{lemma-flows}
$$\sum_j \LL_{\Y^j}\LL_{\Z^j}+\LL_{\Z^j}\LL_{\Y^j}=0.$$
\end{lemma}

\begin{proof}
Since, for fixed $\theta$, we can choose our basis $\{X^j\}$, such that
either
 $\Y^j(\theta)=0$ or $\X^j(\theta)=0$,
and since for $f:\D^s\to\R$ we can write
$$df\Big(\Z^j(\theta)\Big)=\Big(df\circ TR_{\theta}\Big)\Big(K(\theta(x_0))X^j\Big)$$
and
$$df(\Y^j(\theta))=(df\circ TR_{\theta})
\Big(K^\perp(\theta(x_0))X^j\Big), \qquad \theta\in \D^s,$$
it suffices to show that
\begin{equation}
\label{communicative-16}
\sum_j\Big\{(dK^\perp)_{\theta(x_0)}
\Big(\Z^j(\theta)(x_0)\Big) X^j
+(dK)_{\theta(x_0)}\left(\Y^j(\theta)(x_0)\right)X^j
\Big\}=0,
\end{equation}
for all $\theta\in \D^s$.

Now $K^\perp(y)K(y)=0$ for all $y\in M$. Therefore
$$(dK^\perp)_y(v)K(y)+K^\perp(y)(dK)_y(v)=0, \qquad
\forall  v\in T_xM, x\in M.$$
 Writing
 $$X^j=K\big(\theta(x_0)\big)X^j +K^\perp\big(\theta(x_0)\big)X^j$$ this reduces
the right hand side of (\ref{communicative-16})  to
\begin{eqnarray*}&&\sum_j \big(dK^\perp\big)_{\theta(x_0)}
\Big(\Z^j(\theta)(x_0)\Big) \Big(K^\perp(\theta(x_0))X^j\Big) \\
&& +(dK)_{\theta(x_0)}\Big(\Y^j(\theta)(x_0)\Big)
\Big(K(\theta(x_0))X^j\Big)=0
\end{eqnarray*}
with our choice of basis this clearly vanishes, as required.
\end{proof}



\section{Semi-connections on Natural Bundles}

Our bundle $\pi: \Diff M\to M$ can be considered as a universal natural
bundles over $M$, and a connection on it induces a connection on each natural bundle over $M$. Natural bundles are discussed in 
Kolar-Michor-Slovak
\cite{Kolar-Michor-Slovak}), they include bundles such as jet bundles as well as the standard tensor bundles. For example let $G_n^r$ be the Lie group of r-jets of diffeomorphisms $\theta: \R^n\to \R^n$ with $\theta(0)=0$ for positive integer $r$. An `r-th order frame' $u$ at a point $x$ of $M$ is the r-jet at $0$ of some $\psi: U\to M$ which maps an open set $U$ of $\R^n$ diffeomorphically onto an open subset of $M$ with $0\in U$ and $\psi(0)=x$. Clearly $G_n^r$ acts on the right of such jets, by composition. From this we can define the rth order frame bundle $G^r_nM$ of $M$ with group $G_n^r$.

If we fix an rth order frame $u_0$ at $x_0$ we obtain a homomorphism of principal bundles
\begin{eqnarray*}
\Psi^{u_0} &:& D^s\to G_n^rM  \\
&&\theta \mapsto  j_{x_0}^r(\theta)\circ u_0
\end{eqnarray*}
as for $GLM$ (which is the case $r=1$) with associated group homomorphism $\D_{x_0}^s\to G_n$ given by $\theta\to u_0^{-1} \circ j_{x_0}^r(\theta)\circ u_0$. As for the case $r=1$ there is a diffusion operator induced by the flow on $G_n^rM$ and we are in the situation of Proposition \ref{pr:determ-2}. The behaviour of the flow induced on $G_n^2M$ is essentially that of $j^2_{x_0}(\xi_t)$ and so relevant to the effect  on the curvature of sub-manifolds of $M$ as they are moved by the flow \eg see Cranston-LeJan \cite{Cranston-LeJan99}, Lemaire
\cite{Lemaire}.

Alternatively rather having to choose some $u_0$ we see that $G_n^rM$ is (weakly) associated to $\pi: \D^s\to M$ by taking the action
of $\D_{x_0}^s$ on $(G_x^rM)_{x_0}$ by
$$(\theta, \alpha) \mapsto j_{x_0}^r(\theta)\circ \alpha.$$
As a geometrical conclusion we can observe
\begin{theorem}
\label{pr:universal}
Any classifying bundle homomorphism

$$OM\stackrel{\Phi}{\to} V(n,m-n)$$
$$M\stackrel{\to}{\Phi_0} G(n, m-n)$$
for the tangent bundle to a compact Riemannian manifold $M$, (where $G(n,m-n)$ is the Grassmannian of $n$-planes in $\R^m$ and $V(n, m-n)$ the corresponding Stiefel manifold) induces not only a metric connection on $TM$ as the pull back of Narasimhan and Ramanan's universal connection $\varpi_U$, but also a connection on $\Pi: \D^s\to M$. The latter induces a connection on each natural bundle over $M$ to form a consistent family; that induced on the tangent bundle is the adjoint of $\Phi^*(\varpi_U)$. 
The above also holds with smooth stochastic flows replacing classifying bundle homomorphisms, and the resulting map from
stochastic flows to connections on $\pi: \D^s\to M$ is injective.

\end{theorem}
 
\begin{proof}
It is only necessary to observe that $\Phi$ determines and is determined by a surjective vector bundle map $X: M\times \R^m\to TM$
(\eg see \cite{Elworthy-LeJan-Li-book}, Appendix 1). This in turn determines a Hilbert space $\HH$ of sections of $TM$ as in Remark 
\ref{re:flow} so we can apply Proposition \ref{pr:determ-1} and \ref{pr:determ-2}.
\end{proof}

Some of the conclusions of Theorem \ref{pr:universal} are explored further in \cite{Elworthy-abel}.
\begin{remark}
This injectivity result in Proposition \ref{pr:universal} implies that all properties of the flow can, at least theoretically, be obtainable from the induced connection on $\D^s$.
\end{remark}

\subsubsection{Flows on Non-compact Manifolds}
\label{non-compact}

In general if $M$ is not compact we will not be able to use the Hilbert 
manifolds $\D^s$, or other Banach manifolds without growth conditions on the coefficients of our flow. One possibility could be use the space $\Diff M$ of all smooth diffeomorphisms using the Fr\"{o}licher-Kriegl differential calculus as in Michor \cite{Michor}. In order to do any stochastic calculus we would have to localize and use Hilbert manifolds (or possibly rough path theory). The geometric structures would nevertheless be on $\Diff M$. This was essentially what was happening in the compact case. However it is useful to include partial flows of stochastic differential equations which are not strongly complete, see Kunita \cite{Kunita-book} or Elworthy\cite{Elworthy-book}. For the partial solution flow $\{\xi_t: t<\tau\}$ of an SDE as in Remark \ref{re:flow}
we obtain the decomposition in Theorem \ref{th:flow-decom} but now only for $\xi_t(x)$ defined for $t<\tau(x,-)$. This can be proved from the compact versions by localization as in Carverhill-Elworthy \cite{Carverhill-Elworthy} or Elworthy \cite{Elworthy-book}.

\chapter{Appendices}
\section{Girsanov-Maruyama-Cameron-Martin Theorem}
\label{se-GMCM theorem}
To apply the Girsanov-Maruyama theorem it is often thought necessary to verify some condition such as Novikov's condition to ensure that the exponential (local) martingale arising as Radon -Nikodym derivative  is a true martingale. In fact for conservative diffusions this is automatic, and we give a proof of this fact here since it is not widely appreciated. The proof is along the lines of that given for elliptic diffusions in \cite{Elworthy-book} but with the uniqueness of the martingale problem replacing the uniqueness of minimal semi-groups used in \cite{Elworthy-book}. See also  [\cite{Leandre-no-prob}]. On the way we relate  the expectation of the exponential local martingale to the probability of explosion of the trajectories of the associated diffusion process: a special case of this appeared in \cite{McKean-book}.
Let $\B$ be a conservative diffusion operator on a smooth manifold $N$. For fixed $T>0$ and $y_0\in N$ let $\mathbf{ P}_{y_0.}=\mathbf{ P}^\B_{y_0}$ denote the solution to the martingale problem for $\B$ on $C_{y_0}([0,T];N^+)$. Using the notation of  chapter \ref{ch:intertwined}, let $b$ be a vector field on $N$ for which there is a $T^*N$-valued process $\alpha$  in $L^2_ {\B,\loc}$ such that $$ 2\sigma^{\B}(\alpha_t)=b(y_t)\qquad 0\leqslant t\leqslant T $$ for $\mathbf{ P}_{y_0}$ almost all $y_.\in C_{y_0}([0,T];N^+)$. Set
$$ Z_t= \exp\{M^\alpha_t-\frac{1}{2}\left\langle M^\alpha\right\rangle_t\} \qquad 0\leqslant t\leqslant T. $$
This  exists by the non-explosion of the diffusion process generated  by $\B$, and is a local martingale with $\E Z_t\leqslant 1$. 

 For bounded measurable $f:N\to \R$ define $Q_tf(y_0)=\E^\B_{y_0}[Z_tf(y_t)]$ for $y_0\in N$. Since the pair $(y_.,Z_.)$ is Markovian this determines a semi-group on the space of bounded measurable functions with  corresponding  probability measures $\{\mathbf{ Q}_{y_0}\}_{y_0\in N}$.
 
 \begin{proposition} The family $\{\mathbf{ Q}_{y_0}\}_{y_0\in N}$ is a solution to the martingale problem for the operator $\B+b$.
 \end{proposition}
 
 \begin{proof} 
 Let $f:N\to \R$ be $C^\infty$ with compact support. We must show, for arbitrary ${y_0\in N}$, that 
 $$f(y_t)-f(y_0)-\int_0^t(\B+b)f(y_s)ds \qquad 0\leq t\leqslant T$$
 is a local martingale under $\mathbf{ Q}_{y_0}$.
 For this first note that $Z_.$ satisfies the usual stochastic equation which in our notation becomes:
$$ Z_t=1+M^{Z\alpha}_t,  \qquad  \qquad 0\leqslant t\leqslant T$$ 
Now use Ito's formula and the definition of $M^\alpha$ to see that 

\begin{equation}
f(y_t)Z_t=f(y_0)+M^{Z (df)_{y_.}}_t+M^{fZ\alpha}_t+\int_0^t\B f(y_s) Z_s ds +\left\langle M^{df},M^{Z\alpha}\right\rangle_t.
\end{equation}
Now \begin{eqnarray}
\left\langle M^{df},M^{Z\alpha}\right\rangle_t&=&2\int_0^tdf\left(\sigma^\B_{y_s}(Z_s\alpha_s)\right)ds\\
&=&\int_0^t df\big (Z_sb(y_s)\big)ds.
 \end{eqnarray}
 Thus 
 $$f(y_t)Z_t-f(y_0)- \int_0^t\B f(y_s) Z_s ds -\int_0^{t} df\big(Z_s b(y_s)\big)ds, \qquad 0\leq t\leqslant T,$$
  is a local martingale under $\mathbf{ P}^\B_{y_0}$ and so there is a sequence $\{\tau_n\}_n$  of stopping times, increasing to $T$, such that if $\phi:C_{y_0}([0,T];N^+)\to \R$ is $\F^{y_0}_r$-measurable and bounded then, using the definition of $\mathbf{ Q}$ and Fubini's theorem, if $0\leqslant r\leqslant t \leqslant T$,
  \begin{eqnarray*}
&&\E^\mathbf{ Q}_{y_0} \left[\left( f(y_{t\wedge\tau_n})-\int_0^{t\wedge\tau_n}(\B+b)(f)(y_s)ds \right)\phi\right]\\
&=&\E^\B_{y_0}\left[ \left(f(y_{t\wedge\tau_n})Z_{t\wedge\tau_n}- \int_0^{t\wedge\tau_n}(\B+b)(f)(y_s) Z_s ds  \right)\phi\right]\\ 
&=&\E^\B_{y_0}\left[ \left(f(y_{r\wedge\tau_n})Z_{r\wedge\tau_n}- \int_0^{r\wedge\tau_n}(\B+b)(f)(y_s) Z_s ds\right)\phi\right].
\end{eqnarray*}
giving the required martingale property.
 \end{proof}
 Since $Q_t(1)=\E Z_t$ we immediately obtain the following corollary and a theorem:
\begin{corollary}. Suppose further that uniqueness of the martingale problem holds for $\B+b$, e.g suppose $b$ is locally Lipschitz \cite{Ikeda-Watanabe}. Then $$\E^\B_{y_0}Z_t$$
 is the probability that the diffusion process  from $y_0$ generated by $\B+b$ has not exploded by time $t$.
\end {corollary}

\begin {theorem}\label{th-GMCM}
Suppose the diffusion operator $B$ and its perturbation $\B+b$ by a locally Lipschitz vector field $b$ on $N$ are both conservative. Assume that $\B+b$ is cohesive or more generally that there is a locally bounded, measurable one-form $b^\#$ on $N$  such that $$ 2\sigma^{\B}_y(b^\# _y)=b(y),  \qquad \qquad y\in N. $$
Then $$\exp\big(M^{b^\#}_t-\frac{1}{2}\big\langle M^{b^\#}\big\rangle_t\big), \qquad 0\leqslant t\leqslant T$$ is a  martingale under $\mathbf{ P}^\B$ and for each $y_0\in N$ the measures $\mathbf{ P}_{y_0}^\B$ and $\mathbf{ P}_{y_0}^{\B+b}$ on $C_{y_0}([0,T];N)$ are equivalent with 
$$ \frac{d\mathbf{ P}_{y_0}^{\B+b}}{d\mathbf{ P}_{y_0}^{\B}}=\exp\big(M^{b^\#}_T-\frac{1}{2}\big\langle M^{b^\#}\big\rangle_T\big).$$
\end{theorem}

\section{Stochastic differential equations for degenerate diffusions}

Let $\B$ be a (smooth) diffusion diffusion operator on $N$. If its symbol $\sigma^\B: T^*N\to TN$ does not have constant rank there may be no smooth, or even $C^2$, factorisation 
$$T^*N\stackrel{X^*}\to \underline \R^m \stackrel{X}\to TN$$
of $\sigma_x^\B$ into $X(x)X^*(x)$ for $X: N\times \R^m\to TN$, as usual, for any finite dimensional $m$. \cite{}. A factorisation with $X: N\times H\to TN$, for $H$ a separable Hilbert space, can be found following Stroock and Varadhan, Appendix in  \cite{Stroock-Varadhan}.
, with the property that $X$ is continuous and each vector field $X^j$ is $\C^\infty$, where $X^j(x)=X(x)(e^j)$ for an orthonormal basis $(e_j)_{j=1}^\infty$ of $H$. However it seems unclear if such an $X$ can be found with each $x\mapsto X(x)e$, $e\in H$, smooth. The following is well known:
\begin{theorem}
Let $\sigma: \R^d\to \Lo_+(\R^m; \R^m)$ be a $C^2$ map into the symmetric positive semi-definite $(m\times m)$-matrices then $\sqrt \sigma: \R^d \to \Lo_+(\R^m; \R^m)$ is locally Lipschitz .
\end{theorem}
For a proof see Freidlin \cite{Freidlin85}, page 97 in \cite{Stroock87} or Ikeda-Watanabe \cite{Ikeda-Watanabe}.

\begin{corollary}
For a $C^2$ diffusion operator $\B$ on $N$ there is a locally Lipschitz $X: \underline{R}^m\to TN$ with $\sigma^\B=XX^*$ for some $m$.
\end{corollary}
\begin{proof}
Take a smooth inclusion $TN\stackrel{i}\to {\underline \R^m}+$ as a sub-bundle (\eg by embedding $N$ in $\R^m$) and extend $\sigma^\B$ trivially to $\sigma_x^\B: N\to \Lo\big((\R^m)^*; \R^m\big)$ by 
$$(\R^m)^*  \;\; \stackrel{i_x^*} \to T_x^*N \;\; \;\stackrel{\sigma_x^\B}\to T_xN
\;\;\; \stackrel{i_x}\to \R^m$$
identifying $(\R^m)^*$ with $\R^m$ and take the square root.

Let $\nabla$ be a connection on a sub-bundle $G$ of $TN$ and let $X:\underline \R^m\to G$ be a locally Lipschitz bundle map. Let $A$ be a locally Lipschitz vector field on $N$. As in Elworthy \cite{Elworthy-book} (p184) we can form the It\^o stochastic differential equation on $N$
$$(\nabla) \qquad dx_t=X(x_t)dB_t+A(x_t)dt$$
where $(B_t)$ is a Brownian motion on $\R^m$. For given $x_0\in N$ there will be a unique maximal solution $\{x_t: 0\leqslant t<\zeta^{x_0}\}$ as usual, where by a solution we mean a sample continuous adapted process such that for all $C^2$ functions $f: N\to \R$ 
\begin{eqnarray*}
f(x_t)&=&f(x_0)+\int_0^t (df)_{x_s} X(x_s)dB_s+\int_0^t (df)_{x_s}A(x_s)ds\\
&=& \int_0^t \sum_{j=1}^m \nabla_{X^j(x_s)} (df|_G)X^j(x_s)ds.
\end{eqnarray*}
Indeed in a local coordinate $(U, \phi)$ system the equation is represented by \begin{eqnarray*}
dx_t^\phi=X_\phi(x_t^\phi)dB_t-{1\over 2}\sum_{j=1}^m \Gamma_\phi(x_t^\phi) \left(X_\phi^j(x_t^\phi)\right) \left(X^j_\phi(x_t^\phi)\right)dt+A_\phi(x_t^\phi)dt,
\end{eqnarray*}
where $X_\phi$, $X_\phi^i$, and $A_\phi$ are the local representations of $X$, $X^i$ and $A$, and $\Gamma_\phi$ is the Christoffel symbol. 

Note that the generator of the solution process has symbol $\sigma_x=X(x)X(x)^*$, $x\in N$, and so a Lipschitz factorisation of $\sigma^\B$ together with a suitable choice of $A$ will give a diffusion process with generator $\B$.

If in addition we have another generator $G$ on $N$ given in H\"ormander form\index{H\"ormander!form}
$$G=\sum_{k=1}^p \LL_{Y^k} \LL_{Y^k}+ \LL_{Y^0}$$
for $Y^0, Y^1,\dots, Y^k$ vector fields of class $C^2$ we can consider an SDE of  mixed type
$$(\nabla)\qquad
dx_t=\sum_{k=1}^p Y^k(x_t)\circ d\tilde B_t^k+X(x_t)dB_t+(Y^0(x_t)+A(x_t))dt$$
for $\tilde B^1, \dots, \tilde B^k$ independent Brownian motions on $\R$ independent of $(B_t)$. For a $C^2$ map $f:N\to \R$, a  solution $\{x_t: 0\leqslant t<\zeta^{x_0}\}$ will satisfy
\begin{eqnarray*}
f(x_t)&=&f(x_0)+\int_0^t (df)_{x_s} X(x_s)dB_s+\int_0^t \sum_{k=1}^n (df)_{x_s}(X^k(x_s))d\tilde B_s^k\\
&=& \int_0^t (\B+G)f(x_s)ds,  \qquad t<\zeta^{x_0},
\end{eqnarray*}
giving the unique solution to the martingale problem for $\B+G$.
These SDE's fit into the general frame work of the `It\^o bundle' approach of
Belopolskaya-Dalecky \cite{Belopolskaya-Dalecky}, see the Appendix of Brzezniak-Elworthy\cite{Brzezniak-Elworthy}; also see Emery \cite{Emery}(section 6.33, page 85) for a more semi-martingale oriented approach.

\end{proof}
\section{Semi-martingales \& $\Gamma$-martingales along a Sub-bundle\;} \label {se:semi-mart}
		Several of the concepts we have defined for diffusions also have versions for semi-martingales, and these are relevant to the discussion of non-Markovian observations in Chapter \ref{ch:nonlinear-filtering}. Only \emph{continuous} semi-martingales will be considered. Let $S$ denote a sub-bundle of the tangent bundle $TM$ to a smooth manifold $M$.
\begin{definition} A semi-martingale $y_s, 0\leqslant s< \tau$ is said  to be \emph{along $S$}
\index{along $S$!semi-martingale}%
 if whenever
$\phi$ is a $C^2$ one-form Êon $M$ which annihilates $S$ we have vanishing of the Stratonovich integral of $\phi$ along $y_.$:
 $$\int_0^t \phi_{y_s}\circ dy_s =0 \quad 0< t< \tau.$$
 \end{definition}
 For simplicity take $y_0$ to be a point of $M$.
 \begin {proposition}\label{pr: semi-mart-1}
 The following are equivalent:
 \begin{enumerate}
 \item the semi-martingale $y_.$ is along $S$;
 \item if $\alpha_s:0\leqslant s<\tau$ is aÊ semi-martingale with values in the annihilator of $S$ in $T^*M$,  lying over $y_.$ , then 
  $$\int_0^t \alpha_{y_s}\circ dy_s =0 \quad 0< t< \tau;$$
  \item for some, and hence any, connection $\Gamma$ on $S$ the process $y_.$ is the stochastic development \index{development!stochastic}%
   of a semi-martingale $ y^\Gamma_s , 0\leqslant s<\tau$ on the fibre $S_{y_0}$ of $S$ above $y_0$.
  \end{enumerate}
  
  If $\Lo$  is a diffusion operator then the associated diffusion processes are all along $S$ if and only if $\Lo$ is along $S$ in the sense of Section \ref{se:distribution}.
  
  \end{proposition}
  
    \begin{proof}
    Let $\paral_.$ denote the parallel translation along the paths of $y_.$ using $\Gamma$. If (3) holds then $$ dy_.=\paral_.\circ dy^\Gamma_.$$ and it is immediate that (2) is true. Also (2) trivially implies (1).
    
    Now suppose that (1) holds. Let $\Gamma$ be a connection on $E$ and $\Gamma^0$ some extension of it to a connection on $TM$, so that the corresponding parallel translation $\paral^0$ will preserve $S$ and some complementary sub- bundle of $TM$. Let $y^{\Gamma^0}$ be the stochastic anti-development of $y_.$ using  this connection.  To show (3) holds it suffices to show that  $y^{\Gamma^0}$ takes values in $S_{y_0}$. For this choose  a smooth vector bundle map $\Phi:TM \to M\times\R^m$ whose kernel is  precisely $S$ and let $\phi: TM\to \R^m$ denote its principal part and  $\phi^j, j=1,...,m$ the components of $\phi$.  These are one-forms  which annihilates $S$. Then, for each $j$
     $$0= \int_0^t \phi_s \circ dy_s=\int_0^t \phi_s \paral^0_s\circ dy_s^{\Gamma^0} \quad 0<t<\tau.$$  By the lemma below we see that $y_s^{\Gamma^0}\in S_{y_0}$ for each $s$, almost surely,  and the result follows .
    
    Finally suppose that $y_.$ is a diffusion process with generator $\Lo$. By lemma \ref{le:mart-1} we have 
    \begin{equation}
M_t^\alpha=\int_0^t \alpha_{y_s} \circ dy_s -\int_0^t  \big(\delta^\Lo \alpha\big) (y_s) ds,
\qquad 0\leqslant t<\zeta.
\end{equation}
for any $C^2$ one form $\alpha$. Suppose $\alpha$ annihilates $S$. Then if $y_.$ is along $S$  both the martingale and finite variation parts of 
$\int_0^. \alpha_{y_s} \circ dy_s $
vanish and so $\big(\delta^\Lo \alpha\big) (y_s)=0$ almost surely for almost all $0\leqslant s<\tau$. If this is true for all starting points  we see $\Lo$ is along $S$.  On the other hand if $
\Lo$ is along $S$ and $\alpha$ annihilates $S$ we see that $M^\alpha$ vanishes by its characterisation in Proposition \ref{pr:mart-1}, since $\sigma^{\Lo}$ takes values in $S$.  Thus both the martingale and finite variation parts of  $\int_0^. \alpha_{y_s} \circ dy_s$ vanish, and so the integral itself vanishes and the diffusion processes are along $S$.
\end{proof}
\begin{lemma} Suppose $z_.$ and $\Lambda_.$ are semi-martingales with values in a finite dimensional vector space $V$ and the space of linear maps $\Lo(V;W)$ of  $V$ into a finite dimensional vector space $W$, respectively. Let $V_0$ denote  the kernel of $\Lambda_s$ which  is assumed non-random and independent of $s\geqslant  0$ .  
Assume $$ \int_0^. \Lambda_s\circ dz_s=0.$$
Then $z.$ lies in $V_0$ almost surely.
\end{lemma}
\begin{proof}: We can quotient out by $V_0$ to assume that $V_0=0$, so we need to show that $z_.$
vanishes.  Giving $W$ an inner product, let $P_s:W\to \Lambda_s[V]$ be the orthogonal projection. Compose this with the inverse of $\Lambda_s$ considered as taking values in $ \Lambda_s[V]$, to obtain an $\Lo(W;V)$-valued semi-martingale $\tilde{ \Lambda}_.$ formed by left inverses of $\Lambda_.$.
By the composition law for Stratonovich integrals 
\begin{equation}
z_t=\int_0^t dz_s=\int_0^t\tilde{ \Lambda}_s \Lambda_s \circ dz_s
=\int_0^t\tilde{ \Lambda}_s \circ d\big(\int_0^s\Lambda_r \circ dz_r\big)=0
\end{equation}
as required.
\end{proof}
Let $\Gamma$ be a connection on $S$.  Note that by the previous proposition any semi-martingale $y_.$ which is along $S$ has a well defined anti-development $y^\Gamma$, say , which is a semi-martingale in $S_{y_0}$.
\begin {definition} 
An $M$ -valued semi-martingale is said to be a $\Gamma$-martingale \index{$\Gamma$-martingale}
 if its anti-development using $\Gamma$ is a local martingale.
\end{definition}
Also we can make the following definition of an Ito integral of a differential form, using the analogue of a characterisation by Darling, \cite{Darling-84}, for the case $S=TM$;
\begin{definition} If $\alpha_.$ is a predictable process with values in $T^*M$, lying over our semi-martingale $y_.$,   define its Ito integral, $\big(\Gamma\big)\int_0^t \alpha_s dy_s$ along
the paths of $y_.$ with respect to $\Gamma$ by  \begin{equation}
\big(\Gamma\big)\int_0^t \alpha_s dy_s=\int_0^t\alpha_s \parals_s dy^\Gamma
\end{equation}
whenever the (standard) Ito integral on the right hand side exists.
\end{definition}
As usual this Ito integral is a local martingale for all suitable integrands $\alpha_.$ if and only if the process $y_.$ is a $\Gamma$-martingale.


\begin{thebibliography}{10}

\bibitem{Arnaudon-Paycha}
M.~Arnaudon and S.~Paycha.
\newblock Factorisation of semi-martingales on principal fibre bundles and the
  {F}addeev-{P}opov procedure in gauge theories.
\newblock {\em Stochastics Stochastics Rep.}, 53(1-2):81--107, 1995.

\bibitem{Baudoin04}
F.~Baudoin.
\newblock Conditioning and initial enlargement of filtration on a {R}iemannian
  manifold.
\newblock {\em Ann. Probab.}, 32(3A):2286--2303, 2004.

\bibitem{Baudoin-book}
Fabrice Baudoin.
\newblock {\em An introduction to the geometry of stochastic flows}.
\newblock Imperial College Press, London, 2004.

\bibitem{Baxendale76}
P.~Baxendale.
\newblock Gaussian measures on function spaces.
\newblock {\em Amer. J. Math.}, 98(4):891--952, 1976.

\bibitem{Baxendale84}
P.~Baxendale.
\newblock Brownian motions in the diffeomorphism groups {I}.
\newblock {\em Compositio Math.}, 53:19--50, 1984.

\bibitem{Belopolskaya-Dalecky}
Ya.~I. Belopolskaya and Yu.~L. Dalecky.
\newblock {\em Stochastic equations and differential geometry}, volume~30 of
  {\em Mathematics and its Applications (Soviet Series)}.
\newblock Kluwer Academic Publishers Group, Dordrecht, 1990.
\newblock Translated from the Russian.

\bibitem{Berard-Bergery-Bourguignon}
L.~B{\'e}rard-Bergery and J.-P. Bourguignon.
\newblock Laplacians and {R}iemannian submersions with totally geodesic fibres.
\newblock {\em Illinois J. Math.}, 26(2):181--200, 1982.

\bibitem{Berger-Bryant-Griffiths}
E.~Berger, R.~Bryant, and P.~Griffiths.
\newblock The {G}auss equations and rigidity of isometric embeddings.
\newblock {\em Duke Math. J.}, 50(3):803--892, 1983.

\bibitem{Bismut-book}
Jean-Michel Bismut.
\newblock {\em Large deviations and the {M}alliavin calculus}, volume~45 of
  {\em Progress in Mathematics}.
\newblock Birkh\"auser Boston Inc., Boston, MA, 1984.

\bibitem{Boerner-book}
Hermann Boerner.
\newblock {\em Representations of groups. {W}ith special consideration for the
  needs of modern physics}.
\newblock Translated from the German by P. G. Murphy in cooperation with J.
  Mayer-Kalkschmidt and P. Carr. Second English edition. North-Holland
  Publishing Co., Amsterdam, 1970.

\bibitem{Brzezniak-Elworthy}
Z.~Brze{\'z}niak and K.~D. Elworthy.
\newblock Stochastic differential equations on {B}anach manifolds.
\newblock {\em Methods Funct. Anal. Topology}, 6(1):43--84, 2000.

\bibitem{Carverhill-88}
A.~Carverhill.
\newblock Conditioning a lifted stochastic system in a product space.
\newblock {\em Ann. Probab.}, 16(4):1840--1853, 1988.

\bibitem{Carverhill-Elworthy}
A.~P. Carverhill and K.~D. Elworthy.
\newblock Flows of stochastic dynamical systems: the functional analytic
  approach.
\newblock {\em Z. Wahrsch. Verw. Gebiete}, 65(2):245--267, 1983.

\bibitem{Cranston-LeJan99}
Michael Cranston and Yves Le~Jan.
\newblock Asymptotic curvature for stochastic dynamical systems.
\newblock In {\em Stochastic dynamics (Bremen, 1997)}, pages 327--338.
  Springer, New York, 1999.

\bibitem{CFKS87}
H.~L. Cycon, R.~G. Froese, W.~Kirsch, and B.~Simon.
\newblock {\em Schr\"odinger operators with application to quantum mechanics
  and global geometry}.
\newblock Texts and Monographs in Physics. Springer-Verlag, Berlin, study
  edition, 1987.

\bibitem{Darling-84}
R.~W.~R. Darling.
\newblock Approximating {I}to integrals of differential forms and geodesic
  deviation.
\newblock {\em Z. Wahrsch. Verw. Gebiete}, 65(4):563--572, 1984.

\bibitem{Driver92}
B.~K. Driver.
\newblock A {C}ameron-{M}artin type quasi-invariance theorem for {B}rownian
  motion on a compact {R}iemannian manifold.
\newblock {\em J. Functional Analysis}, 100:272--377, 1992.

\bibitem{Duncan}
T.~E. Duncan.
\newblock Some filtering results in {R}iemann manifolds.
\newblock {\em Information and Control}, 35(3):182--195, 1977.

\bibitem{Eberle-book}
A.~Eberle.
\newblock {\em Uniqueness and non-uniqueness of semigroups generated by
  singular diffusion operators}, volume 1718 of {\em Lecture Notes in
  Mathematics}.
\newblock Springer-Verlag, Berlin, 1999.

\bibitem{Ebin-Marsden}
D.~G. Ebin and J.~Marsden.
\newblock {Groups of diffeomorphisms and the motion of an incompressible
  fluid}.
\newblock {\em Ann. Math.}, pages 102--163, 1970.

\bibitem{Elworthy-book}
K.~D. Elworthy.
\newblock {\em Stochastic {D}ifferential {E}quations on {M}anifolds}.
\newblock LMS Lecture Notes Series 70, Cambridge University Press, 1982.

\bibitem{Elworthy-Stflour}
K.~D. Elworthy.
\newblock Geometric aspects of diffusions on manifolds.
\newblock In P.~L. Hennequin, editor, {\em Ecole d'Et\'e de Probabilit\'es de
  Saint-Flour XV-XVII, 1985-1987. Lecture Notes in Mathematics 1362}, volume
  1362, pages 276--425. Springer-Verlag, 1988.

\bibitem{Elworthy-flows}
K.~D. Elworthy.
\newblock Stochastic flows on {R}iemannian manifolds.
\newblock In {\em Diffusion processes and related problems in analysis, {V}ol.\
  {II} ({C}harlotte, {NC}, 1990)}, volume~27 of {\em Progr. Probab.}, pages
  37--72. Birkh\"auser Boston, Boston, MA, 1992.

\bibitem{Elworthy-Kendall}
K.~D. Elworthy and W.~S. Kendall.
\newblock Factorization of harmonic maps and {B}rownian motions.
\newblock In {\em From local times to global geometry, control and physics
  (Coventry, 1984/85), Pitman Res. Notes Math. Ser., 150}, pages 75--83.
  Longman Sci. Tech., Harlow, 1986.

\bibitem{Elworthy-LeJan-Li-principal}
K.~D. Elworthy, Y.~Le~Jan, and Xue-Mei Li.
\newblock Equivariant diffusions on principal bundles.
\newblock In {\em Stochastic analysis and related topics in Kyoto}, volume~41
  of {\em Adv. Stud. Pure Math.}, pages 31--47. Math. Soc. Japan, Tokyo, 2004.

\bibitem{Elworthy-LeJan-Li-Tani}
K.~D. Elworthy, Y.~LeJan, and Xue-Mei Li.
\newblock Concerning the geometry of stochastic differential equations and
  stochastic flows.
\newblock In {\em 'New Trends in stochastic Analysis', Proc. Taniguchi
  Symposium, Sept. 1994, Charingworth, ed. K. D. Elworthy and S. Kusuoka, I.
  Shigekawa}. World Scientific Press, 1996.

\bibitem{Elworthy-LeJan-Li-book}
K.~D. Elworthy, Y.~LeJan, and Xue-Mei Li.
\newblock {\em On the geometry of diffusion operators and stochastic flows,
  Lecture Notes in Mathematics 1720}.
\newblock Springer, 1999.

\bibitem{Elworthy-Rosenberg}
K.~D. Elworthy and S.~Rosenberg.
\newblock Homotopy and homology vanishing theorems and the stability of
  stochastic flows.
\newblock {\em Geom. Funct. Anal.}, 6(1):51--78, 1996.

\bibitem{Elworthy-Yor}
K.~D. Elworthy and M.~Yor.
\newblock Conditional expectations for derivatives of certain stochastic flows.
\newblock In J.~Az\'ema, P.A. Meyer, and M.~Yor, editors, {\em Sem. de Prob.
  XXVII. Lecture Notes in Maths. 1557}, pages 159--172. Springer-Verlag, 1993.

\bibitem{Elworthy-abel}
K.~David Elworthy.
\newblock The space of stochastic differential equations.
\newblock In {\em Stochastic analysis and applications}, volume~2 of {\em Abel
  Symp.}, pages 327--337. Springer, Berlin, 2007.

\bibitem{Emery}
M.~{\'E}mery.
\newblock {\em Stochastic calculus in manifolds}.
\newblock Universitext. Springer-Verlag, Berlin, 1989.
\newblock With an appendix by P.-A. Meyer.

\bibitem{Estrade-Pontier-Florchinger}
A.~Estrade, M.~Pontier, and P.~Florchinger.
\newblock Filtrage avec observation discontinue sur une vari\'et\'e.
  {E}xistence d'une densit\'e r\'eguli\`ere.
\newblock {\em Stochastics Stochastics Rep.}, 56(1-2):33--51, 1996.

\bibitem{Falcitelli-Ianus-Pastore}
Maria Falcitelli, Stere Ianus, and Anna~Maria Pastore.
\newblock {\em Riemannian submersions and related topics}.
\newblock World Scientific Publishing Co. Inc., River Edge, NJ, 2004.

\bibitem{Freidlin85}
M.~Freidlin.
\newblock {\em Functional integration and partial differential equations},
  volume 109 of {\em Annals of Mathematics Studies}.
\newblock Princeton University Press, Princeton, NJ, 1985.

\bibitem{Ge-92}
Zhong Ge.
\newblock Betti numbers, characteristic classes and sub-{R}iemannian geometry.
\newblock {\em Illinois J. Math.}, 36(3):372--403, 1992.

\bibitem{Gromov-CC}
Mikhael Gromov.
\newblock Carnot-{C}arath\'eodory spaces seen from within.
\newblock In {\em Sub-Riemannian geometry}, volume 144 of {\em Progr. Math.},
  pages 79--323. Birkh\"auser, Basel, 1996.

\bibitem{Hermann}
Robert Hermann.
\newblock A sufficient condition that a mapping of {R}iemannian manifolds be a
  fibre bundle.
\newblock {\em Proc. Amer. Math. Soc.}, 11:236--242, 1960.

\bibitem{Humphreys}
J.~E. Humphreys.
\newblock {\em Introduction to {L}ie algebras and representation theory},
  volume~9 of {\em Graduate Texts in Mathematics}.
\newblock Springer-Verlag, New York, 1978.
\newblock Second printing, revised.

\bibitem{Ikeda-Watanabe}
N.~Ikeda and S.~Watanabe.
\newblock {\em Stochastic Differential Equations and Diffusion Processes ,
  second edition}.
\newblock North-Holland, 1989.

\bibitem{Kobayashi-Nomizu-I}
S.~Kobayashi and K.~Nomizu.
\newblock {\em Foundations of differential geometry, Vol. I}.
\newblock Interscience Publishers, 1969.

\bibitem{Kobayashi-Nomizu-II}
S.~Kobayashi and K.~Nomizu.
\newblock {\em Foundations of differential geometry, Vol. II}.
\newblock Interscience Publishers, 1969.

\bibitem{Kolar-Michor-Slovak}
I.~Kol{\'a}{\v{r}}, P.~W. Michor, and J.~Slov{\'a}k.
\newblock {\em Natural operations in differential geometry}.
\newblock Springer-Verlag, Berlin, 1993.

\bibitem{Kunita-book}
Hiroshi Kunita.
\newblock {\em Stochastic flows and stochastic differential equations},
  volume~24 of {\em Cambridge Studies in Advanced Mathematics}.
\newblock Cambridge University Press, Cambridge, 1990.

\bibitem{LazaroCami-Ortega-reduction}
Joan-Andreu Lazaro-Cami and Juan-Pablo Ortega.
\newblock Reduction, reconstruction, and skew-product decomposition of
  symmetric stochastic differential equations.
\newblock {\em arXiv:0705.3156v2 [math.PR]}.

\bibitem{Leandre-no-prob}
Remi Leandre.
\newblock Applications of the {M}alliavin calculus of {B}ismut type without
  probability.
\newblock {\em WSEAS Trans. Math.}, 5(11):1205--1210, 2006.

\bibitem{Lemaire}
S.~Lemaire.
\newblock Invariant jets of a smooth dynamical system.
\newblock {\em Bull. Soc. Math. France}, 129(3):379--448, 2001.

\bibitem{flow}
Xue-Mei Li.
\newblock Stochastic differential equations on noncompact manifolds: moment
  stability and its topological consequences.
\newblock {\em Probab. Theory Related Fields}, 100(4):417--428, 1994.

\bibitem{Liao89}
M.~Liao.
\newblock Factorization of diffusions on fibre bundles.
\newblock {\em Trans. Amer. Math. Soc.}, 311(2):813--827, (1989).

\bibitem{Liao00}
Ming Liao.
\newblock Decomposition of stochastic flows and {L}yapunov exponents.
\newblock {\em Probab. Theory Related Fields}, 117(4):589--607, 2000.

\bibitem{McKean-book}
Henry~P. McKean.
\newblock {\em Stochastic integrals}.
\newblock AMS Chelsea Publishing, Providence, RI, 2005.
\newblock Reprint of the 1969 edition, with errata.

\bibitem{Michor}
P.~W. Michor.
\newblock {\em Gauge theory for fiber bundles}, volume~19 of {\em Monographs
  and Textbooks in Physical Science. Lecture Notes}.
\newblock Bibliopolis, Naples, 1991.

\bibitem{Montgomery-book}
Richard Montgomery.
\newblock {\em A tour of subriemannian geometries, their geodesics and
  applications}, volume~91 of {\em Mathematical Surveys and Monographs}.
\newblock American Mathematical Society, Providence, RI, 2002.

\bibitem{Narasimhan-Ramanan61}
M.~S. Narasimhan and S.~Ramanan.
\newblock Existence of universal connections.
\newblock {\em American J. Math.}, 83, 1961.

\bibitem{Oleinik-66}
O.~A. Ole{\u\i}nik.
\newblock On linear equations of the second order with a non-negative
  characteristic form.
\newblock {\em Mat. Sb. (N.S.)}, 69 (111):111--140, 1966.

\bibitem{O'Neill}
Barrett O'Neill.
\newblock The fundamental equations of a submersion.
\newblock {\em Michigan Math. J.}, 13:459--469, 1966.

\bibitem{Pardoux-Nato}
E.~Pardoux.
\newblock Nonlinear filtering, prediction and smoothing.
\newblock In {\em Stochastic systems: the mathematics of filtering and
  identification and applications ({L}es {A}rcs, 1980)}, volume~78 of {\em NATO
  Adv. Study Inst. Ser. C: Math. Phys. Sci.}, pages 529--557. Reidel,
  Dordrecht, 1981.

\bibitem{Pardoux-Stflour}
{\'E}tienne Pardoux.
\newblock Filtrage non lin\'eaire et \'equations aux d\'eriv\'ees partielles
  stochastiques associ\'ees.
\newblock In {\em \'{E}cole d'\'{E}t\'e de {P}robabilit\'es de {S}aint-{F}lour
  {XIX}---1989}, volume 1464 of {\em Lecture Notes in Math.}, pages 67--163.
  Springer, Berlin, 1991.

\bibitem{Pauwels-Rogers}
E.~J. Pauwels and L.~C.~G. Rogers.
\newblock Skew-product decompositions of {B}rownian motions.
\newblock In {\em Geometry of random motion (Ithaca, N.Y., 1987)}, volume~73 of
  {\em Contemp. Math.}, pages 237--262. Amer. Math. Soc., 1988.

\bibitem{Pontier-Szpirglas}
Monique Pontier and Jacques Szpirglas.
\newblock Filtering with observations on a {R}iemannian symmetric space.
\newblock In {\em Stochastic differential systems (Bad Honnef, 1985)},
  volume~78 of {\em Lecture Notes in Control and Inform. Sci.}, pages 316--329.
  Springer, Berlin, 1986.

\bibitem{Quillen}
D.~Quillen.
\newblock Superconnection character forms and the {C}ayley transform.
\newblock {\em Topology}, 27(2):211--238, 1988.

\bibitem{Reed-Simon}
Michael Reed and Barry Simon.
\newblock {\em Methods of modern mathematical physics. {I}}.
\newblock Academic Press Inc. [Harcourt Brace Jovanovich Publishers], New York,
  second edition, 1980.
\newblock Functional analysis.

\bibitem{Revuz-Yor}
Daniel Revuz and Marc Yor.
\newblock {\em Continuous martingales and {B}rownian motion}, volume 293 of
  {\em Grundlehren der Mathematischen Wissenschaften [Fundamental Principles of
  Mathematical Sciences]}.
\newblock Springer-Verlag, Berlin, third edition, 1999.

\bibitem{Rogers-Williams-II}
L.~C.~G. Rogers and D.~Williams.
\newblock {\em Diffusions, {M}arkov processes, and martingales. {V}ol. 2}.
\newblock Cambridge Mathematical Library. Cambridge University Press,
  Cambridge, 2000.

\bibitem{Rosenberg}
S.~Rosenberg.
\newblock {\em The {L}aplacian on a {R}iemannian manifold}, volume~31 of {\em
  London Mathematical Society Student Texts}.
\newblock Cambridge University Press, Cambridge, 1997.
\newblock An introduction to analysis on manifolds.

\bibitem{Ruffino}
Paulo R.~C. Ruffino.
\newblock Decomposition of stochastic flows and rotation matrix.
\newblock {\em Stoch. Dyn.}, 2(1):93--107, 2002.

\bibitem{Stroock-Varadhan-deg}
D.~Stroock and S.~R.~S. Varadhan.
\newblock On degenerate elliptic-parabolic operators of second order and their
  associated diffusions.
\newblock {\em Comm. Pure Appl. Math.}, 25:651--713, 1972.

\bibitem{Stroock87}
D.~W. Stroock.
\newblock {\em Lectures on stochastic analysis: diffusion theory}, volume~6 of
  {\em London Mathematical Society Student Texts}.
\newblock Cambridge University Press, Cambridge, 1987.

\bibitem{Stroock-Varadhan}
D.~W. Stroock and S.~R.~S. Varadhan.
\newblock {\em Multidimensional diffusion processes}, volume 233 of {\em
  Grundlehren der Mathematischen Wissenschaften [Fundamental Principles of
  Mathematical Sciences]}.
\newblock Springer-Verlag, Berlin, 1979.

\bibitem{Sussmann}
H.~J. Sussmann.
\newblock Orbits of families of vector fields and integrability of
  distributions.
\newblock {\em Trans. Amer. Math. Soc.}, 180:171--188, 1973.

\bibitem{Taira}
Kazuaki Taira.
\newblock {\em Diffusion processes and partial differential equations}.
\newblock Academic Press Inc., Boston, MA, 1988.

\bibitem{Tsirelson-01}
B~Tsirelson.
\newblock Filtrations of random processes in the light of classification
  theory. (i). a topological zero-one law.
\newblock {\em arXiv:math.PR/0107121}, 2001.

\end{thebibliography}


\begin{theindex}

  \item $GLM$, 42
  \item $\D'(x)$, 32
  \item $\D^0(x)$, 32
  \item $\Gamma$-martingale, 89, 140

  \indexspace

  \item  cohesive
    \subitem descends, 28
  \item  connection
    \subitem LW, 45

  \indexspace

  \item adjoint
    \subitem semi-connection, 41
  \item along $S$
    \subitem diffusion operator, 14
    \subitem semi-martingale, 138

  \indexspace

  \item connection
    \subitem complete, 32
    \subitem Levi-Civita, 45
    \subitem LW, 41
    \subitem semi-, 21
    \subitem stochastically complete, 102
    \subitem strongly stochastically complete, 102
  \item connections
    \subitem determined by stochastic flows, 117
  \item curvature
    \subitem form, 37
    \subitem operator, 42, 57
    \subitem parallel, 111
    \subitem Ricci, 42, 46, 56
    \subitem Weitzenb\"ock, 48, 52, 54, 57

  \indexspace

  \item derivative
    \subitem flow, 46
  \item development
    \subitem stochastic, 138
  \item diffusion
    \subitem measure, 59
    \subitem operator, 11
  \item distribution, 14
    \subitem regular, 15
  \item Duncan-Mortensen-Zakai equation, 85

  \indexspace

  \item flow
    \subitem derivative, 44, 45, 48, 81
    \subitem gradient, 46
    \subitem holonomy, 33
    \subitem of isometries, 48
    \subitem of SDE, 44
    \subitem stochastic, 119
  \item frame bundle, 42

  \indexspace

  \item Gradient Brownian
    \subitem SDE, 45

  \indexspace

  \item H\"ormander
    \subitem condition, 106, 109, 110
    \subitem form, 4, 23, 24, 26, 41, 50, 55, 97, 98, 104, 108, 109, 
		126, 137
    \subitem representation, 12
  \item Heisenberg group, 27
  \item holonomy group, 37
  \item horizontal, 17

  \indexspace

  \item innovations process, 88

  \indexspace

  \item Kallianpur-Striebel formula, 85

  \indexspace

  \item LW connection, 41

  \indexspace

  \item operator
    \subitem curvature, 42
    \subitem diffusion, 11
    \subitem semi-elliptic, 11
    \subitem symbol, 11

  \indexspace

  \item reproducing kernel, 120, 125

  \indexspace

  \item second fundamental form, 45
  \item semi-elliptic, 11
  \item shape operator, 45
  \item stochastic partial differential equation, 93
  \item symbol, 11
    \subitem projectible, 22

  \indexspace

  \item vertical, 17

\end{theindex}

\end{document}